\documentclass[12pt]{amsart}
\usepackage{mathrsfs}
\usepackage{cases}
\usepackage{epic}
\usepackage{amsfonts}
\usepackage{tikz}
\usepackage{tikz-cd}
\usepackage{graphicx}
\usepackage{amsmath}
\usepackage{amssymb, upgreek}
\usepackage{bm}
\usepackage{latexsym,todonotes}
\usepackage{pdflscape}
\usepackage[all]{xypic}
\usepackage[all]{xy}
\usepackage{tikz-cd}
\usepackage{color}
\usepackage{colordvi}
\usepackage{multicol}
\usepackage{mathtools}
\usepackage[normalem]{ulem}
\usepackage[linktocpage=true]{hyperref}
\hypersetup{colorlinks,linkcolor=blue,urlcolor=cyan,citecolor=blue}
\usepackage{graphicx}
\NewDocumentCommand{\sslash}{s}{%
	\IfBooleanTF{#1}
	{\big/\mkern-7mu\big/}
	{/\mkern-6mu/}%
}
\makeatletter
\newsavebox{\@brx}
\newcommand{\llangle}[1][]{\savebox{\@brx}{\(\m@th{#1\langle}\)}%
	\mathopen{\copy\@brx\kern-0.5\wd\@brx\usebox{\@brx}}}
\newcommand{\rrangle}[1][]{\savebox{\@brx}{\(\m@th{#1\rangle}\)}%
	\mathclose{\copy\@brx\kern-0.5\wd\@brx\usebox{\@brx}}}
\makeatother

\tikzcdset{scale cd/.style={every label/.append style={scale=#1},
		cells={nodes={scale=#1}}}}

\DeclareMathAlphabet{\mathbbb}{U}{bbold}{m}{n}
\DeclareMathOperator{\dimv}{\underline{\dim}}

\makeatletter

\newcommand{\Rmnum}[1]{\textup{\expandafter\@slowromancap\romannumeral #1@}}
\makeatother

\allowdisplaybreaks

\def \tUU{(\tU\otimes\tU)}
\def \tUUi {(\tU\otimes\tU)^\imath}

\begin{document}
	\input xy
	\xyoption{all}
	\newcommand{\iLa}{\Lambda^{\imath}}
	\newcommand{\iadd}{\operatorname{iadd}\nolimits}
	\renewcommand{\mod}{\operatorname{mod}\nolimits}
	\newcommand{\fproj}{\operatorname{f.proj}\nolimits}
	\newcommand{\Fac}{\operatorname{Fac}\nolimits}
	\newcommand{\ci}{{\I}_{\btau}}
	\newcommand{\proj}{\operatorname{proj}\nolimits}
	\newcommand{\inj}{\operatorname{inj}\nolimits}
	\newcommand{\rad}{\operatorname{rad}\nolimits}
	\newcommand{\Span}{\operatorname{Span}\nolimits}
	\newcommand{\soc}{\operatorname{soc}\nolimits}
	\newcommand{\ind}{\operatorname{inj.dim}\nolimits}
	\newcommand{\Ginj}{\operatorname{Ginj}\nolimits}
	\newcommand{\res}{\operatorname{res}\nolimits}
	\newcommand{\np}{\operatorname{np}\nolimits}
	\newcommand{\Mor}{\operatorname{Mor}\nolimits}
	\newcommand{\Mod}{\operatorname{Mod}\nolimits}
	\newcommand{\End}{\operatorname{End}\nolimits}
	\newcommand{\lf}{\operatorname{l.f.}\nolimits}
	\newcommand{\Iso}{\operatorname{Iso}\nolimits}
	\newcommand{\Aut}{\operatorname{Aut}\nolimits}
	\newcommand{\Rep}{\operatorname{Rep}\nolimits}
	
	\newcommand{\colim}{\operatorname{colim}\nolimits}
	\newcommand{\gldim}{\operatorname{gl.dim}\nolimits}
	\newcommand{\cone}{\operatorname{cone}\nolimits}
	\newcommand{\rep}{\operatorname{rep}\nolimits}
	\newcommand{\Ext}{\operatorname{Ext}\nolimits}
	\newcommand{\Tor}{\operatorname{Tor}\nolimits}
	\newcommand{\Hom}{\operatorname{Hom}\nolimits}
	\newcommand{\Top}{\operatorname{top}\nolimits}
	\newcommand{\Coker}{\operatorname{Coker}\nolimits}
	\newcommand{\thick}{\operatorname{thick}\nolimits}
	\newcommand{\rank}{\operatorname{rank}\nolimits}
	\newcommand{\Gproj}{\operatorname{Gproj}\nolimits}
	\newcommand{\Len}{\operatorname{Length}\nolimits}
	\newcommand{\RHom}{\operatorname{RHom}\nolimits}
	\renewcommand{\deg}{\operatorname{deg}\nolimits}
	\renewcommand{\Im}{\operatorname{Im}\nolimits}
	\newcommand{\Ker}{\operatorname{Ker}\nolimits}
	\newcommand{\Coh}{\operatorname{Coh}\nolimits}
	\newcommand{\Id}{\operatorname{Id}\nolimits}
	\newcommand{\Qcoh}{\operatorname{Qch}\nolimits}
	\newcommand{\CM}{\operatorname{CM}\nolimits}
	\newcommand{\sgn}{\operatorname{sgn}\nolimits}
	\newcommand{\utMH}{\operatorname{\cm\ch(\iLa)}\nolimits}
	\newcommand{\GL}{\operatorname{GL}}
	\newcommand{\Perv}{\operatorname{Perv}}
	
	\newcommand{\IC}{\operatorname{IC}}
	\def \hU{\widehat{\U}}
	\def \hUi{\widehat{\U}^\imath}
	\newcommand{\bb}{\psi_*}
	\newcommand{\bvs}{{\boldsymbol{\varsigma}}}
	\def \ba{\mathbf{a}}
	\newcommand{\vs}{\varsigma}
	\def \bfk {\mathbf{k}}

	\def \bd{\mathbf{d}}
	\newcommand{\e}{{\bf 1}}
	\newcommand{\EE}{E^*}
	\newcommand{\dbl}{\operatorname{dbl}\nolimits}
	\newcommand{\ga}{\gamma}
	\newcommand{\tM}{\cm\widetilde{\ch}}
	\newcommand{\la}{\lambda}
	
	\newcommand{\For}{\operatorname{{\bf F}or}\nolimits}
	\newcommand{\coker}{\operatorname{Coker}\nolimits}
	\newcommand{\rankv}{\operatorname{\underline{rank}}\nolimits}
	\newcommand{\diag}{{\operatorname{diag}\nolimits}}
	\newcommand{\swa}{{\operatorname{swap}\nolimits}}
	\newcommand{\supp}{{\operatorname{supp}}}
	
	\renewcommand{\Vec}{{\operatorname{Vec}\nolimits}}
	\newcommand{\pd}{\operatorname{proj.dim}\nolimits}
	\newcommand{\gr}{\operatorname{gr}\nolimits}
	\newcommand{\id}{\operatorname{id}\nolimits}
	\newcommand{\aut}{\operatorname{Aut}\nolimits}
	\newcommand{\Gr}{\operatorname{Gr}\nolimits}
	
	\newcommand{\pdim}{\operatorname{proj.dim}\nolimits}
	\newcommand{\idim}{\operatorname{inj.dim}\nolimits}
	\newcommand{\Gd}{\operatorname{G.dim}\nolimits}
	\newcommand{\Ind}{\operatorname{Ind}\nolimits}
	\newcommand{\add}{\operatorname{add}\nolimits}
	\newcommand{\pr}{\operatorname{pr}\nolimits}
	\newcommand{\oR}{\operatorname{R}\nolimits}
	\newcommand{\oL}{\operatorname{L}\nolimits}
	\def \brW{\mathrm{Br}(W_\btau)}
	\newcommand{\Perf}{{\mathfrak Perf}}
	\newcommand{\cc}{{\mathcal C}}
	\newcommand{\gc}{{\mathcal GC}}
	\newcommand{\ce}{{\mathcal E}}
	\newcommand{\calI}{{\mathcal I}}
	\newcommand{\cs}{{\mathcal S}}
	\newcommand{\cf}{{\mathcal F}}
	\newcommand{\cx}{{\mathcal X}}
	\newcommand{\cy}{{\mathcal Y}}
	\newcommand{\ct}{{\mathcal T}}
	\newcommand{\cu}{{\mathcal U}}
	\newcommand{\cv}{{\mathcal V}}
	\newcommand{\cn}{{\mathcal N}}
	\newcommand{\mcr}{{\mathcal R}}
	\newcommand{\ch}{{\mathcal H}}
	\newcommand{\ca}{{\mathcal A}}
	\newcommand{\cb}{{\mathcal B}}
	\newcommand{\cj}{{\mathcal J}}
	\newcommand{\cl}{{\mathcal L}}
	\newcommand{\cm}{{\mathcal M}}
	\newcommand{\cp}{{\mathcal P}}
	\newcommand{\cg}{{\mathcal G}}
	\newcommand{\cw}{{\mathcal W}}
	\newcommand{\co}{{\mathcal O}}
	\newcommand{\cq}{{\mathcal Q}}
	\newcommand{\cd}{{\mathcal D}}
	\newcommand{\ck}{{\mathcal K}}
	\newcommand{\calr}{{\mathcal R}}
	\newcommand{\cz}{{\mathcal Z}}
	\newcommand{\ol}{\overline}
	\newcommand{\ul}{\underline}
	\newcommand{\st}{[1]}
	\newcommand{\ow}{\widetilde}
	\renewcommand{\P}{\mathbf{P}}
	\newcommand{\pic}{\operatorname{Pic}\nolimits}
	\newcommand{\Spec}{\operatorname{Spec}\nolimits}
	\newcommand{\Fr}{\mathrm{Fr}}
	\newcommand{\Gp}{\mathrm{Gp}}

	\newtheorem{innercustomthm}{{\bf Theorem}}
	\newenvironment{customthm}[1]
	{\renewcommand\theinnercustomthm{#1}\innercustomthm}
	{\endinnercustomthm}
	
	\newtheorem{innercustomcor}{{\bf Corollary}}
	\newenvironment{customcor}[1]
	{\renewcommand\theinnercustomcor{#1}\innercustomcor}
	{\endinnercustomthm}
	
	\newtheorem{innercustomprop}{{\bf Proposition}}
	\newenvironment{customprop}[1]
	{\renewcommand\theinnercustomprop{#1}\innercustomprop}
	{\endinnercustomthm}
	
	\newtheorem{theorem}{Theorem}[section]
	\newtheorem{acknowledgement}[theorem]{Acknowledgement}
	\newtheorem{algorithm}[theorem]{Algorithm}
	\newtheorem{axiom}[theorem]{Axiom}
	\newtheorem{case}[theorem]{Case}
	\newtheorem{claim}[theorem]{Claim}
	\newtheorem{conclusion}[theorem]{Conclusion}
	\newtheorem{condition}[theorem]{Condition}
	\newtheorem{conjecture}[theorem]{Conjecture}
	\newtheorem{construction}[theorem]{Construction}
	\newtheorem{corollary}[theorem]{Corollary}
	\newtheorem{criterion}[theorem]{Criterion}
	\newtheorem{definition}[theorem]{Definition}
	\newtheorem{example}[theorem]{Example}
	\newtheorem{assumption}[theorem]{Assumption}
	\newtheorem{lemma}[theorem]{Lemma}
	\newtheorem{notation}[theorem]{Notation}
	\newtheorem{problem}[theorem]{Problem}
	\newtheorem{proposition}[theorem]{Proposition}
	\newtheorem{solution}[theorem]{Solution}
	\newtheorem{summary}[theorem]{Summary}
	\newtheorem{hypothesis}[theorem]{Hypothesis}
	\newtheorem*{thm}{Theorem}
	
	\theoremstyle{remark}
	\newtheorem{remark}[theorem]{Remark}
	
	\def \Br{\mathrm{Br}}
	\newcommand{\tK}{K}
	
	\newcommand{\tk}{\widetilde{k}}
	\newcommand{\tU}{\widetilde{{\mathbf U}}}
	\newcommand{\Ui}{{\mathbf U}^\imath}
	\newcommand{\tUi}{\widetilde{{\mathbf U}}^\imath}
	\newcommand{\qbinom}[2]{\begin{bmatrix} #1\\#2 \end{bmatrix} }
	\newcommand{\ov}{\overline}
	\newcommand{\tMHg}{\operatorname{\widetilde{\ch}(Q,\btau)}\nolimits}
	\newcommand{\tMHgop}{\operatorname{\widetilde{\ch}(Q^{op},\btau)}\nolimits}
	
	\newcommand{\rMHg}{\operatorname{\ch_{\rm{red}}(Q,\btau)}\nolimits}
	\newcommand{\dg}{\operatorname{dg}\nolimits}
	\def \fu{{\mathfrak{u}}}
	\def \fv{{\mathfrak{v}}}
	\def \sqq{{\mathbbb{v}}}
	\def \bp{{\mathbf p}}
	\def \bv{{\mathbf v}}
	\def \bw{{\mathbf w}}
	\def \bA{{\mathbf A}}
	\def \bL{{\mathbf L}}
	\def \bF{{\mathbf F}}
	\def \bS{{\mathbf S}}
	\def \bC{{\mathbf C}}
	\def \bU{{\mathbf U}}
	\def \U{{\mathbf U}}
	\def \btau{\varpi}
	\def \La{\Lambda}
	\def \Res{\Delta}
	\newcommand{\ev}{\bar{0}}
	\newcommand{\odd}{\bar{1}}
	\def \fk{\mathfrak{k}}
	\def \ff{\mathfrak{f}}
	\def \fp{{\mathfrak{P}}}
	\def \fg{\mathfrak{g}}
	\def \fn{\mathfrak{n}}
	\def \gr{\mathfrak{gr}}
	\def \Z{\mathbb{Z}}
	\def \F{\mathbb{F}}
	\def \D{\mathbb{D}}
	\def \C{\mathbb{C}}
	\def \N{\mathbb{N}}
	\def \Q{\mathbb{Q}}
	\def \G{\mathbb{G}}
	\def \P{\mathbb{P}}
	\def \K{\mathbb{K}}
	\def \E{\mathbb{K}}
	\def \I{\mathbb{I}}
	
	\def \eps{\varepsilon}
	\def \BH{\mathbb{H}}
	\def \btau{\varrho}
	\def \cv{\varpi}
	
	\def \tR{\widetilde{\bf R}}
	\def \tRZ{\widetilde{\bf R}_\cz}
	\def \hR{\widehat{\bf R}}
	\def \hRZ{\widehat{\bf R}_\cz}
	\def\tRi{\widetilde{\bf R}^\imath}
	\def\hRi{\widehat{\bf R}^\imath}
	\def\tRiZ{\widetilde{\bf R}^\imath_\cz}
	\def\reg{\mathrm{reg}}
	
	\def\hRiZ{\widehat{\bf R}^\imath_\cz}
	\def \tTT{\widetilde{\mathbf{T}}}
	\def \TT{\mathbf{T}}
	\def \br{\mathbf{r}}
	\def \bp{{\mathbf p}}
	\def \tS{\texttt{S}}
	\def \bq{{\bm q}}
	\def \bvt{{v}}
	\def \bs{{ r}}
	\def \tt{{v}}
	\def \k{k}
	\def \bnu{\bm{\nu}}
	\def\bc{\mathbf{c}}
	\def \ts{\textup{\texttt{s}}}
	\def \tt{\textup{\texttt{t}}}
	\def \tr{\textup{\texttt{r}}}
	\def \tc{\textup{\texttt{c}}}
	\def \tg{\textup{\texttt{g}}}
	\def \bW{\mathbf{W}}
	\def \bV{\mathbf{V}}

	\newcommand{\browntext}[1]{\textcolor{brown}{#1}}
	\newcommand{\greentext}[1]{\textcolor{green}{#1}}
	\newcommand{\redtext}[1]{\textcolor{red}{#1}}
	\newcommand{\bluetext}[1]{\textcolor{blue}{#1}}
	\newcommand{\brown}[1]{\browntext{ #1}}
	\newcommand{\green}[1]{\greentext{ #1}}
	\newcommand{\red}[1]{\redtext{ #1}}
	\newcommand{\blue}[1]{\bluetext{ #1}}
	\numberwithin{equation}{section}
	\renewcommand{\theequation}{\thesection.\arabic{equation}}
	
	\newcommand{\wtodo}{\rightarrowdo[inline,color=orange!20, caption={}]}
	\newcommand{\lutodo}{\rightarrowdo[inline,color=green!20, caption={}]}
	\def \tT{\widetilde{\mathcal T}}
	
	\def \tTL{\tT(\iLa)}
	\def \iH{\widetilde{\ch}}
	
	
	\title[Dual canonical bases of quantum groups and $\imath$quantum groups II]{Dual canonical bases of quantum groups and $\imath$quantum groups II: geometry}
	
	\author[Ming Lu]{Ming Lu}
	\address{Department of Mathematics, Sichuan University, Chengdu 610064, P.R.China}
	\email{luming@scu.edu.cn}

	\author[Xiaolong Pan]{Xiaolong Pan}
	\address{Department of Mathematics, Sichuan University, Chengdu 610064, P.R.China}
	\email{xiaolong\_pan@stu.scu.edu.cn}

	\subjclass[2020]{Primary 17B37, 18G80.}
	\keywords{dual canonical bases, quantum groups,  $\imath$quantum groups, Hall algebras, quiver varieties}

	\begin{abstract}
		The $\imath$quantum groups admit two realizations: one via the $\imath$Hall algebras and the other via the quantum Grothendieck rings of quiver varieties, as developed by the first author and Wang. Based on these two realizations, we establish the dual canonical bases for $\imath$quantum groups of type ADE in two distinct ways, using perverse sheaves and Hall algebras respectively. In this paper, we prove that these two dual canonical bases coincide, thereby proving their invariance under braid group actions, and that their structural constants are integral and positive. Furthermore, we establish the positivity of the coefficients of the transition matrix from the Hall basis (and PBW basis) to the dual canonical basis.
	\end{abstract}

	\maketitle
	\setcounter{tocdepth}{1}
	
	\tableofcontents
	
	\section{Introduction}

	\subsection{Background}

	The (universal) $\imath$quantum group $\tUi =\langle B_i, \tk_i \mid i\in \I \rangle$ is by definition a subalgebra of the Drinfeld double quantum group $\tU$ associated to a Satake diagram, and $(\tU, \tUi)$ forms a quantum symmetric pair; cf. \cite{Let99, Ko14,LW19}. We view Letzter's $\imath$quantum groups as a vast generalization of Drinfeld--Jimbo quantum groups, and quantum groups can be viewed as $\imath$quantum groups of diagonal type. 

	The $\imath$quiver algebra $\iLa$,  introduced in \cite{LW19}, is produced from a quiver $Q =(Q_0, Q_1)$ together with an involution $\btau$ on $Q$. 
	In \cite{LW19}, the first author and Wang introduced  the $\imath$Hall algebras associated to $\iLa$, denoted by $\widetilde{\ch}(\bfk Q,\varrho)$, to realize quasi-split $\imath$quantum groups $\tUi$ of type ADE; see \cite{LW20} for an extension to Kac-Moody setting. This construction is inspired by the Hall algebra realization of quantum groups $\tU$ by Ringel \cite{Rin90} and Bridgeland \cite{Br13}.

	Lusztig in \cite{Lus90,Lus91,Lus93} used perverse sheaves on the varieties of representations of a quiver $Q$ to realize $\U^+$, and produced the canonical basis of $\U^+$ by simple perverse sheaves; cf. also \cite{Ka91} for the crystal basis. Lusztig \cite{Lus90} also used PBW bases to give an elementary construction of canonical bases of $\U^+$; also see \cite[\S11.6]{DDPW}. 
	The relation between Ringel’s realization and Lusztig’s categorification is given by sheaf-function correspondence; see \cite{Lus98}. 
	Nakajima \cite{Na01,Na04} further developed Lusztig's construction, and introduced (graded) Nakajima quiver varieties. 

	Inspired by Bridgeland's construction \cite{Br13}, via (dual) quantum Grothendieck rings of cyclic quiver varieties, Qin \cite{Qin} provided a geometric construction of $\tU$ of type ADE. A bonus of the geometric approach is the construction of a positive basis on $\tU$, which contains as subsets the (mildly rescaled in $\Q(v^{1/2})$) dual canonical bases for halves of the quantum groups \cite{Lus90}. Qin's work was built on the construction of Hernandez-Leclerc \cite{HL15} (who realized half a quantum group  $\U^+$ using graded Nakajima quiver varieties) as well as the concept of quantum Grothendieck rings introduced by Nakajima \cite{Na04} and Varagnolo--Vasserot \cite{VV}. 
	In order to study dual canonical bases of $\tU$ and $\tUi$, following \cite{HL15,Qin} (see also \cite{BG17,LW21b,Sh22}), the quantum groups (and  also $\imath$quantum group) are considered to be over $\Q(v^{1/2})$.
	
	%

	Motivated by the works \cite{Na01,HL15,LeP13},
	Keller and Scherotzke \cite{KS16,Sch19} formulated the notion of regular/singular Nakajima categories $\mcr, \cs$ from the mesh category of the repetition category $\Z Q$ for an acyclic quiver $Q$. 
	Note that $\mcr, \cs$ were called the generalized Nakajima categories {\em loc. cit.}, and they are called Nakajima-Keller-Scherotzke (NKS) categories in \cite{LW21b}. The NKS varieties are by definition (cf. Definition~\ref{def:NKS}) the representation varieties of $\mcr$ and $\cs$.

	For a Dynkin $\imath$quiver $(Q,\varrho)$, the module category of its $\imath$quiver algebra $\Lambda^\imath$ can be realized as a singular NKS category $\cs^\imath$, and the (dual) quantum Grothendieck rings of perverse sheaves over $\cs^\imath$ can be used to realize quasi-split $\imath$quantum groups $\tUi$ of type ADE; see \cite{LW21b}. Qin's construction \cite{Qin} of the Drinfeld double $\tU$ uses such cyclic quiver varieties, which can be viewed as NKS varieties of $\cs^\imath$ of diagonal type; cf. the $\imath$quiver algebra used to formulate Bridgeland's Hall algebra; cf. \cite{LW19}.

	The geometric construction in \cite{Qin,LW21b} provides a favorable basis for $\tU$ and  $\tUi$  with positive integral structure constants, which are called {\em dual canonical basis}. Unlike Lusztig's canonical basis for $\U^+$, the dual canonical basis is defined for the entire $\tU$ and $\tUi$. 
	To avoid confusion, we sometimes call this basis to be {\em dual IC basis} in this paper. 
	In a prequel \cite{LP25}, we use Hall basis and Lusztig's Lemma to produce the {\em dual canonical basis} for the $\imath$Hall algebra $\widetilde{\ch}(\bfk Q,\varrho)$. This dual canonical basis can also be transferred to $\tU$ and  $\tUi$. As an application of BGP reflection functors and Fourier transforms, we prove that this basis is invariant under braid group actions and is independent of the orientation of the quiver; see \cite{LP25}.
	
	\subsection{Goal}
	
	The goal of this paper is to prove that the dual IC basis provided by perverse sheaves coincides with the dual canonical basis via Hall algebras for the quantum group $\tU$ and the $\imath$quantum group $\tUi$ of type ADE. It is known that there exists a natural isomorphism from the $\imath$Hall algebra to the quantum Grothendieck ring, which is induced by the (dual) trace map.
	These two types of dual canonical bases coincide with each other under this natural isomorphism.
	In this manner, we deduce that the dual canonical bases of $\tU$ and $\tUi$ possess almost the same nice properties as Lusztig's canonical basis, including bar-invariant, integral and positive, as well as invariant under braid group actions.
	Furthermore, the positivity of the Hall basis and PBW basis with respect to the dual canonical basis is also established.

	\subsection{Main results}
	From a Cartan matrix $C=(c_{ij})_{i,j\in\I}$, 
	we \cite{LP25} introduce a variant of Drinfeld double quantum group $\hU$, which is generated by
	$E_i,F_i,K_i,K_i'$ ($i\in\I$). The Drinfeld double $\tU$ is constructed from $\hU$ by making $K_i,K_i'$ ($i\in\I$) invertible. Given an involution $\btau\in\aut(C)$, we define the (universal) $\imath$quantum group $\tUi$ to be the subalgebra of $\tU$ generated by $B_i= F_i +  E_{\btau i} \tK_i',
	\tk_i = \tK_i \tK_{\btau i}'$ ($i \in \I$), and the inverses of $\tk_i$. A variant of $\imath$quantum group $\hUi$ is defined to be subalgebra of $\hU$ similarly, but $\tk_i$ ($i\in\I$) are not invertible. 
	
	
	For a Dynkin $\imath$quiver $(Q,\varrho)$, we can associate a generic $\imath$Hall algebra $\widetilde{\ch}(Q,\varrho)$ defined over $\Q(v^{1/2})$ \cite{LW19}. This algebra $\widetilde{\ch}(Q,\varrho)$ has a basis (called Hall basis) 
	\begin{align}
		\label{eq:iHallbasis}\{\K_\alpha*\fu_\lambda\mid \alpha\in\Z^\I,\lambda\in\fp\},
	\end{align}
	where $\fp$ is the set of $\N$-valued functions over the set $\Phi^+$ of positive roots, and $*$ is the (twisted) Hall product. We also define $\widehat{\ch}(Q,\btau)$ to be the generic $\imath$Hall algebra with the Hall basis $\{\K_\alpha*\fu_\lambda\mid \alpha\in\N^\I,\lambda\in\fp\}$. Let $\widehat{\ct}(Q,\varrho)$ be the subalgebra generated by $\K_\alpha$, $\alpha\in\N^\I$. Inspired by \cite{BG17}, we define an action $\diamond$ of $\widehat{\ct}(Q,\varrho)$ on $\widehat{\ch}(Q,\varrho)$ such that $\ov{\K_\alpha\diamond \fu_\lambda}=\K_\alpha\diamond\ov{\fu_\lambda}$. This construction can be carried to $\widetilde{\ch}(Q,\btau)$. 
	
	For the Dynkin $\imath$quiver $(Q,\btau)$, let $\cs^\imath$ and $\mcr^\imath$ be the singular/regular $\imath$NKS category constructed in \cite{LW21b}; see Definition \ref{def:RS for iQG}. 
	Let $\hRi$ be the quantum Grothendieck ring of $\imath$NKS varieties constructed in \cite{LW21b}. It has a basis 
	dual to the basis of intersection cohomology (IC for short) sheaves 
	associated to the strata 
	in the NKS variety $\cm_0(\bw, \mcr^\imath)=\rep(\bw,\cs^\imath)$, called the dual IC basis. 
	We denote by $\tRi$ a certain  localization of $\hRi$; see \S\ref{subsec: graded Groth ring}. 
	
	Let $\cz:=\Z[v^{1/2},v^{-1/2}]$. The integral form of $\hRi$ is the subspace $\hRiZ$ spanned over $\cz$ by the dual IC basis, and we can similarly define the integral form $\tRiZ$ of $\tRi$. From the construction we know that $\hRiZ$ and $\tRiZ$ are algebras over $\mathcal{Z}$. On the other hand, the integral form of the $\imath$Hall algebra $\widetilde{\ch}(Q,\varrho)_\cz$ is the $\cz$-algebra defined over the (free) $\cz$-module spanned by \eqref{eq:iHallbasis}. Similarly for $\widehat{\ch}(Q,\varrho)_\cz$.
	
	By the results of \cite{LW19} and \cite{LW21b} (see Theorem \ref{thm:iQG-sheaf} and Lemma \ref{lem:Hall-iQG}), we can define isomorphisms
	\[\widehat{\Omega}:\widehat{\ch}(Q,\varrho)\longrightarrow\hRi,\quad  \widetilde{\Omega}:\widetilde{\ch}(Q,\varrho)\longrightarrow\tRi.\]
	It is surprising that these isomorphisms preserve integral forms.
	
	\begin{customthm}{{\bf A}} [Theorem~\ref{prop:Hall-sheaf-integ}]\label{thm A}
		The isomorphisms $\widehat{\Omega}:\widehat{\ch}(Q,\varrho)\rightarrow\hRi$ and $  \widetilde{\Omega}:\widetilde{\ch}(Q,\varrho)\rightarrow\tRi$ preserve their integral forms.
	\end{customthm}

	Using the isomorphism between $\imath$quantum groups and $\imath$Hall algebras (see Theorem \ref{thm:iQG-sheaf}), the bar-involution of $\tUi$ (and $\hUi$) induces the bar-involution of $\widetilde{\ch}(Q,\btau)$ (and $\widehat{\ch}(Q,\btau)$).
	The partial order on $\fp$ given by orbit closures \cite{Lus90} induces a partial order $\prec$ on $\N^\I\times\fp$. With this partial order, we can apply Lusztig's Lemma \cite{BZ14} to a rescaled Hall basis $\{\K_\alpha\diamond\mathfrak{U}_\lambda\mid \alpha\in\N^\I,\lambda\in\fp\}$ of $\widehat{\ch}(Q,\btau)$; see \cite{LP25}. This allows us to construct a bar-invariant integral basis $\{\mathfrak{L}_{\alpha,\lambda}:=\K_\alpha\diamond\mathfrak{L}_\lambda\mid\alpha\in\Z^\I,\lambda\in\fp\}$ on $\widetilde{\ch}(Q,\btau)$, called the dual canonical basis.

	We shall compare the dual canonical basis of $\widetilde{\ch}(Q,\varrho)$ with the  dual IC basis of $\tRi$. For this purpose we use functions on the representation variety $\rep(\bw,\mathcal{S}^\imath)$ of $\mathcal{S}^\imath$ defined over the algebraic closure $\F$ of a finite field $\bfk=\F_q$ to give a realization of the $\imath$Hall algebra $\widetilde{\ch}(\bfk Q,\varrho)$ (recall that we have $\rep(\bw,\mathcal{S}^\imath)\cong\mathcal{M}_0(\bw,\mathcal{R}^\imath)$); see \S\ref{subsec:QV function iHA}. We also construct a generic version $\widetilde{\mathbf{M}}^*(Q,\varrho)$ of $\widetilde{\ch}(\bfk Q,\varrho)$ following \cite[\S 9]{Lus90}. 
	
	Since the transversal slice theorem is not available in positive characteristic, we cannot directly define the analogue of $\tRi$ over $\F$ to apply the sheaf-function correspondence. To overcome this, we instead consider the subvariety $\rep^{\Gp}(\bw,\mathcal{S}^\imath)$ of Gorenstein projective $\mathcal{S}^\imath$-modules, which has the advantage that the induced stratification aligns with the orbit stratification, enabling us to define the quantum Grothendieck rings $\tR^{\imath,\Gp}$ over $\F$ verbatim. We also examine the corresponding generic form $\widetilde{\mathbf{F}}^{\Gp,}(Q,\varrho)$ of the semi-derived Hall algebra, which, by the results in \cite[Appendix A.4]{LW19}, is isomorphic to $\widetilde{\mathbf{M}}^*(Q,\varrho)$. The dual trace map then provides an isomorphism from $\widetilde{\mathbf{F}}^{\Gp,*}(Q,\varrho)$ to $\tR^{\imath,\Gp}$, sending the dual canonical basis to the dual IC basis. Combined with the result of \cite{SS16} (extended to $\imath$quivers), the same can be stated for the isomorphism $\widetilde{\Omega}$ in Theorem~\ref{thm A}.
	
	
	
	\begin{customthm}{{\bf B}} [Corollary~\ref{coro:Upsilon maps dCB}, Proposition~\ref{prop:iHA dual trace map}]\label{thm C}
		The isomorphism $\widetilde{\Omega}$ in Theorem~\ref{thm A} sends the dual canonical basis of $\widetilde{\ch}(Q,\varrho)$ to the dual IC basis of $\tRi$.
	\end{customthm}
	
	
	This in turn allows various positivity properties to be stated for the dual canonical basis of $\widetilde{\ch}(Q,\varrho)$, for example the structure constants lie in $\N[v^{\frac{1}{2}},v^{-\frac{1}{2}}]$. 
	
	
	
	
	
	For a Dynkin $\imath$quiver $(Q,\varrho)$, if $(Q',\varrho)$ is constructed from $(Q,\varrho)$ by reversing some arrows, the Fourier transforms $\Phi_{Q',Q}:\widehat{\ch}(\bfk Q,\btau)\rightarrow \widehat{\ch}(\bfk Q',\btau)$ and $\Phi_{Q',Q}:\widetilde{\ch}(\bfk Q,\btau)\rightarrow \widetilde{\ch}(\bfk Q',\btau)$ are constructed in \cite[Theorem {\bf C}]{LP25}.
	Using the isomorphism in Theorem~\ref{thm A}, they can be transferred to quantum Grothendieck rings $\tRi$ and $'\tRi$ of $(Q,\btau)$ and $(Q',\btau)$ respectively, that is, there is an isomorphism $\Psi=\Psi_{Q',Q}:\tRi\rightarrow {'\tRi}$. 
	
	To see that $\Psi$ preserves the dual IC bases, we consider the natural inclusion $i:E_{Q,\bw}\rightarrow\rep(\bw,\mathcal{S}^\imath)$, where $E_{Q,\bw}$ is the variety of representations of $Q$. The variety $E_{Q,\bw}$ can be identified as the fixed point subvariety of a $\G_m$-action, so the pullback functor $i^*$ sends $\mathcal{L}(\bv,\bw)$ into Lusztig's category $\mathcal{Q}_\bw$, as defined in \cite[\S 9]{Lus93}. This observation allows us to define an embedding 
	\[\iota:\hR^+\longrightarrow\tRi\]
	where $\hR^+$ is the quantum Grothendieck ring for the quiver variety of $Q$ (constructed in \cite{HL15}), with $B(\bv,\bw)$ being the basis dual to the IC sheaves $\mathcal{P}(\bv,\bw)$ associated to orbits in $E_{Q,\bw}$. Using Lusztig's Lemma, we can then construct $L(\bv,\bw)$ from $\iota(B(\bv,\bw))$, and this implies the invariance of dual canonical bases under $\Psi$. 
	
	\begin{customthm}{{\bf C}}[Theorem~\ref{FT of CB}] \label{thm E}
		For a different orientation $(Q',\varrho)$ of the $\imath$quiver $(Q,\varrho)$, the Fourier transform $\Psi:\tRi\rightarrow {'\tRi}$ preserves the dual IC basis. The Fourier transform $\Phi_{Q',Q}:\widetilde{\ch}(Q,\varrho)\rightarrow \widetilde{\ch}(Q',\varrho)$ therefore preserves the dual canonical bases.
	\end{customthm}
	
	
	Lusztig in \cite{Lus90} proved the purity of IC sheaves associated to orbits in $E_{Q,\bw}$. This translates into the positivity of the (rescaled) Hall basis with respect to dual canonical basis on $\tR^+$. By using the embedding $\iota$ one can also establish the positivity of the rescaled Hall basis with respect to dual canonical basis. The positivity of PBW basis also follows, see Proposition \ref{prop:iQG PBW to dCB positivity}.
	
	\begin{customthm}{{\bf D}} [Theorem~\ref{iHA dCB positivity}] \label{thm F}
		Transition matrix coefficients of the rescaled Hall basis of $\widetilde{\ch}(Q,\varrho)$ with respect to dual canonical basis belong to $\N[v^{-1}]$.
	\end{customthm}
	
	The dual canonical basis of $\tUi$ is defined to be the image of $L(\bv,\bw)$ under the isomorphism given in Theorem~\ref{thm:iQG-sheaf}. By Theorem~\ref{thm E}, this basis is independent of the orientations of the $\imath$quivers, convincing the uniqueness of the dual canonical bases. 
	
	\subsection{Perspectives}
	
	
	The dual canonical basis of $\tU$ (and also $\tUi$) of type ADE considered in this paper can be viewed as an ideal extension of Lusztig’s dual canonical basis for $\U^+$.
	We expect this basis (integral, bar-invariant, invariant under the braid group actions) also exists for other types, and moreover it is positive for symmetric  types.  
	As $\imath$quantum groups have many more types than quantum groups, even for the quasi-split type, 
	it is a challenging and important open question to find an algebraic characterization and algorithm for the  dual canonical bases for $\tUi$ of all types.
	
	Most of the quasi-split $\imath$quantum groups of Kac-Moody type can be realized via $\imath$Hall algebras \cite{LW20}. We shall construct dual canonical bases by using the Hall bases of $\imath$Hall algebras in this general setting, especially for affine $\imath$quantum groups. It is interesting and important to extend the geometric realization given in \cite{LW21b} to Kac-Moody type, and the perverse sheaves could be the dual canonical basis of $\tUi$; compare with Lusztig's geometric realization of quantum groups of Kac-Moody type \cite{Lus90a,Lus90,Lus93}. 
	
	There has been a completely different geometric construction of the modified quantum groups $\dot\U$ (of type A) and their canonical bases with positivity \cite{BLM90, Lus93}, which is compatible with the Khovanov-Lauda-Rouquier (KLR)  categorification \cite{KL10, R08}. There has also been a geometric construction of the modified $\imath$quantum groups $\dot\U^\imath$ (of type AIII) and their $\imath$canonical bases with positivity \cite{BKLW, LiW18, BW18b}, which is again compatible with a KLR type categorification \cite{BSWW}. It is natural to ask if there is any connection between these constructions for modified quantum groups and canonical bases (respectively, for $\imath$quantum groups and $\imath$canonical bases) and the geometric constructions of the Drinfeld doubles and dual canonical bases in \cite{Qin, SS16,BG17} (respectively, \cite{LW21b} and this paper).

	In another direction, this paper hopefully brings quantum symmetric pairs and their (dual) canonical bases a step closer to quantum cluster algebras (see \cite{S24} for a cluster realization of $\imath$quantum groups of type AI), as there have been various connections of quantum groups to cluster algebras in the works \cite{GLS13,GY21,HL15, KKKO18, Qin14,Qin} and references therein. In fact, using the cluster algebra method, Shen \cite{Sh22} constructed a basis of the quantum group $\U$ which is integral, positive, and also invariant under braid group actions. It is interesting to compare this basis with the dual (double) canonical basis as they share these nice properties. For quantum $\mathfrak{sl}_2$, they coincide with each other (see \cite{Sh22}), but it is mysterious for higher rank.

	\subsection{Organization}
	
	The paper is organized as follows. In \S\ref{sec:QG and iQG} we review the basics of quantum groups and $\imath$quantum groups associated to a Dynkin ($\imath$)quiver. Here we will use different presentations following \cite{BG17}, and the braid group actions are modified accordingly. This presentation will simplify the realization by means of $\imath$Hall algebras and quantum Grothendieck rings, which are reviewed in \S\ref{sec:iHA} and \S\ref{NKS QV section}, respectively. We also extend the result of \cite{SS16} to $\imath$quivers.
	
	
	We compare the dual canonical basis of $\imath$Hall algebra with the dual IC basis of quantum Grothendieck ring in \S\ref{sec:dCB via sheaf}. We first construct the $\imath$Hall algebra via functions on the representation variety of $\Lambda^\imath$, and then construct a generic version as in \cite[\S 9]{Lus90}. By shifting to the subvariety of Gorenstein projective modules and make use of the aforementioned result of \cite{SS16} proved in \S\ref{NKS QV section}, we prove Theorem~\ref{thm C}.
	
	Using the function realization given in \S\ref{sec:dCB via sheaf}, we define in \S\ref{sec:FT of iHA} the Fourier transforms of quantum Grothendieck rings, and proved the invariance of dual canonical basis under these isomorphisms. We also deduce the positivity of transition matrix coefficients of the rescaled Hall basis of $\imath$Hall algebra with respect to the dual canonical basis. 
	
	
	In \S\ref{sec:dCB irank I} we compute the dual canonical bases of $\tUi_v(\mathfrak{sl}_2)$, and establish a closed formula for $L(\bv,\bw)$ in terms of generators in this case.
	
	\vspace{2mm}
	
	\noindent{\bf Acknowledgments.}
	We thank Weiqiang Wang for his collaboration in related projects, also for his helpful comments and stimulating discussions. ML is partially supported by the National Natural Science Foundation of China (No. 12171333, 12261131498). 
	
	\section{Quantum groups and $\imath$quantum groups}\label{sec:QG and iQG}
	
	In this section, we review the basic construction of quantum groups and $\imath$quantum groups, as in \cite{LP25}.
	
	\subsection{Quantum groups}
	\label{subsec:QG}
	
	Let $\I=\{1,\dots,n\}$ be the index set. 
	Let $C=(c_{ij})_{i,j \in \I}$ be the Cartan matrix of of a simply-laced semi-simple Lie algebra $\fg$.
	Let $\Delta^+=\{\alpha_i\mid i\in\I\}$ be the set of simple roots of $\fg$, and denote the root lattice by $\Z^{\I}:=\Z\alpha_1\oplus\cdots\oplus\Z\alpha_n$. Let $\Phi^+$ be the set of positive roots. 
	

	Let $v$ be an indeterminate. 
	Denote, for $r,m \in \N$,
	\[
	[r]=\frac{v^r-v^{-r}}{v-v^{-1}},
	\quad
	[r]!=\prod_{i=1}^r [i], \quad \qbinom{m}{r} =\frac{[m][m-1]\ldots [m-r+1]}{[r]!}.
	\]
	Following \cite{BG17}, the Drinfeld double $\hU := \hU_v(\fg)$ is defined to be the $\Q(v^{1/2})$-algebra generated by $E_i,F_i, \tK_i,\tK_i'$, $i\in \I$, 
	subject to the following relations:  for $i, j \in \I$,
	\begin{align}
		[E_i,F_j]= \delta_{ij}(v^{-1}-v) (\tK_i-\tK_i'),  &\qquad [\tK_i,\tK_j]=[\tK_i,\tK_j']  =[\tK_i',\tK_j']=0,
		\label{eq:KK}
		\\
		\tK_i E_j=v^{c_{ij}} E_j \tK_i, & \qquad \tK_i F_j=v^{-c_{ij}} F_j \tK_i,
		\label{eq:EK}
		\\
		\tK_i' E_j=v^{-c_{ij}} E_j \tK_i', & \qquad \tK_i' F_j=v^{c_{ij}} F_j \tK_i',
		\label{eq:K2}
	\end{align}
	and the Serre relations. 
	
	We define $\tU=\tU_v(\fg)$ to be the $\Q(v^{1/2})$-algebra constructed from $\widehat{\U}$ by making $\tK_i,\tK_i'$ ($i\in\I$) invertible. Then $\tU$ and $\hU$ are $\Z^\I$-graded by setting $\deg E_i=\alpha_i$, $\deg F_i=-\alpha_i$, $\deg K_i=0=\deg K_i'$. 

	The quantum group $\bU$ is defined to be quotient algebra of $\hU$ (also $\tU$) modulo the ideal generated by $K_iK_i'-1$ ($i\in\I$). In fact, $\U$ 
	is 
	the $\Q(v^{1/2})$-algebra generated by $E_i,F_i, K_i, K_i^{-1}$, $i\in \I$, subject to the  relations modified from \eqref{eq:KK}--\eqref{eq:K2} with $\tK_i'$ replaced by $K_i^{-1}$. 
	
	Let $\hU^+$ be the subalgebra of $\hU$ generated by $E_i$ $(i\in \I)$, $\hU^0$ be the subalgebra of $\widehat{\bU}$ generated by $\tK_i, \tK_i'$ $(i\in \I)$, and $\hU^-$ be the subalgebra of $\widehat{\bU}$ generated by $F_i$ $(i\in \I)$, respectively.
	The subalgebras $\tU^+$, $\tU^0$ and $\tU^-$ of $\tU$, and subalgebras $\bU^+$, $\bU^0$ and $\bU^-$ of $\bU$ are defined similarly. Then the algebras $\hU$, $\widetilde{\bU}$ and $\bU$ have triangular decompositions:
	\begin{align*}
		\hU=\hU^+\otimes\hU^0\otimes \hU^-,\qquad 
		\widetilde{\bU} =\widetilde{\bU}^+\otimes \widetilde{\bU}^0\otimes\widetilde{\bU}^-,
		\qquad
		\bU &=\bU^+\otimes \bU^0\otimes\bU^-.
	\end{align*}
	Clearly, ${\bU}^+\cong\widetilde{\bU}^+\cong\hU^+$, ${\bU}^-\cong \widetilde{\bU}^-\cong\hU^-$, $\hU^0\cong\Q(v^{1/2})[\tK_i,\tK_i'\mid i\in\I]$, $\tU^0\cong \Q(v^{1/2})[\tK_i^{\pm1},(\tK_i')^{\pm1}\mid i\in\I]$ and $\bU^0\cong\Q(v^{1/2})[K_i^{\pm1}\mid i\in\I]$. Note that 
	${\bU}^0 \cong \hU^0/(\tK_i \tK_i' -1 \mid   i\in \I)$. 
	For any $\mu=\sum_{i\in\I}m_i\alpha_i\in\Z^\I$, we denote $K_\mu=\prod_{i\in\I} K_i^{m_i}$, $K_\mu'=\prod_{i\in\I} (K_i')^{m_i}$ in $\tU$ (or $\U$); we can view $\tK_\mu,\tK_\mu'$ in $\hU$ if $\mu\in\N^\I$.
	
	\begin{lemma}[cf. \cite{BG17}]\label{QG bar-involution def}
		There exists an anti-involution $u\mapsto \ov{u}$ on $\hU$ (also $\tU$, $\U$) given by $\ov{v^{1/2}}=v^{-1/2}$, $\ov{E_i}=E_i$, $\ov{F_i}=F_i$, and $\ov{K_i}=K_i$, $\ov{K_i'}=K_i'$, for $i\in\I$.
	\end{lemma}

	\subsection{The $\imath$quantum groups}
	
	For a  Cartan matrix $C=(c_{ij})$, let $\Aut(C)$ be the group of all permutations $\btau$ of the set $\I$ such that $c_{ij}=c_{\btau i,\btau j}$. An element $\btau\in\Aut(C)$ is called an \emph{involution} if $\btau^2=\Id$.

	For a  Cartan matrix $C=(c_{ij})$, 
	let $\btau$ be an involution in $\Aut(C)$. We define the universal $\imath$quantum groups ${\hU}^\imath:=\hU'_v(\fk)$ (resp. $\tUi:=\tU'_v(\fk)$) to be the $\Q(v^{1/2})$-subalgebra of $\hU$ (resp. $\tU$) generated by
	\begin{equation}
		\label{eq:Bi}
		B_i= F_i +  E_{\btau i} \tK_i',
		\qquad \tk_i = \tK_i \tK_{\btau i}', \quad \forall i \in \I,
	\end{equation}
	(with $\tk_i$ invertible in $\tUi$).
	Let $\hU^{\imath 0}$ be the $\Q(v^{1/2})$-subalgebra of $\hUi$ generated by $\tk_i$, for $i\in \I$. Similarly, let $\tU^{\imath 0}$ be the $\Q(v^{1/2})$-subalgebra of $\tUi$ generated by $\tk_i^{\pm1}$, for $i\in \I$. 

	It is known \cite{Let99, Ko14,LW19} that the algebra $\widetilde{\bU}^\imath$ (resp. $\hUi$) is a right coideal subalgebra of $\widetilde{\bU}$ (resp. $\hU$); we call $(\widetilde{\bU}, \widetilde{\bU}^\imath)$ and $(\hU,\hUi)$ quantum symmetric pairs.
	
	We shall refer to $\hUi$ and $\tUi$ 
	as the universal {\em (quasi-split) $\imath${}quantum groups} (cf. the $\imath$quantum groups defined in  \cite{Let99,Ko14}); they are called {\em split} if $\btau =\Id$.

	\begin{example}
		\label{ex:QGvsiQG}
		Let us explain quantum groups are $\imath$quantum group of diagonal type. 
		Consider the $\Q(v)$-subalgebra $\tUUi$ of $\tUU$
		generated by
		\[
		\ck_i:=\tK_{i} \tK_{i^{\diamond}}', \quad
		\ck_i':=\tK_{i^{\diamond}} \tK_{i}',  \quad
		\cb_{i}:= F_{i}+ E_{i^{\diamond}} \tK_{i}', \quad
		\cb_{i^{\diamond}}:=F_{i^{\diamond}}+ E_{i} \tK_{i^{\diamond}}',
		\qquad \forall i\in \I.
		\]
		Here we drop the tensor product notation and use instead $i^\diamond$ to index the generators of the second copy of $\tU$ in $\tUU$. There exists a $\Q(v)$-algebra isomorphism $\widetilde{\phi}: \tU \rightarrow \tUUi$ such that
		\[
		\widetilde{\phi}(E_i)= \cb_{i},\quad \widetilde{\phi}(F_i)= \cb_{i^{\diamond}}, \quad \widetilde{\phi}(\tK_i)= \ck_i', \quad \widetilde{\phi}(\tK_i')= \ck_i, \qquad \forall  i\in \I.
		\]
	\end{example}

	For any $i\in\I$, we set
	\begin{align}
		\label{eq:bbKi}
		\K_i:=v\tk_i, \text{ if }\varrho i=i;
		\qquad
		\K_j:=\tk_j, \text{ otherwise.}
	\end{align}
	and $\K_\alpha:=\prod_{i\in\I}\K_i^{a_i}$ if $\alpha\in\Z^\I$.
	\begin{lemma}
		\label{iQG bar-involution def}
		There exists an anti-involution $u\mapsto \ov{u}$ on $\hUi$ (also $\tUi$) given by $\ov{v^{1/2}}=v^{-1/2}$, $\ov{B_i}=B_i$, $\ov{\K_i}=\K_i$, for $i\in\I$. In particular, $\ov{\tk_i}=\tk_i$ if $\varrho i\neq i$; $\ov{\tk_i}=v^2\tk_i$ if $\varrho i=i$.
	\end{lemma}

	\section{$\imath$Quiver algebras and $\imath$Hall algebras}\label{sec:iHA}
	
	Throughout this paper, we denote $\bfk=\F_q$ the finite field of $q$ elements. 
	
	\subsection{Notations for quivers and representations}
	\label{subsec:Quiver-represent}
	
	Let $Q=(Q_0,Q_1,\ts,\tt)$ be a quiver, where, $Q_0$ is the set of vertices (we sometimes write $\I =Q_0$), $Q_1$ is the set of arrows, and $\ts,\tt:Q_1\rightarrow Q_0$ are source and target maps, respectively. 
	
	Let $\rep_\bfk(Q)$ be the category of finite dimensional representations of $Q$ over $\bfk$, and we identify $\rep_\bfk(Q)$ with the category $\mod(\bfk Q)$ of finite dimensional left modules over the path algebra $\bfk Q$. Similarly, for a (finite-dimensional) quotient algebra $A=\bfk Q/I$, let $\rep(A)$ be the category of finite dimensional representations of $Q$ satisfying the relations in $I$, and identify $\rep(A)$ with $\mod(A)$ of finite dimensional left modules over $A$. Let $\proj(A)$ the subcategory of projective $A$-modules. Denote by ${\rm proj.dim}_AM$ (resp. ${\rm inj.dim}_AM$) the  projective (resp. injective) dimension of an $A$-module $M$.
	

	For a representation $V=(V_i,V_h)_{i\in\I,h\in Q_1}$, the vector $\bd=\dimv V:=(\dim_\bfk V_i)_i\in\N^\I$ is called the dimension vector of $V$.  For a dimension vector $\bd$, let $\rep(\bd,A)$ be the variety of $A$-modules of dimension vector $\bd$. Let $G_\bd:=\prod_{i\in\I}\mathrm{GL}(d_i,\bfk)$. The group $G_\bd$ acts on $\rep(\bd,A)$ by conjugation: 
	\[
	(g_i)_{i\in\I}\cdot (V_i,V_h)=(V_i,g_{\tt(h)}V_\alpha g_{\ts(h)}^{-1}),\quad \forall (V_i,V_h)\in\rep(\bd,A). 
	\]
	Let $\mathfrak{O}_V$ be the $G_\bd$-orbit of $V$.  
	
	Let $S_i$ be the ($1$-dimensional) simple $A$-modules to each $i\in\I$. The Grothendieck group $K_0(\mod(A))$ of $\mod(A)$ is a free abelian group generated by $\widehat{S_i}$ $(i\in\I)$. Denote by $\widehat{M}$ the class of $M$ in $K_0(\mod(A))$. 
	We identify $K_0(\mod(A))$ (especially $K_0(\mod(\bfk Q))$) with the root lattice 
	$\Z^\I=\oplus_{i\in\I} \Z\alpha_i$ (see \S\ref{subsec:QG}) by identifying $\widehat{S_i}$ with the simple root $\alpha_i$, and $\widehat{M}$ with its dimension vector $\dimv M$ for any $M\in\mod(A)$. 

	Let $\langle-,-\rangle_Q$ be the Euler form of $Q$. 
	The Euler form descends to the Grothendieck group $K_0(\mod(\bfk Q))$. We denote by $(-,-)_Q:K_0(\mod(\bfk Q))\times K_0(\mod(\bfk Q))\rightarrow\Z$ the symmetrized Euler form defined by $(\alpha,\beta)_Q=\langle\alpha,\beta\rangle_Q+\langle \beta,\alpha\rangle_Q$.
	
	Let $\cd^b(A)$  be the bounded derived category with suspension functor $\Sigma$. The singularity category  $\cd_{sg}(A)$ of $\mod(A)$ is the Verdier quotient of $\cd^b(A)$ modulo the thick subcategory generated by modules with finite projective dimensions. 
	Denote by $\Gproj(A)$ the (exact) subcategory of Gorenstein projective $A$-modules. 

	\subsection{The $\imath$quivers and doubles}
	\label{subsec:i-quiver}

	An {\em involution} of $Q$ is defined to be an automorphism $\btau$ of the quiver $Q$ such that $\btau^2=\Id$. 
	An involution $\btau$ of $Q$ induces an involution of the path algebra $\bfk Q$, again denoted by $\btau$.
	A quiver together with an involution $\btau$, $(Q, \btau)$, will be called an {\em $\imath$quiver}. In this paper, we focus on Dynkin quivers.
	Associated to the $\imath$quiver $(Q,\varrho)$, the $\imath$quiver algebra $\iLa$ is introduced in \cite{LW19}, which can be described in terms of a certain quiver $\ov{Q}$ and its admissible ideal $\ov{I}$ such that $\iLa \cong \bfk \ov{Q} / \ov{I}$; see \cite[Proposition 2.6]{LW19}.
	We recall $\ov{Q}$ and $\ov{I}$ as follows:
	\begin{itemize}
		\item[(i)] $\ov{Q}$ is constructed from $Q$ by adding a loop $\varepsilon_i$ at the vertex $i\in Q_0$ if $\btau i=i$, and adding an arrow $\varepsilon_i: i\rightarrow \btau i$ for each $i\in Q_0$ if $\btau i\neq i$;
		\item[(ii)] $\ov{I}$ is generated by
		\begin{itemize}
			\item[(1)] (Nilpotent relations) $\varepsilon_{i}\varepsilon_{\btau i}$ for any $i\in\I$;
			\item[(2)] (Commutative relations) $\varepsilon_i\alpha-\btau(\alpha)\varepsilon_j$ for any arrow $\alpha:j\rightarrow i$ in $Q_1$.
		\end{itemize}
	\end{itemize}
	
	We know that $\iLa$ is a 1-Gorenstein algebra.
	
	The following quivers are examples of the quivers $\ov{Q}$ used to describe the $\imath$quiver algebras $\Lambda^\imath$
	associated to non-split $\imath$quivers of type ADE; cf. \cite{LW19}.

	\begin{center}\setlength{\unitlength}{0.7mm}
		\vspace{-1.5cm}
		\begin{equation}
			\label{diag: A}
			\begin{picture}(100,40)(0,20)
				\put(0,10){$\circ$}
				\put(0,30){$\circ$}
				\put(50,10){$\circ$}
				\put(50,30){$\circ$}
				\put(72,10){$\circ$}
				\put(72,30){$\circ$}
				\put(92,20){$\circ$}
				\put(0,6){$r$}
				\put(-2,34){${-r}$}
				\put(50,6){\small $2$}
				\put(48,34){\small ${-2}$}
				\put(72,6){\small $1$}
				\put(70,34){\small ${-1}$}
				\put(92,16){\small $0$}
				
				\put(3,11.5){\vector(1,0){16}}
				\put(3,31.5){\vector(1,0){16}}
				\put(23,10){$\cdots$}
				\put(23,30){$\cdots$}
				\put(33.5,11.5){\vector(1,0){16}}
				\put(33.5,31.5){\vector(1,0){16}}
				\put(53,11.5){\vector(1,0){18.5}}
				\put(53,31.5){\vector(1,0){18.5}}
				\put(75,12){\vector(2,1){17}}
				\put(75,31){\vector(2,-1){17}}
				\color{purple}
				\put(0,13){\vector(0,1){17}}
				\put(2,29.5){\vector(0,-1){17}}
				\put(50,13){\vector(0,1){17}}
				\put(52,29.5){\vector(0,-1){17}}
				\put(72,13){\vector(0,1){17}}
				\put(74,29.5){\vector(0,-1){17}}
				
				\put(-5,20){$\varepsilon_r$}
				\put(3,20){$\varepsilon_{-r}$}
				\put(45,20){\small $\varepsilon_2$}
				\put(53,20){\small $\varepsilon_{-2}$}
				\put(67,20){\small $\varepsilon_1$}
				\put(75,20){\small $\varepsilon_{-1}$}
				\put(92,30){\small $\varepsilon_0$}
				
				\qbezier(93,23)(90.5,25)(92,27)
				\qbezier(92,27)(94,30)(97,27)
				\qbezier(97,27)(98,24)(95.5,22.6)
				\put(95.6,23){\vector(-1,-1){0.3}}
			\end{picture}
		\end{equation}
		\vspace{-0.6cm}
	\end{center}

	\begin{center}\setlength{\unitlength}{0.8mm}
		\begin{equation}
			\label{diag: D}
			\begin{picture}(100,25)(-5,0)
				\put(0,-1){$\circ$}
				\put(0,-5){\small$1$}
				\put(20,-1){$\circ$}
				\put(20,-5){\small$2$}
				\put(64,-1){$\circ$}
				\put(84,-10){$\circ$}
				\put(80,-13){\small${n-1}$}
				\put(84,9.5){$\circ$}
				\put(84,12.5){\small${n}$}

				\put(19.5,0){\vector(-1,0){16.8}}
				\put(38,0){\vector(-1,0){15.5}}
				\put(64,0){\vector(-1,0){15}}
				
				\put(40,-1){$\cdots$}
				\put(83.5,9.5){\vector(-2,-1){16}}
				\put(83.5,-8.5){\vector(-2,1){16}}
				\color{purple}
				\put(86,-7){\vector(0,1){16.5}}
				\put(84,9){\vector(0,-1){16.5}}
				
				\qbezier(63,1)(60.5,3)(62,5.5)
				\qbezier(62,5.5)(64.5,9)(67.5,5.5)
				\qbezier(67.5,5.5)(68.5,3)(66.4,1)
				\put(66.5,1.4){\vector(-1,-1){0.3}}
				\qbezier(-1,1)(-3,3)(-2,5.5)
				\qbezier(-2,5.5)(1,9)(4,5.5)
				\qbezier(4,5.5)(5,3)(3,1)
				\put(3.1,1.4){\vector(-1,-1){0.3}}
				\qbezier(19,1)(17,3)(18,5.5)
				\qbezier(18,5.5)(21,9)(24,5.5)
				\qbezier(24,5.5)(25,3)(23,1)
				\put(23.1,1.4){\vector(-1,-1){0.3}}
				
				\put(-1,9.5){$\varepsilon_1$}
				\put(19,9.5){$\varepsilon_2$}
				\put(59,9.5){$\varepsilon_{n-2}$}
				\put(79,-1){$\varepsilon_{n}$}
				\put(87,-1){$\varepsilon_{n-1}$}
			\end{picture}
		\end{equation}
		\vspace{.8cm}
	\end{center}

	\begin{center}\setlength{\unitlength}{0.8mm}
		\vspace{-3cm}
		\begin{equation}
			\label{diag: E}
			\begin{picture}(100,40)(0,20)
				\put(10,6){\small${6}$}
				\put(10,10){$\circ$}
				\put(32,6){\small${5}$}
				\put(32,10){$\circ$}
				
				\put(10,30){$\circ$}
				\put(10,33){\small${1}$}
				\put(32,30){$\circ$}
				\put(32,33){\small${2}$}
				
				\put(31.5,11){\vector(-1,0){19}}
				\put(31.5,31){\vector(-1,0){19}}
				
				\put(52,22){\vector(-2,1){17.5}}
				\put(52,20){\vector(-2,-1){17.5}}
				
				\put(54.7,21.2){\vector(1,0){19}}
				
				\put(52,20){$\circ$}
				\put(52,16.5){\small$3$}
				\put(74,20){$\circ$}
				\put(74,16.5){\small$4$}
				\color{purple}
				\put(10,12.5){\vector(0,1){17}}
				\put(12,29.5){\vector(0,-1){17}}
				\put(32,12.5){\vector(0,1){17}}
				\put(34,29.5){\vector(0,-1){17}}
				
				\qbezier(52,22.5)(50,24)(51,26.5)
				\qbezier(51,26.5)(53,29)(56,26.5)
				\qbezier(56,26.5)(57.5,24)(55,22)
				\put(55.1,22.4){\vector(-1,-1){0.3}}
				\qbezier(74,22.5)(72,24)(73,26.5)
				\qbezier(73,26.5)(75,29)(78,26.5)
				\qbezier(78,26.5)(79,24)(77,22)
				\put(77.1,22.4){\vector(-1,-1){0.3}}
				
				\put(35,20){$\varepsilon_2$}
				\put(27,20){$\varepsilon_5$}
				\put(13,20){$\varepsilon_1$}
				\put(5,20){$\varepsilon_6$}
				\put(52,30){$\varepsilon_3$}
				\put(73,30){$\varepsilon_4$}
			\end{picture}
		\end{equation}
		\vspace{1cm}
	\end{center}
	
	\begin{example}[$\imath$quivers of diagonal type]
		\label{ex:diagquiver}
		Let $Q$ be an arbitrary quiver, and $Q^{\dbl} =Q\sqcup  Q^{\diamond}$,  where $Q^{\diamond}$ is an identical copy of $Q$  with a vertex set $\{i^{\diamond} \mid i\in Q_0\}$ and an arrow set $\{ \alpha^{\diamond} \mid \alpha \in Q_1\}$. We let $\rm{swap}$ be the involution of $Q^{\rm dbl}$ uniquely determined by $\swa(i)=i^\diamond$ for any $i\in Q_0$. 
		Then $(Q^{\rm dbl},\mathrm{swap})$ is an $\imath$quiver, and its $\imath$quiver algebra is denoted by $\Lambda$; see \cite[Example 2.10]{LW19}. 
	\end{example}

	
	By \cite[Corollary 2.12]{LW19}, $\bfk Q$ is naturally a subalgebra and also a quotient algebra of $\Lambda^\imath$.
	Viewing $\bfk Q$ as a subalgebra of $\Lambda^{\imath}$, we have a restriction functor
	\[
	\res: \mod (\Lambda^{\imath})\longrightarrow \mod (\bfk Q).
	\]
	Viewing $\bfk Q$ as a quotient algebra of $\Lambda^{\imath}$, we obtain a pullback functor
	\begin{equation}\label{eqn:rigt adjoint}
		\iota:\mod(\bfk Q)\longrightarrow\mod(\Lambda^{\imath}).
	\end{equation}
	
	For each $i\in Q_0$, define a $\bfk$-algebra (which can be viewed as a subalgebra of $\iLa$)
	\begin{align}\label{dfn:Hi}
		\BH _i:=\left\{ \begin{array}{cc}  \bfk[\varepsilon_i]/(\varepsilon_i^2) & \text{ if }\btau i=i,
			\\
			\bfk(\xymatrix{i \ar@<0.5ex>[r]^{\varepsilon_i} & \btau i \ar@<0.5ex>[l]^{\varepsilon_{\btau i}}})/( \varepsilon_i\varepsilon_{\btau i},\varepsilon_{\btau i}\varepsilon_i)  &\text{ if } \btau i \neq i .\end{array}\right.
	\end{align}
	Note that $\BH _i=\BH _{\btau i}$ for any $i\in Q_0$. 
	
	Choose one representative for each $\btau$-orbit on $\I=Q_0$, and let
	\begin{align}   \label{eq:ci}
		\ci = \{ \text{the chosen representatives of $\btau$-orbits in $\I$} \}.
	\end{align}
	Define the following subalgebra of $\Lambda^{\imath}$:
	\begin{equation}  \label{eq:H}
		\BH =\bigoplus_{i\in \ci }\BH _i.
	\end{equation}
	Denote by
	\begin{align}
		\res_\BH :\mod(\iLa)\longrightarrow \mod(\BH )
	\end{align}
	the natural restriction functor.
	
	For $i\in \I$, define the indecomposable module over $\BH _i$
	\begin{align}
		\label{eq:E}
		\E_i =\begin{cases}
			\bfk[\varepsilon_i]/(\varepsilon_i^2), & \text{ if }\btau i=i;
			\\
			\xymatrix{\bfk\ar@<0.5ex>[r]^1 & \bfk\ar@<0.5ex>[l]^0} \text{ on the quiver } \xymatrix{i\ar@<0.5ex>[r]^{\varepsilon_i} & \btau i\ar@<0.5ex>[l]^{\varepsilon_{\btau i}} }, & \text{ if } \btau i\neq i.
		\end{cases}
	\end{align}
	Then $\E_i$, for $i\in Q_0$, can be viewed as a $\iLa$-module. 
	
	\begin{lemma}[\text{\cite[Proposition 3.9]{LW19}}]
		\label{cor: res proj}
		For any $M\in\mod(\Lambda^{\imath})$ the following are equivalent: (i) $\pd M<\infty$; 
		(ii) $\ind M<\infty$;
		(iii) $\pd M\leq1$;
		(iv) $\ind M\leq1$;
		(v) $\res_\BH (M)$ is projective as an $\BH $-module.
	\end{lemma}
	

	\subsection{$\imath$Hall algebras}
	\label{subsec:iHall}
	
	Let $\ca$ be an essentially small exact category, linear over the finite field $\bfk=\F_q$.
	Assume that $\ca$ has finite morphism and extension spaces:
	$$|\Hom_\ca(A,B)|<\infty,\quad |\Ext^1_\ca(A,B)|<\infty,\,\,\forall A,B\in\ca.$$
	

	We denote by $\Iso(\ca)$ the set of isoclasses of objects of $\ca$.	Given objects $X,Y,Z\in\ca$, define $\Ext^1_\ca(X,Z)_Y\subseteq \Ext^1_\ca(X,Z)$ to be the subset parameterising extensions with the middle term isomorphic to $Y$. We define the Ringel-Hall algebra (also called Hall algebra) $\ch(\ca)$ to be the $\Q$-vector space whose basis is formed by the isomorphism classes $[X]$ of objects $X$ of $\ca$, with the multiplication
	defined by
	\begin{align}
		\label{eq:mult}
		[X]\cdot [Z]=\sum_{[Y]\in \Iso(\ca)}G_{XZ}^Y[Y],\text{ where }G_{XZ}^Y:=\frac{|\Ext_\ca^1(X,Z)_Y|}{|\Hom_\ca(X,Z)|}.
	\end{align}
	It is well known that
	the algebra $\ch(\ca)$ is associative and unital; see \cite{Rin90,Br13}. 
	
	Let $\sqq$ be a fixed square root of $q$ in $\C$. We define the twisted Hall algebra $\widetilde{\ch}(\Lambda^\imath)$ to be the $\Q(\sqq^{1/2})$-algebra with basis given by isoclasses $[M]$ ($M\in\mod(\iLa)$), with the multiplication twisted by the Euler form $\langle-,-\rangle_Q$:
	$$[M]*[N]=\sqq^{\langle M, N\rangle_{Q}} [M]\cdot [N]= \sqq^{\langle M, N\rangle_{Q}} \frac{|\Ext^1_{\Lambda^\imath}(M,N)_L|}{|\Hom_{\Lambda^\imath}(M,N)|}[L].$$
	Here and below, by a slight abuse of notation, we set $\langle M,N\rangle_N:=\langle \res M,\res N\rangle_Q$ for any $M,N\in\mod(\iLa)$.
	
	Let $\mathcal{I}$ be the subspace (ideal) of $\widetilde{\ch}(\Lambda^\imath)$ spanned by all differences 
	\begin{align}
		\label{def:I}
		[M]-[N], \text{ if $\res_\BH(M)=\res_\BH(N)$ and $M\cong N$ in $\cd_{sg}(\Lambda^\imath)$}.
	\end{align}
	Denote  $\widehat{\ch}(\bfk Q,\varrho):=\widetilde{\ch}(\Lambda^\imath)/\mathcal{I}$.

	Let	\begin{equation}
		\label{eq:Sca}
		\cs := \{ a[K] \in \widehat{\ch}(\bfk Q,\varrho) \mid a\in \Q(\sqq^{1/2})^\times, \pd K\leq1\}.
	\end{equation}
	By \cite{LW19}, the right localization of $\ch(\Lambda^\imath)/\mathcal{I}$ with respect to $\cs$ exists, and will be denoted by $\widetilde{\ch}(\bfk Q,\btau)$ or $\cs\cd\widetilde{\ch}(\Lambda^\imath)$, called the $\imath$Hall algebra (also called the twisted semi-derived Ringel-Hall algebra) of $\mod(\Lambda^\imath)$.

	
	
	We have $[\K_i]*[\K_j]=[\K_j]*[\K_i]=[\K_i\oplus \K_j]$ in $\widehat{\ch}(\bfk Q,\varrho)$ for any $i,j\in\I$; cf. \cite[Lemma 4.7]{LW19}. 
	For any $\alpha=(a_i)_{i\in\I}\in\N^{\I}$, we define in $\widehat{\ch}(\bfk Q,\btau)$: 
	$$[\K_\alpha]=[\oplus_{i\in\I}\K_i^{\oplus a_i}]=\prod_{i\in\I}[\K_i]^{a_i}.$$
	Similarly, one can define $[\K_\alpha]$ in $\widetilde{\ch}(\bfk Q,\btau)$ for $\alpha\in\Z^\I$.

	\begin{lemma}[\text{cf. \cite[Proposition 4.9]{LW19}}]
		\label{basis-iHall}
		
		(1) The algebra $\widetilde{\ch}(\bfk Q,\btau)$ has a (Hall) basis given by
		\begin{align}
			\label{eq:Hallbasis-hat}
			\{[X]*[\K_\alpha]\mid X\in\mod(\bfk Q)\subseteq \mod(\Lambda^\imath), \alpha\in\Z^{\I}\}.
		\end{align}
		
		(2) The algebra $\widehat{\ch}(\bfk Q,\btau)$ has a (Hall) basis given by
		\begin{align}
			\label{eq:Hallbasis-tilde}
			\{[X]*[\K_\alpha]\mid X\in\mod(\bfk Q)\subseteq \mod(\Lambda^\imath), \alpha\in\N^{\I}\}.
		\end{align}
	\end{lemma}

	Let $\widetilde{\ct}(\bfk Q,\btau)$ be the subalgebra of $\widetilde{\ch}(\bfk Q,\btau)$ generated by $[\K_\alpha]$, $\alpha\in K_0(\mod(\bfk Q))$, which is a Laurent polynomial algebra in $[\E_i]$, for $i\in \I$. Similarly, one can define the subalgebra $\widehat{\ct}(\bfk Q,\btau)$ of $\widehat{\ch}(\bfk Q,\btau)$, which is a polynomial algebra in $[\E_i]$, for $i\in \I$.

	\begin{lemma}[cf. \text{\cite[Theorem 7.7]{LW19}}]
		\label{lem:Hall-iQG}
		Let $(Q, \btau)$ be a Dynkin $\imath$quiver. Then we have the following isomorphism $\widetilde{\psi}:\tUi|_{v=\sqq}\stackrel{\simeq}{\rightarrow} \widetilde{\ch}(\bfk Q,\btau)$ of $\Q({\sqq^{1/2}})$-algebras, which sends
		\begin{align}
			\label{eq:psi}
			B_i \mapsto \sqq^{-\frac{1}{2}}[S_i], \qquad 
			\tilde{k}_i \mapsto
			\begin{cases}
				[\K_i]&\text{if $\varrho i\neq i$},\\
				\sqq^{-1}[\K_{i}]&\text{if $\varrho i=i$}.
			\end{cases}
		\end{align}
	\end{lemma}
	
	
	Similarly, we have the isomorphism of $\Q(\sqq^{1/2})$-algebras
	$$\widehat{\psi}:\hUi|_{v=\sqq}\longrightarrow \widehat{\ch}(\bfk Q,\btau).$$

	As Example \ref{ex:QGvsiQG} shows, we can view the quantum group $\tU$ as an $\imath$quantum group. Using Example \ref{ex:diagquiver}, we can reformulate Bridgeland's realization $\tU$. 
	\begin{lemma}[Bridgeland's Theorem reformulated] 
		Let $Q$ be a Dynkin quiver. Then we have the following isomorphism of $\Q(\sqq^{1/2})$-algebras
		\begin{align*}
			&\widetilde{\psi}:\tU|_{v=\sqq}\stackrel{\simeq}{\longrightarrow} \widetilde{\ch}(\bfk Q^{\rm dbl},\swa),
			\\
			E_i \mapsto \sqq^{-\frac{1}{2}}[S_i],&\qquad F_i\mapsto \sqq^{-\frac{1}{2}}[S_{i^\diamond}],
			\qquad
			K_i\mapsto [\K_{i^\diamond}],\qquad K_i'\mapsto [\K_i].
		\end{align*}
	\end{lemma}
	
	Similarly, we have the isomorphism of $\Q(\sqq^{1/2})$-algebras
	\[
	\widehat{\psi}:\hU|_{v=\sqq}\longrightarrow \widehat{\ch}(\bfk Q^{\rm dbl},\swa).
	\]

	
	
	\subsection{Generic $\imath$Hall algebras}
	\label{sub:generic}
	
	For a Dynkin $\imath$quiver $(Q,\btau)$,
	we recall the generic $\imath$Hall algebras defined in \cite[\S9.3]{LW19}.
	Recall that $\Phi^+$ is the set of positive roots.
	For any $\beta\in\Phi^+$, denote by $M_q(\beta)$ its corresponding indecomposable $\bfk Q$-module, i.e., $\dimv M_q(\beta)=\beta$.
	Let $\mathfrak{P}:=\mathfrak{P}(Q)$ be the set of functions $\lambda: \Phi^+\rightarrow \N$.
	Then the modules
	\begin{align}
		\label{def:Mlambda}
		M_q(\lambda):= \bigoplus_{\beta\in\Phi^+}\lambda(\beta) M_q(\beta),\quad \text{ for } \lambda\in\mathfrak{P},
	\end{align}
	provide a complete set of isoclasses of $\bfk Q$-modules.
	
	For $(\alpha,\nu),(\beta,\mu)\in\Z^\I\times\fp$, 
	there exists a polynomial $\boldsymbol{\varphi}^{\lambda,\gamma}_{\mu,\alpha;\nu,\beta}(v)\in\Z[v,v^{-1}]$ such that
	\[\big([\K_\alpha]\ast[M_q(\mu)]\big)\ast\big([\K_\beta]\ast[M_q(\nu)]\big)=\sum_{\lambda\in\fp,\gamma\in\Z^\I}\boldsymbol{\varphi}^{\lambda,\gamma}_{\mu,\alpha;\nu,\beta}({\sqq})[\K_\gamma]\ast[M_q(\lambda)]
	\]
	in $\widetilde{\ch}(\bfk Q,\btau)$. 
	The generic $\imath$Hall algebra $\tMHg$ 
	is defined to be the  $\Q(v^{1/2})$-space with a basis $\{\K_\alpha*\fu_\lambda\mid \alpha\in\Z^\I,\lambda\in\fp\}$ 
	with multiplication
	\begin{align}
		\label{eq:generic-mult}
		(\K_\alpha*\fu_\mu)*(\K_\beta*\fu_\nu)=\sum_{\lambda\in\fp,\gamma\in\Z^\I}\boldsymbol{\varphi}^{\lambda,\gamma}_{\mu,\alpha;\nu,\beta}(v)\K_\gamma*\fu_\lambda.
	\end{align}
	
	For $\widehat{\ch}(\bfk Q,\btau)$, we can construct its generic version $\widehat{\ch}(Q,\btau)$ with a basis $\{\K_\alpha*\fu_\lambda\mid \alpha\in\N^\I,\lambda\in\fp\}$ similarly.
	
	From Lemma \ref{lem:Hall-iQG}, see also \cite[Theorem 9.8]{LW19}, we obtain 
	the isomorphisms of $\Q(v^{1/2})$-algebras
	\begin{alignat*}{2}
		&\widetilde{\psi}:\tUi\longrightarrow \widetilde{\ch}(Q,\btau),&\qquad &\widehat{\psi}:\hUi\longrightarrow \widehat{\ch}(Q,\btau),\\
		&\widetilde{\psi}:\tU\longrightarrow \widetilde{\ch}(Q^{\rm dbl},\swa),&\qquad &\widehat{\psi}:\hU\longrightarrow \widehat{\ch}(Q^{\rm dbl},\swa).
	\end{alignat*}


	\subsection{Dual canonical bases of $\imath$Hall aglebras}\label{subsec:dcb of iHall}
	
	Set $\cz=\Z[v^{1/2},v^{-1/2}]$ and $\cz_\sqq=\Z[\sqq^{1/2},\sqq^{-1/2}]$. We recall the integral form and dual canonical bases of $\imath$Hall algebras defined in \cite{LP25}. 
	
	The integral forms
	$\widehat{\ch}(\bfk Q,\varrho)_{\cz_\sqq}$ and $\widetilde{\ch}(\bfk Q,\varrho)_{\cz_\sqq}$ are free $\cz_\sqq$-module with their bases given by \eqref{eq:Hallbasis-hat} and \eqref{eq:Hallbasis-tilde} respectively. Furthermore, the algebra $\tMHg_{\cz}$ (respectively, $\widehat{\ch}(Q,\btau)_{\cz}$) is defined to be the free $\cz$-module with a basis $\{\K_\alpha*\fu_\lambda\mid \alpha\in\Z^\I, \lambda\in\fp\}$ (respectively, $\{\K_\alpha*\fu_\lambda\mid \alpha\in\N^\I, \lambda\in\fp\}$ ), and the multiplication as in \eqref{eq:generic-mult}.  
	
	Using the  isomorphism of $\Q(v^{1/2})$-algebras
	$$\widetilde{\psi}:\tUi\rightarrow \widetilde{\ch}(Q,\btau),\qquad \widehat{\psi}:\hUi\rightarrow \widehat{\ch}(Q,\btau),$$
	we can define the integral form of $\tUi$ (resp. $\hUi$) as the preimage of $\widehat{\ch}(Q,\varrho)_\cz$ (resp. $\widehat{\ch}(Q,\varrho)_\cz$), which is denoted by $\tUi_\cz$ (resp. $\hUi_\cz$). Both of them are free $\cz$-modules, which are independent of orientations; see \cite[Corollary 5.10]{LP25}. Note that $\hUi_\cz=\tUi_\cz\cap\hUi$.

	The bar-involution of $\widehat{\ch}(Q,\varrho)$ (and also $\widetilde{\ch}(Q,\varrho)$) can be defined similarly to $\widetilde{\ch}(Q)$ by setting
	\[\ov{\fu_{\alpha_i}}=v^{-1}\fu_{\alpha_i},\quad \ov{\K_{\alpha_i}}=\K_{\alpha_i}.\]
	
	\begin{lemma}[\cite{LP25}]\label{lem:bar on iHall integral}
		The bar-involution of $\widetilde{\ch}(Q,\varrho)$ (resp. $\widehat{\ch}(Q,\varrho)$) preserves $\widetilde{\ch}(Q,\varrho)_\cz$ (resp. $\widehat{\ch}(Q,\btau)_\cz$).
	\end{lemma}
	
	Following \cite{BG17}, we define an action of $\widehat{\ct}(Q,\varrho)$ on $\widehat{\ch}(Q,\varrho)$ by
	\[\K_\alpha\diamond\fu_\lambda=v^{\frac{1}{2}(\alpha-\varrho\alpha,\,\dimv{M_q(\lambda)})_Q}\K_\alpha\ast\fu_\lambda.\]
	For $\lambda\in\mathfrak{P}$, set
	\begin{equation}\label{eq:HA element H_lambda}
		\mathfrak{U}_\lambda=v^{-\dim\End_{\bfk Q}(M_q(\lambda))+\frac{1}{2}\langle M_q(\lambda),M_q(\lambda)\rangle_Q}\fu_\lambda.
	\end{equation}

	Similarly, we define a partial order on $\N^\I\times\mathfrak{P}$: we say $(\alpha,\lambda)\prec(\beta,\mu)$ if $\alpha+\btau(\alpha)+\dimv M_q(\lambda)=\beta+\btau(\beta)+\dimv M_q(\mu)$ and either $\alpha\prec\beta$ (i.e. $\alpha\neq\beta$ and $\beta-\alpha\in\N^\I$) or $\alpha=\beta$ and $\lambda\prec\mu$. The following theorem is the main result of \cite{LP25}.

	\begin{theorem}[\cite{LP25}]\label{iHA dCB theorem}
		For each $\alpha\in\N^\I$ and $\lambda\in\mathfrak{P}$, there exists a unique element $\mathfrak{L}_{\alpha,\lambda}\in\widehat{\ch}(Q,\varrho)$ such that $\ov{\mathfrak{L}_{\alpha,\lambda}}=\mathfrak{L}_{\alpha,\lambda}$ and
		\[
		\mathfrak{L}_{\alpha,\lambda}-\K_\alpha\diamond \mathfrak{U}_\lambda\in\sum_{(\beta,\mu)}v^{-1}\Z[v^{-1}]\cdot \K_\beta\diamond \mathfrak{U}_\mu.
		\]
		Moreover, $\mathfrak{L}_{\alpha,\lambda}$ satisfies 
		\[
		\mathfrak{L}_{\alpha,\lambda}-\K_\alpha\diamond \mathfrak{U}_\lambda\in\sum_{(\alpha,\lambda)\prec(\beta,\mu)}v^{-1}\Z[v^{-1}]\cdot \K_\beta\diamond \mathfrak{U}_\mu,
		\]
		and $\mathfrak{L}_{\alpha,\lambda}=\K_\alpha\diamond \mathfrak{L}_{0,\lambda}$.
	\end{theorem}
	
	Because of the last property, we often write $\mathfrak{L}_\lambda:=\mathfrak{L}_{0,\lambda}$ for $\lambda\in\mathfrak{P}$ and use $\{\K_\alpha\diamond \mathfrak{L}_\lambda\mid \alpha\in\N^\I,\lambda\in\mathfrak{P}\}$ to denote the basis constructed in Theorem~\ref{iHA dCB theorem}. This is called the dual canonical basis of $\widehat{\ch}(Q,\varrho)$. The dual canonical basis of $\widetilde{\ch}(Q,\varrho)$ is defined to be $\{\K_\alpha\diamond \mathfrak{L}_\lambda\mid \alpha\in\Z^\I,\lambda\in\mathfrak{P}\}$.
	
	\section{NKS quiver varieties for $\imath$quivers}\label{NKS QV section}
	
	In this section, we review the Nakajima-Keller-Scherotzke categories (NKS categories for short) and the NKS quiver varieties for Dynkin $\imath$quivers given in \cite{LW21b}. 
	Throughout this section, we assume that $k$ is an algebraically closed field of characteristic zero, and denote by $\mod(k)$ the category of $k$-vector spaces. For a $k$-linear category $\mathcal{C}$, we denote by $\mod(\mathcal{C})$ the category of $\mathcal{C}$–modules, i.e. $k$-linear contravariant functors $\mathcal{C}\to\mod(k)$. Furthermore, for a category $\mathcal{C}$, we denote by $\mathcal{C}_0$ the set of objects of $\mathcal{C}$.

	\subsection{NKS categories}
	\label{subsec:NKS}
	
	Let $Q=(Q_0=\I,Q_1)$ be a Dynkin quiver and set
	\[\ov{Q}_1=\{\bar{h}:j\rightarrow i\mid (h:i\rightarrow j)\in Q_1\}.\] 
	Define the \emph{repetition quiver} $\Z Q$ of $Q$ as follows:
	
	$\triangleright$ the set of vertices is $\{(i,p)\in Q_0\times \Z\}$;
	
	$\triangleright$ an arrow $(\alpha,p):(i,p)\rightarrow (j,p)$ and an arrow $(\bar{\alpha},p):(j,p-1)\rightarrow (i,p)$ are given, for any arrow $\alpha:i\rightarrow j$ in $Q$ and any vertex $(i,p)$.
	
	Define the automorphism $\tau$ of $\Z Q$ to be the shift by one unit to the left, i.e., $\tau(i,p)=(i,p-1)$ for all $(i,p)\in Q_0\times \Z$.
	
	By a slight abuse of notation, associated to $\beta:y\rightarrow x$ in $\Z Q$, we denote by $\bar{\beta}$ the arrow that runs from $\tau x\rightarrow y$. 
	Let $k(\Z Q)$ be the \emph{mesh category} of $\Z Q$, that is, the objects are given by the vertices of $\Z Q$ and the morphisms are $k$-linear combinations of paths modulo the ideal spanned by the mesh relations
	$R_x:=\sum_{\alpha:y\rightarrow x}\alpha\bar{\alpha},$
	where the sum runs through all arrows of $\Z Q$ ending at $x$.

	By a theorem of Happel \cite{Ha2}, there is an equivalence
	\begin{align}
		\label{eq:Happelfunctor}
		H: k(\Z Q)\stackrel{\simeq}{\longrightarrow} \Ind \cd_Q,
	\end{align}
	where $\Ind \cd_Q$ denotes the category of indecomposable objects in the bounded derived category $\cd_Q=\cd^b(kQ)$.
	Using this equivalence, we label once and for all the vertices of $\Z Q$ by the isoclasses of indecomposable objects of $\cd_Q$. 
	Note that the action of $\tau$ on $\Z Q$ corresponds to the action of the AR-translation on $\cd_Q$, and this explains the notation $\tau$.

	Let 
	\begin{align}
		\label{eq:C}
		C=
		\{\text{the vertices labeled by }\Sigma^j\tS_i, \text{ for all } i\in Q_0 \text{ and } j\in\Z \}.
	\end{align}
	Let $\Z Q_C$ be the quiver constructed from $\Z Q$ by adding to every vertex $c \in C$ a new object denoted by $\sigma c$ together with arrows $\tau c \rightarrow \sigma c$ and $\sigma c \rightarrow c$; we refer to $\sigma c$, for $c\in C$, as {\em frozen vertices}.

	The \emph{graded NKS category} $\mcr^{\gr}_C$ is defined to be the $k$-linear category with
	\begin{itemize}
		\item[$\triangleright$] objects: the vertices of $\Z Q_C$;
		\item[$\triangleright$] morphisms: $k$-linear combinations of paths in $\Z Q_C$ modulo the ideal spanned by the mesh relations
		$\sum_{\alpha:y\rightarrow x}\alpha\bar{\alpha},$
		where the sum runs through all arrows of $\Z Q_C$ ending at $x\in\Z Q$ (including the new arrow $\sigma x \rightarrow x$ if $x\in C$).
	\end{itemize}
	These categories were formulated in \cite{KS16, Sch19}; here and below NKS stands for Nakajima-Keller-Scherotzke. The work \cite{KS16} was in turn motivated by \cite{Na01,Na04,HL15,LeP13}; also cf. \cite{Qin}.

	\begin{remark}
		\label{ex:cyclic}
		Let $Q$ be a Dynkin quiver. In Nakajima's original construction \cite{Na01,Na04}, the graded Nakajima category $\mcr^{\gr}$ is defined similarly with
		$C$  the set of all vertices of $\Z Q$. We denote by $\cs^{\gr}$ the full subcategory of $\mcr^\gr$ formed by all $\sigma c$, $c\in \Z Q$. 
	\end{remark}

	Let $(Q,\btau)$ be a Dynkin $\imath$quiver.
	Let $F:\cd_Q\rightarrow \cd_Q$ be a triangulated isomorphism induced by an isomorphism $F$ of $\Z Q$; \eqref{eq:Happelfunctor}.
	In this paper, we only consider $F=\Sigma^2$ and $F=F^\imath=\Sigma\circ \widehat{\varrho}$. 
	Note that the subset $C$ in \eqref{eq:C} is $F$-invariant. The isomorphism $F$ of $\Z Q$ can be uniquely lifted to an isomorphism of $\Z Q_C$  by setting $F(\sigma c)=\sigma(Fc)$ for any $c\in C$, and then the functor $F$ of the mesh category $k(\Z Q)$ can be uniquely lifted to $\mcr^{\gr}_C$, which is also denoted by $F$.
	
	Let
	\[
	\mcr=\mcr_{C,F}:=\mcr^{\gr}_C/F,
	\]
	and let $\cs=\cs_{C,F}$ be the full subcategory of $\mcr$ formed by all $\sigma c$ ($c\in C$), following \cite{Sch19}. Then $\mcr$ and $\cs$ are called the \emph{regular NKS category} and the \emph{singular NKS category} of the pair $(F,C)$. The quotient category 
	$$\cp:=\mcr/( \cs ),$$ which is equivalent to $k(\Z Q/F)$ is called the \emph{preprojective NKS category}.  By our assumption, $\cd_Q/F$ is a triangulated category and $\Ind \cd_Q/F\simeq \cp$.

	\subsection{NKS quiver varieties}
	Let $\cs$ be a singular NKS category, and $\mcr$ its corresponding regular NKS category. An $\mcr$-module $M$ is \emph{stable} (resp. \emph{costable}) if the support of $\soc(M)$ (respectively,  $\Top(M)$) is contained in $\cs_0$. A module is \emph{bistable} if it is both stable and costable.
	
	Let $\bv \in \N^{\mcr_0-\cs_0}$ and $\bw \in \N^{\cs_0}$ be dimension vectors (with finite supports). Denote by $\rep(\bv,\bw,\mcr)$ the variety of $\mcr$-modules of dimension vector $(\bv,\bw)$. Let $\e_x$ denote the characteristic function of $x\in\mcr_0$, which is also viewed as the unit vector supported at $x$. Let $G_\bv:=\prod_{x\in \mcr_0-\cs_0}\GL({\bv(x)}, k)$.
	
	\begin{definition}
		\label{def:NKS}
		The {\em NKS quiver variety}, or simply {\em NKS variety}, $\cm(\bv,\bw)$ is the quotient $\cs t(\bv,\bw)/G_\bv$,  where $\cs t(\bv,\bw)$ is the subset of $\rep(\bv,\bw,\mcr)$ consisting of all stable $\mcr$-modules of dimension vector $(\bv,\bw)$. Define the affine variety
		\begin{equation}
			\label{eq:M0}
			\cm_0(\bv,\bw)=\cm_0(\bv,\bw, \mcr) :=\rep(\bv,\bw,\mcr)\sslash G_\bv
		\end{equation}
		to be the categorical quotient, whose coordinate algebra is $k[\rep(\bv,\bw,\mcr)]^{G_\bv}$.
	\end{definition}
	Then $\cm(\bv,\bw)$ is a pure dimensional smooth quasi-projective variety; see \cite[Theorem 3.2]{Sch19}. 
	
	We define a partial order $\le$ on the set $\N^{\mcr_0}$ as follows:
	\begin{equation}
		\label{eq:leq}
		\text{  $\bv' \leq \bv
			\Leftrightarrow \bv'(x) \le \bv(x), \forall x \in \mcr_0$. Moreover, $\bv' < \bv \Leftrightarrow \bv'\leq \bv$ and $\bv' \neq \bv$.}
	\end{equation}
	Regarding a dimension vector on $\mcr_0-\cs_0$ as a dimension vector on $\mcr_0$ (by extension of zero),
	we obtain by restriction a partial order $\le$ on the set $\N^{\mcr_0-\cs_0}$.
	
	For $\bv', \bv \in \N^{\mcr_0-\cs_0}$ with $\bv'\leq \bv$ and $\bw \in \N^{\cs_0}$,  there is an inclusion
	\[
	\rep(\bv',\bw,\mcr)\longrightarrow \rep(\bv,\bw,\mcr)
	\]
	by taking a direct sum with the semisimple module of dimension vector $\bv-\bv'$. This yields an inclusion
	\[
	\rep(\bv',\bw,\mcr)\sslash G_{\bv'}\longrightarrow \rep(\bv,\bw,\mcr)\sslash G_\bv.
	\]
	Define the affine variety
	\[
	\cm_0(\bw) =\cm_0(\bw, \mcr) :=\colim\limits_{\bv} \cm_0(\bv,\bw)
	\]
	to be the colimit of $\cm_0(\bv,\bw)$ along the inclusions. 
	Then the projection map
	\begin{equation}   \label{eq:pi}
		\pi:\cm(\bv,\bw)\longrightarrow \cm_0(\bv,\bw),
	\end{equation}
	which sends the $G_\bv$-orbit of a stable $\mcr$-module $M$ to the unique closed $G_\bv$-orbit in the closure of $G_\bv M$, is proper; see \cite[Theorem 3.5]{Sch19}. 
	
	Denote by $\cm^{\text{reg}}(\bv,\bw)\subset \cm(\bv,\bw)$ the open subset consisting of the union of closed $G_\bv$-orbits of stable modules, and then
	\[
	\cm_0^{\text{reg}}(\bv,\bw):=\pi(\cm^{\text{reg}}(\bv,\bw))
	\]
	is an open subset of $\cm_0(\bv,\bw)$. \cite[Lemma 3.4]{Sch19} shows that $\cm(\bv,\bw)$ vanishes on all but finitely many dimension vectors $\bv$. Then $\pi$ induces a stratification
	\begin{equation}\label{eq:stratification for M_0}
		\cm_0(\bw)=\bigsqcup_\bv\cm_0^{\text{reg}}(\bv,\bw)
	\end{equation}
	into finitely many smooth locally closed strata $\cm_0^{\text{reg}}(\bv,\bw)$; see \cite[Theorem 3.5]{Sch19} and its proof.
	
	Let $\res:\mod(\mcr)\rightarrow\mod(\cs)$ be the restriction functor. Then $\res$ induces  morphisms of varieties
	\[
	\res:\cm_0(\bw)\longrightarrow \rep(\bw,\cs),\qquad
	\res\circ\pi:\cm(\bv,\bw)\longrightarrow \rep(\bw,\cs).
	\]
	
	Given $\bv \in \N^{\mcr_0-\cs_0}$,  we define a quantum Cartan matrix (cf. \cite{Sch19})
	\begin{align}
		\label{def:Cq}
		\begin{split}
			{\mathcal C}_q \bv:  \mcr_0-\cs_0 & \longrightarrow\Z,
			\\
			({\mathcal C}_q\bv)(x)& =\bv(x)+\bv(\tau x) -\sum_{y\rightarrow x}\bv(y),
			\quad \text{ for }x \in\mcr_0-\cs_0,
		\end{split}
	\end{align}
	where the sum runs over all arrows $y\rightarrow x$ of $\mcr$ with $y\in\mcr_0-\cs_0$.
	Given $\bw \in \N^{\cs_0}$, define a dimension vector
	\[
	\sigma^*\bw:\mcr_0-\cs_0\longrightarrow\N,
	\qquad
	x \mapsto
	\begin{cases}
		\bw(\sigma x), & \text{ if } x\in C,
		\\
		0, & \text{otherwise}.
	\end{cases}
	\]
	Given $\bv \in \N^{\mcr_0-\cs_0}$, define the dimension vector
	\[
	\tau^*\bv: \mcr_0-\cs_0 \longrightarrow \N,
	\qquad
	x \mapsto \bv(\tau x).
	\]
	
	By \cite[Proposition 4.6]{Sch19}, 
	we obtain the following more precise form of a stratification of $\cm_0(\bw)$
	\begin{equation}   \label{eqn:stratification}
		\cm_0(\bw)=\bigsqcup_{\bv:\sigma^*\bw-{\mathcal C}_q\bv\geq0} \cm_0^{\text{reg}}(\bv,\bw).
	\end{equation}
	
	\begin{definition}   \label{def:pair}
		A pair $(\bv,\bw)$ of dimension vectors $\bv \in \N^{\mcr_0-\cs_0}$ and $\bw \in \N^{\cs_0}$ is called $l$\emph{-dominant} if $\sigma^*\bw-{\mathcal C}_q\bv\geq0$, and $(\bv, \bw)$ is called {\em strongly $l$\emph-dominant} if it is $l$\emph-dominant and  $\cm_0^{\reg}(\bv,\bw) \neq \emptyset.$
	\end{definition}
	
	Recall the $\imath$quiver algebra $\Lambda$ of diagonal type in Example \ref{ex:diagquiver}.
	\begin{definition} [NKS regular/singular categories] 
		\label{def:RS for QG}
		Let $Q$ be a Dynkin quiver. Denote by $\mcr$ and $\cs$ the regular and singular NKS categories associated to the 
		pair $(F=\Sigma^2, C)$. 
	\end{definition}
	
	\begin{proposition}[cf. \text{\cite[Proposition 3.2]{LW21b}}]
		\label{prop:2-complexes}
		Let $Q$ be a Dynkin quiver. Let $\mcr$ and $\cs$ be the regular and singular NKS categories as in Definition~\ref{def:RS for QG}. Then $\Gproj(\Lambda)$ is equivalent to $\proj(\mcr)$, $\proj(\Lambda)$ is Morita equivalent to $\cs$, and we have $\Gproj(\Lambda)\simeq \Gproj(\cs)\simeq \proj(\mcr)$
		as Frobenius categories.
	\end{proposition}

	\begin{definition} [$\imath$NKS regular/singular categories] 
		\label{def:RS for iQG}
		Let $(Q,\btau)$ be a Dynkin $\imath$quiver. Denote by $\mcr^\imath$ and $\cs^\imath$ the regular and singular NKS categories associated to the admissible pair $(F^\imath, C)$, where
		\[
		F^\imath=\Sigma \widehat{\btau}, 
		\quad \text{and }
		C=\{\text{the vertices labeled by }\Sigma^j \tS_i, \text{ for all } i\in Q_0 \text{ and } j\in\Z \}.
		\] 
	\end{definition}

	
	\begin{proposition}[cf. \text{\cite[Theorem 3.6]{LW21b}}]
		\label{thm:iNKS}
		Let $(Q,\btau)$ be a Dynkin $\imath$quiver. Let $\mcr^\imath$ and $\cs^\imath$ be the regular and singular NKS categories as in Definition~\ref{def:RS for iQG}. Then $\Gproj(\Lambda^\imath)$ is equivalent to $\proj(\mcr^\imath)$, $\proj(\Lambda^\imath)$ is Morita equivalent to $\cs^\imath$, and there are equivalences of Frobenius categories
		$
		\Gproj(\Lambda^\imath)\simeq \Gproj(\cs^\imath)\simeq \proj(\mcr^\imath).
		$
	\end{proposition}
	
	
	\subsection{The stratification functors}
	As the restriction functor $\res:\mod(\mathcal{R})\rightarrow \mod(\mathcal{S})$ is a localization functor, it admits a left and a right adjoint functor, which are denoted by $K_L$ and $K_R$ respectively. The intermediate extension $K_{LR}:\mod(\mathcal{S})\rightarrow\mod(\mathcal{R})$ is the image of the canonical map $K_L\rightarrow K_R$. Note that in this case $K_L$ is right exact and $K_R$ is left exact, $K_{LR}$ preserves monomorphisms and epimorphisms (however, it is not exact in general).
	
	\begin{lemma}[{\cite[Lemma 2.12, Lemma 3.8]{Sch19}}]\label{lem:strata iff KLR}
		An $\mathcal{R}$-module $M$ is bistable if and only if $K_{LR}(\res(M))\cong M$. Furthermore, we have $N\in\mathcal{M}_0^{\reg}(\bv,\bw,\mathcal{R})$ if and only if $\dimv K_{LR}(M)=(\bv,\bw)$.
	\end{lemma}
	
	\begin{lemma}\label{lem:K_LR exact if projdim}
		For any exact sequence of $\mathcal{S}$-modules $0\to P\to N\to M\to 0$ with $P$ of finite projective dimension, we have
		\[\dimv K_{LR}(N)=\dimv K_{LR}(M)+\dimv K_{LR}(P),\]
	\end{lemma}
	\begin{proof}
		Recall that $K_L(P)\cong K_{LR}(P)\cong K_R(P)$ by \cite[Corollary 4.3]{LW21b}. Consider the following commutative diagram
		\[\begin{tikzcd}
			&K_L(P)\ar[r]\ar[d,"\cong"]   &K_L(N)\ar[r]\ar[d,two heads]   &K_L(M)\ar[r]\ar[d,two heads]   &0\\
			0\ar[r]&K_{LR}(P)\ar[r]\ar[d,"\cong"]&K_{LR}(N)\ar[r]\ar[d,hook]&K_{LR}(M)\ar[r]\ar[d,hook]&0\\
			0\ar[r]&K_R(P)\ar[r]         &K_R(N)\ar[r]                 &K_R(M)
		\end{tikzcd}\]
		where the first and last rows are exact, and the middle row is exact except at $K_{LR}(N)$. A simple diagram chasing shows that the middle row is also exact at $K_{LR}(N)$, whence the claim.
	\end{proof}
	
	\subsection{Quantum Grothendieck rings}
	\label{subsec: graded Groth ring}
	
	In this subsection, we assume that $\cs,\mcr$ are the NKS categories defined in Definition \ref{def:RS for QG} and Definition \ref{def:RS for iQG}. By \cite[Lemma 3.14]{LW21b}, we know that 
	$\cm_0(\bw)\cong \rep(\bw,\cs)$ for any dimension vector $\bw$, and identify them below.
	
	Let $X$ be an algebraic variety. We denote by $\cd_c(X):=\cd_c^b(X)$ the bounded derived category of constructible sheaves on $X$. Let $f:X\rightarrow Y$ be a morphism of algebraic varieties. There are induced functors $f^*:\cd_c(Y)\rightarrow \cd_c(X)$, $f_*:\cd_c(X)\rightarrow \cd_c(Y)$ and $f_!:\cd_c(X)\rightarrow \cd_c(Y)$ (direct image with compact support). 
	
	We review the quantum Grothendieck ring and its convolution product, following \cite{VV}; also see \cite{Na01,Na04, Qin}. 
	
	For any two dimension vectors $\alpha,\beta$ of $\cs$, let $V_{\alpha+\beta}$ be a vector space of graded dimension $\alpha+\beta$. Fix a vector subspace $W_0\subset V_{\alpha+\beta}$ of graded dimension $\alpha$, and let
	\[
	F_{\alpha,\beta}:=\{y\in \rep(\alpha+\beta,\cs)\mid y(W_0)\subset W_0\}
	\]
	be the closed subset of $\rep(\alpha+\beta,\cs)$. 
	Then $y \in F_{\alpha,\beta}$ induces a natural linear map $y': V/W_0 \rightarrow V/W_0$, i.e., $y' \in \rep(\beta,\cs)$.
	Hence we obtain the following convolution diagram
	\[
	\rep(\alpha,\cs)\times \rep(\beta,\cs) \stackrel{p}{\longleftarrow} F_{\alpha,\beta}\stackrel{q}{\longrightarrow}\rep(\alpha+\beta,\cs),
	\]
	where $p(y):=(y|_{W_0},y')$ and $q$ is the natural closed embedding.

	Let $\cd_c(\rep(\alpha,\cs))$ be the derived category of constructible sheaves on $\rep(\alpha,\cs)$. We have the following restriction functor (called comultiplication),
	\begin{align}
		\label{eq:Delta1}
		\widetilde{\Delta}^{\alpha+\beta}_{\alpha,\beta}: \cd_c \big(\rep(\alpha+\beta,\cs) \big) \longrightarrow \cd_c \big(\rep(\alpha,\cs) \big)\times \cd_c \big(\rep(\beta,\cs) \big), \quad F\mapsto p_!q^*(F).
	\end{align}
	
	Choose a set $\{\alpha_\bv\}$ such that it parameterizes the connected components 
	of $\cm_0(\bv,\bw,\mathcal{R})$. For any $l$-dominant pair $(\bv,\bw)$, since the restriction of $\pi$ on the regular stratum $\cm_0^{\text{reg}}(\bv,\bw,\mathcal{R})$ is a homeomorphism by \cite[Definition 3.6]{Sch19}, the set $\{\alpha_\bv\}$  naturally parameterizes the connected components of this regular stratum:
	\begin{align}
		\cm_0^{\text{reg}}(\bv,\bw,\mathcal{R})=\bigsqcup_{\alpha_\bv} \cm_0^{\reg;\alpha_\bv}(\bv,\bw,\mathcal{R}).
	\end{align}
	
	Let $\underline{k}_{\cm(\bv,\bw)}$ be the constant sheaf on $\cm(\bv,\bw)$. Denote by $\pi^{\cs}(\bv,\bw)\in \cd_c(\rep(\bw,\cs))$ (or $\pi(\bv,\bw)$ when there is no confusion) the pushforward along $\pi: \cm(\bv,\bw)\rightarrow \cm_0(\bw)\cong \rep(\bw,\cs)$ of $\underline{k}_{\cm(\bv,\bw)}$ with a grading shift:
	\begin{align}
		\label{eq:piPS}
		\pi(\bv,\bw):= \pi_!(\underline{k}_{\cm(\bv,\bw)}) [\dim \cm(\bv,\bw,\mathcal{R})].
	\end{align}
	
	For a strongly $l$-dominant pair $(\bv,\bw)$, let $\cl^{\cs}(\bv,\bw)$ (or $\cl(\bv,\bw)$) be the intersection cohomology (IC for short) complex associated to the stratum $\cm_0^{\reg}(\bv,\bw,\mathcal{R})$ with respect to the trivial local system, that is,
	\begin{align}
		\label{eq:decomp}
		\cl(\bv,\bw)=\IC(\cm_0^{\reg}(\bv,\bw,\mathcal{R}))=\bigoplus_{\alpha_\bv}\IC(\cm_0^{\reg;\alpha_\bv}(\bv,\bw,\mathcal{R})).
	\end{align}
	
	We have a decomposition
	\begin{equation}
		\label{eqn:decomposition theorem}
		\pi(\bv,\bw)=\sum_{\bv':\sigma^*\bw-{\mathcal C}_q\bv'\geq0,\bv'\leq \bv} a_{\bv,\bv';\bw}(v)\cl(\bv',\bw),
	\end{equation}
	where we denote by $\cf^{\oplus m}[d]$ by $mv^{-d}\cf$ using an indeterminate $v$, for any sheaf $\cf$, $m\in\N$, and $d\in\Z$ (to compare with \cite{LW21b}, set $v=t^{-1}$). 
	Moreover, we have $a_{\bv,\bv';\bw}(v)\in\N[v, v^{-1}]$, $a_{\bv,\bv';\bw}(v^{-1})=a_{\bv,\bv';\bw}(v)$, and $a_{\bv, \bv;\bw}(v) =1$. (Any $f(v)\in \N[v, v^{-1}]$ such that $f(v^{-1})=f(v)$ is called \emph{bar-invariant}.)

	For each $\bw\in \N^{\cs_0}$, the Grothendieck group $K_\bw(\mod(\cs))$ is defined as the free abelian group generated by the perverse sheaves $\cl(\bv,\bw)$ appearing in (\ref{eqn:decomposition theorem}), for various $\bv$. It has two distinguished $\Z[v, v^{-1}]$-bases by \eqref{eqn:decomposition theorem}:
	\begin{align}
		\label{eq:bases}
		\begin{split}
			\{\pi(\bv,\bw) \mid \sigma^*\bw-{\mathcal C}_q\bv\geq0, \cm_0^{\text{reg}}(\bv,\bw)\neq\emptyset\};
			\\
			\{\cl(\bv,\bw) \mid \sigma^*\bw-{\mathcal C}_q\bv\geq0, \cm_0^{\text{reg}}(\bv,\bw)\neq\emptyset\}.
		\end{split}
	\end{align}
	Consider the free $\Z[v, v^{-1}]$-module
	\begin{equation}
		\label{eq:Kgr}
		K^{\mathrm{gr}}(\mod(\cs)) := \bigoplus_\bw K_\bw(\mod(\cs)).
	\end{equation}
	Then $\{\widetilde{\Delta}^{\bw}_{\bw_1,\bw_2} \}$ induces a comultiplication $\widetilde{\Delta}$ on $K^{\mathrm{gr}}(\mod(\cs))$.
	
	Introduce a bilinear form $d(\cdot,\cdot)$ on $\N^{\mcr_0-\cs_0}$ by letting
	\begin{equation}\label{definition:d}
		d\big((\bv_1,\bw_1),(\bv_2,\bw_2) \big)=(\sigma^*\bw_1-{\mathcal C}_q\bv_1)\cdot \tau^* \bv_2+\bv_1\cdot \sigma^*\bw_2,
	\end{equation}
	where $\cdot$ denotes the standard inner product, i.e., $\bv' \cdot \bv'' =\sum_{x\in \mcr_0-\cs_0} \bv'(x)\bv''(x)$.
	
	Using the same proof of \cite{VV}, Scherotzke-Sibilla obtained the following result. 
	\begin{proposition}[\text{\cite[Proposition 4.8]{SS16}; see also \cite[Lemma 4.1]{VV},\cite[eq. (11)]{Qin} }]
		The comultiplication $\widetilde{\Delta}$ is coassociative and given by
		\begin{align}
			\label{eqn:comultiplication}
			\widetilde{\Delta}^{\bw}_{\bw_1,\bw_2} \big(\pi(\bv,\bw) \big)=\bigoplus_{\stackrel{\bv_1+\bv_2=\bv}{\bw_1+\bw_2=\bw}} v^{d((\bv_1,\bw_1),(\bv_2,\bw_2))-d((\bv_2,\bw_2),(\bv_1,\bw_1))}\pi(\bv_1,\bw_1)\boxtimes\pi(\bv_2,\bw_2).
		\end{align}
	\end{proposition}

	Denote
	\begin{align}
		\label{eq:Kgr2}
		\begin{split}
			{\bf R}_\bw(\mod(\cs)) &=\Hom_{\Z[v, v^{-1}]} (K_\bw(\mod(\cs)),\Z[v, v^{-1}]),
			\\
			K^{\mathrm{gr}*}(\mod(\cs)) &=\bigoplus_{\bw\in \N^{\cs_0}}{\bf R}_\bw(\mod(\cs)).
		\end{split}
	\end{align}
	Then as the graded dual of a coalgebra, $K^{\mathrm{gr}*}(\mod(\cs) )$ becomes a $\Z[v, v^{-1}]$-algebra, whose multiplication is denoted by $\ast$.
	Note that $K^{\mathrm{gr}*}(\mod(\cs))$ is a $\N^{\cs_0}$-graded algebra (called the {\em quantum Grothendieck ring}).  It has two distinguished bases
	\begin{align}
		\label{eq:bases dual}
		\begin{split}
			\{\chi (\bv,\bw) \mid \sigma^*\bw-{\mathcal C}_q \bv\geq0, \cm_0^{\text{reg}}(\bv,\bw)\neq\emptyset\},
			\\
			\{L (\bv,\bw) \mid \sigma^*\bw-{\mathcal C}_q \bv\geq0, \cm_0^{\text{reg}}(\bv,\bw)\neq\emptyset\},
		\end{split}
	\end{align}
	dual to the two bases in \eqref{eq:bases}, respectively.
	
	By the transversal slice Theorem \cite{Na01, LW21b},  the direct summands appearing in $\pi(\bv,\bw)$ are the shifts of sheaves $\cl(\bv',\bw)$ with $\bv'\leq \bv$; cf. \eqref{eq:leq}. So we have
	\begin{equation}
		\label{equation multiplication}
		L(\bv_1,\bw_1)\ast  L(\bv_2,\bw_2)=\sum_{\bv \geq \bv_1+\bv_2} c_{\bv_1,\bv_2}^\bv(v) L(\bv,\bw_1+\bw_2),
	\end{equation}
	with a leading term $c_{\bv_1,\bv_2}^{\bv_1+\bv_2}(v) L(\bv_1+\bv_2,\bw_1+\bw_2)$; moreover, we have
	\begin{align}
		c_{\bv_1,\bv_2}^\bv(v) & \in \N[v, v^{-1}],
		\label{eq:positive}
		\\
		c_{\bv_1,\bv_2}^{\bv_1+\bv_2}(v) &= v^{d((\bv_1,\bw_1),(\bv_2,\bw_2))-d((\bv_2,\bw_2),(\bv_1,\bw_1))}.
		\label{eqn: leading term}
	\end{align}
	

	\subsubsection{Quantum Grothendieck rings for double framed quivers}
	We assume that $\cs,\mcr$ are the NKS categories defined in Definition \ref{def:RS for QG}.
	
	
	Recall $\e_x$ denotes the characteristic function of $x\in\mcr_0$.  
	Define
	\begin{align}
		\bw^{i}&=\e_{\sigma \tS_{ i}}+\e_{\sigma\Sigma\tS_{i}},
		\qquad
		\bv^{i}= \sum_{z\in\mcr_0-\cs_0} \dim\cp(\tS_i ,z)\e_z,\qquad
		\bv^{\Sigma i}=\Sigma^*\bv^i.
	\end{align}
	Denote
	\begin{align}
		&W^{+}=\bigoplus_{x\in\{\tS_i,i\in Q_0\}}\N \e_{\sigma x},&&
		V^{+}=\bigoplus_{x\in\Ind \mod(kQ),\, x\text{ is not injective}} \N \e_x,\\
		&W^{-}=\Sigma^*W^{+},&&
		V^{-}=\Sigma^*V^{+},\\
		&W^{0}=\bigoplus_{i\in Q_0} \N \bw^{i},&&
		V^{0}=\bigoplus_{i\in Q_0} \N \bv^{i}\oplus \N \bv^{\Sigma i}.
	\end{align}
	
	\begin{lemma}[{\cite{Qin}}]
		Let $\bv\in V^0$, $\bw\in W^0$ be such that $\sigma^*\bw-{\mathcal C}_q \bv=0$. Then we have $(\bv,\bw)\in\bigoplus_{i\in \I}\N(\bv^{i},\bw^{i})+\bigoplus_{i\in \I}\N(\bv^{\Sigma i},\bw^{i})$.
	\end{lemma}

	\begin{lemma} [$l$-dominant pair decomposition; see {\cite{Qin}}]
		\label{lem:Qin decomposition}
		If $(\bv,\bw)$ is an $l$-dominant pair for $\mcr$, then we have a unique decomposition of $(\bv,\bw)$ into $l$-dominant pairs $(\bv^+,\bw^+) \in (V^+, W^+)$, $(\bv^-,\bw^-) \in (V^-, W^-)$, $(\bv^0,\bw^0) \in (V^0, W^0)$ such that
		\[
		(\bv^+, \bw^+) + (\bv^-, \bw^-) + (\bv^0, \bw^0)= (\bv, \bw),
		\qquad
		\sigma^*\bw^0-{\mathcal C}_q \bv^0=0.
		\]
	\end{lemma}
	
	\begin{lemma}[{\cite[Proposition 4.2.3]{Qin}}]
		\label{lem:unique V^+W^+}
		For any $l$-dominant pair $(\bv,\bw)\in(V^+,W^+)$, there exists a unique (up to isomorphism) $kQ$-module $N$ such that $\dimv K_{LR}(N)=(\bv,\bw)$.
	\end{lemma}
	
	We have $\mod(\cs)\simeq \mod(\Lambda)$, and we identify them in the following.
	Recall the restriction functor $\res:\mod(\cs) \rightarrow \mod(kQ^{\rm dbl})$. It is natural to identify the Grothendieck groups $K_0(\mod(\cs))$ and $K_0(\mod(kQ^{\rm dbl}))$. Recall the Euler form $\langle-,-\rangle_{Q^{\dbl}}$ of $kQ^{\dbl}$. We define the bilinear forms $\langle-,-\rangle_{Q^{\rm dbl},a}$ as follows: for any dimension vectors $\bw,\bw' \in \N^{\cs_0}$, let
	\begin{align}
		\langle \bw,\bw'\rangle_{Q^{\rm dbl},a}&= \langle \bw,\bw'\rangle_{Q^{\rm dbl}}-\langle \bw',\bw\rangle_{Q^{\rm dbl}}.\label{eqn: antisymmetric bilinear form}
	\end{align}

	Let us fix a square root $v^{1/2}$ of $v$ once for all. As we need a twisting involving $v^{1/2}$ shortly, we shall consider the ring $\cz:=\Z[v^{1/2}, v^{-1/2}]$ and the field $\Q(v^{1/2})$.
	
	The coalgebra $K^{\mathrm{gr}}(\mod(\cs))$ (cf. \eqref{eq:Kgr}) and its graded dual (cf. \eqref{eq:Kgr2}) up to a base change here in the current setting read as follows:
	\begin{gather}
		K^{\mathrm{gr}}(\mod(\cs)) =\bigoplus_{\bw\in W^{+}+W^{-}} K_\bw(\mod(\cs)),
		\label{eq:KS}
		\\
		\hRZ=\bigoplus_{\bw\in W^{+}+W^{-}} \hR_{\cz,\bw},
		\quad \text{where } \hR_{\cz,\bw}  =   \Hom_{\cz} \big(K_\bw(\mod(\cs)),\cz\big).
		\label{eq:tR}
	\end{gather}
	
	Then $(\hRZ, \cdot)$ is the $\Z[v^{1/2}, v^{-1/2}]$-algebra corresponding to the coalgebra $K^{\mathrm{gr}}(\mod(\cs))$ with the {\em twisted} comultiplication
	\begin{equation}
		\label{eq:tw}
		\{\Res^{\bw}_{\bw^1,\bw^2} :=v^{\frac12\langle \bw^1,\bw^2\rangle_{Q^{\dbl},a}} \widetilde{\Res}^{\bw}_{\bw^1,\bw^2}\};
	\end{equation}
	in practice, the product sign $\cdot$ is often omitted.
	
	We shall need the $\Q(v^{1/2})$-algebra obtained by a base change below:
	\begin{align}
		\label{eq:hRQZ}
		\hR  =\Q(v^{1/2}) \otimes_{\cz} \hRZ.
	\end{align}
	We also set $\hR^+$ to be the submodule of $\hR$ generated by $L(\bv,\bw)$, $(\bv,\bw)\in (V^+,W^+)$. It also becomes an algebra under $\Delta^\bw_{\bw_1+\bw_2}$, and the inclusion $\hR^+\hookrightarrow\hR$ is an algebra homomorphism. Similarly we can define $\hR^-$.
	
	We define $\tRZ$ to be the localization of $\tRZ$ with respect to the multiplicatively closed subset generated by $L(\bv^i,\bw^i)$ and $L(\bv^{\Sigma i},\bw^i)$ ($i\in Q_0$). Define 
	\begin{align}
		\label{eq:tRQZ}
		\tR  =\Q(v^{1/2}) \otimes_{\cz} \tRZ.
	\end{align} 
	
	The following theorem recovers the main result of \cite{Qin}.
	
	\begin{theorem}
		Let $Q$ be a Dynkin $\imath$quiver. Then there exists an isomorphism of $\Q(v^{\frac{1}{2}})$-algebras $\tilde{\kappa}:\tU\stackrel{\simeq}{\rightarrow}\tR$ which  sends
		\begin{align*}
			E_i\mapsto L(0,\e_{\sigma\tS_i}),\quad F_i\mapsto L(0,\e_{\sigma \Sigma \ts_i}),\quad K_i\mapsto L(\bv^i,\bw^i),\quad K_i'\mapsto L(\bv^{\Sigma i},\bw^i).
		\end{align*}
	\end{theorem}

	\subsubsection{Quantum Grothendieck rings for $\imath$quivers}
	Let $(Q,\btau)$ be a Dynkin $\imath$quiver and $\cs^\imath$, $\mcr^\imath$ and $\cp^\imath$ be the $\imath$NKS categories defined in Definition \ref{def:RS for iQG} and Theorem \ref{thm:iNKS}. We define
	\begin{align}
		\label{eq:v^i}
		\bw^{i}&=\e_{\sigma \tS_{ i}}+\e_{\sigma\tS_{\btau i}},
		&&
		\bv^{i}= \sum_{z\in\mcr^\imath_0-\cs^\imath_0} \dim\cp^\imath(\tS_i ,z)\e_z,
		\\
		\label{eq:VW+}
		{^\imath W^{+}}&=\bigoplus_{x\in\{\tS_i,i\in Q_0\}}\N \e_{\sigma x},&&
		{^\imath V^{+}}=\bigoplus_{
			x\in\Ind \mod(kQ),\, x\text{ is not injective}} \N \e_x,\\
		\label{eq:VW0}
		{^\imath W^{0}}&=\bigoplus_{i\in Q_0} \N \bw^{i},&&
		{^\imath V^{0}}=\bigoplus_{i\in Q_0} \N \bv^{i}.
	\end{align}
	
	\begin{lemma}[\text{\cite[Lemma 6.13]{LW21b}}] 
		\label{lem:R0 subalgebra}
		Let $\bv\in {}^\imath V^0$, $\bw\in {}^\imath W^0$ be such that $\sigma^*\bw-{\mathcal C}_q \bv=0$. Then we have $(\bv,\bw)\in\bigoplus_{i\in \I}\N(\bv^{i},\bw^{i})$.
	\end{lemma}

	\begin{lemma} [$l$-dominant pair decomposition; see  \text{\cite[Lemma 6.10]{LW21b}}]
		\label{lem:DP3}
		If $(\bv,\bw)$ is an $l$-dominant pair for $\mcr^\imath$, then we have a unique decomposition of $(\bv,\bw)$ into $l$-dominant pairs $(\bv^+,\bw^+) \in ({^\imath V^+}, {^\imath W^+})$, $(\bv^0,\bw^0) \in ({}^\imath V^0, {}^\imath W^0)$ such that
		\[
		(\bv^+, \bw^+) + (\bv^0, \bw^0)= (\bv, \bw),
		\qquad
		%
		\sigma^*\bw^0-{\mathcal C}_q \bv^0=0.
		\]
	\end{lemma}
	
	\begin{remark}\label{remk:strongly dominant iff dominant}
		If $(\bv,\bw)=(\bv^0,\bw^0)+(\bv^+,\bw^+)$ is the decomposition of an $l$-dominant pair $(\bv,\bw)$ given by Lemma~\ref{lem:DP3}, then by Lemma~\ref{lem:unique V^+W^+} we can find a $kQ$-module $M$ such that $\dimv K_{LR}(M)=(\bv^+,\bw^+)$. Also, by Lemma~\ref{lem:R0 subalgebra} there exists an $\mathcal{S}^\imath$-module $P$ such that $\dimv K_{LR}(P)=(\bv^0,\bw^0)$. Then we have $\dimv K_{LR}(P\oplus M)=(\bv,\bw)$ (see Lemma~\ref{lem:K_LR exact if projdim}), so $\mathcal{M}_0^{\reg}(\bv,\bw,\mathcal{R}^\imath)\neq\emptyset$. In other words, every $l$-dominant pair is strongly $l$-dominant.
	\end{remark}
	
	We have $\mod(\cs^\imath)\simeq \mod(\Lambda^\imath)$, and we will identify them in the following. Recall the restriction functor $\res:\mod(\cs^\imath) \rightarrow \mod(kQ)$. It is natural to identify the Grothendieck groups $K_0(\mod(\cs^\imath))$ and $K_0(\mod(kQ))$. Then we can define $\langle \bw,\bw'\rangle_{Q,a}$ for dimension vectors $\bw,\bw'\in\N^{\cs_0^\imath}$ similar to \eqref{eqn: antisymmetric bilinear form}.

	Like the case for $(Q^{\rm dbl},\mathrm{swap})$, the coalgebra $K^{\mathrm{gr}}(\mod(\cs^\imath))$ (cf. \eqref{eq:Kgr}) and its graded dual (cf. \eqref{eq:Kgr2}) up to a base change here in the current setting read as follows:
	\begin{gather}
		K^{\mathrm{gr}}(\mod(\cs^\imath)) =\bigoplus_{\bw\in W^{+}} K_\bw(\mod(\cs^\imath)),
		\label{eq:KSi}
		\\
		\hRiZ=\bigoplus_{\bw\in W^{+}} \hR^{\imath}_{\cz,\bw},
		\quad \text{ where } \hR^{\imath}_{\cz,\bw}  =   \Hom_{\cz} \big(K_\bw(\mod(\cs^\imath)),\cz\big).
		\label{eq:tRi}
	\end{gather}
	
	Then $(\hRiZ, \cdot)$ is the $\cz$-algebra corresponding to the coalgebra $K^{\mathrm{gr}}(\mod(\cs^\imath))$ with the {\em twisted} comultiplication
	\begin{equation}
		\label{eq:tw}
		\{\Res^{\bw}_{\bw^1,\bw^2} :=v^{\frac12\langle \bw^1,\bw^2\rangle_{Q,a}} \widetilde{\Res}^{\bw}_{\bw^1,\bw^2}\};
	\end{equation}
	in practice, the product sign $\cdot$ is often omitted.
	
	We shall need the $\Q(v^{1/2})$-algebra obtained by a base change below:
	\begin{align}
		\label{eq:hRiQZ}
		\hRi  =\Q(v^{1/2}) \otimes_{\cz} \hRiZ.
	\end{align}

	We define $\tRiZ$ to be the localization of $\hRiZ$ with respect to the multiplicatively closed subset generated by $L(\bv^i,\bw^i)$ ($i\in Q_0$). (It is well defined by \cite[Lemmas 5.2 and 5.3]{LW21b}.) Define \begin{align}
		\label{eq:tRiQZ}
		\tRi  =\Q(v^{1/2}) \otimes_{\cz} \tRiZ.
	\end{align} 
	Note that there is a natural inclusion $\hR^+\hookrightarrow\hRi$. But it is not an algebra homomorphism in general.
	
	\begin{theorem}[{\cite[Theorem 6.25]{LW21b}}]
		\label{thm:iQG-sheaf}
		Let $(Q,\varrho)$ be a Dynkin $\imath$quiver. Then there exists an isomorphism of $\Q(v^{\frac{1}{2}})$-algebras $\tilde{\kappa}:\tUi\stackrel{\simeq}{\rightarrow}\tRi$ which  sends
		\begin{align}
			\label{eq:kappa}
			B_i\mapsto L(0,\e_{\sigma\tS_i}),  \qquad
			\tilde{k}_i \mapsto
			\begin{cases}
				L(\bv^{\varrho i},\bw^i)&\text{if $\varrho i\neq i$},\\
				v^{-1}L(\bv^i,\bw^i)&\text{if $\varrho i=i$}.
			\end{cases}
		\end{align}
	\end{theorem}

	\begin{lemma}\label{lem: multiply by L^0}
		For any strongly $l$-dominant pairs $(\bv,\bw)$ and $(\bv^0,\bw^0)\in ({^\imath V^0},{^\imath W^0})$ such that $\sigma^*\bw^0-\mathcal{C}_q\bv^0=0$, we have 
		\[L(\bv^0,\bw^0)\cdot L(\bv,\bw)\in v^{\frac{1}{2}\Z}L(\bv+\bv^0,\bw+\bw^0).\]
	\end{lemma}
	\begin{proof}
		By induction it suffices to prove the assertion for $(\bv^0,\bw^0)=(\bv^i,\bw^i)$ for some $i\in Q_0$. For this we need to consider $\widetilde{\Delta}^{\bw+\bw^i}_{\bw^i,\bw}\pi(\bv,\bw+\bw^i)$: by \eqref{eqn:comultiplication} we have
		\begin{equation}\label{lem: multiply by L^0-1}
			\widetilde{\Delta}^{\bw+\bw^i}_{\bw^i,\bw+\bw^i}\pi(\bv,\bw+\bw^i)=\bigoplus_{\bv_1+\bv_2=\bv}v^{d((\bv_1,\bw^i),(\bv_2,\bw))-d((\bv_2,\bw),(\bv_1,\bw^i))}\pi(\bv_1,\bw^i)\boxtimes\pi(\bv_2,\bw).
		\end{equation}
		
		First note that if for some vector $(\bv_1,\bw^i)$ we have $\mathcal{M}(\bv_1,\bw^i,\mathcal{R}^\imath)\neq\emptyset$ and $\mathcal{L}(\bv^i,\bw^i)$ appears in the decomposition of $\pi(\bv_1,\bw^i)$, then $\mathcal{M}_0^{\reg}(\bv^i,\bw^i,\mathcal{R}^\imath)$ is contained in the image of $\mathcal{M}(\bv_1,\bw^i,\mathcal{R}^\imath)$ under $\pi:\mathcal{M}(\bv_1,\bw^i,\mathcal{R}^\imath)\rightarrow\mathcal{M}_0(\bv_1,\bw^i,\mathcal{R}^\imath)$. By the transversal slice theorem \cite[Proposition 3.16]{LW19}, the fiber over a point $x\in\mathcal{M}_0^{\reg}(\bv^i,\bw^i,\mathcal{R}^\imath)$ is isomorphic to the fiber of $\pi:\mathcal{M}(\bv_1^\bot,\bw^\bot,\mathcal{R}^\imath)\rightarrow\mathcal{M}_0(\bv_1^\bot,\bw^\bot,\mathcal{R}^\imath)$ over the origin, where 
		\[\bv^\bot_1=\bv_1-\bv^i,\quad \bw^\bot=\bw^i-{\mathcal C}_q\bv^i\sigma^{-1}=0.\]
		However, we know that $\mathcal{M}(\bv_1^\bot,0,\mathcal{R}^\imath)\neq\empty$ only if $\bv_1^\bot=0$, so the only possibility is $\bv_1=\bv^i$. Moreover, as for the component $\pi(\bv^i,\bw^i)\boxtimes\pi(\bv-\bv^i,\bw)$, we note that
		\begin{align*}
			d(\bw)&:=d((\bv^i,\bw^i),(\bv-\bv^i,\bw^i))-d((\bv-\bv^i,\bw),(\bv^i,\bw^i))\\
			&=\bv^i\cdot\sigma^*\bw-(\sigma^*\bw-\mathcal{C}_q\bv)\cdot\tau^*\bv^i-\sigma^*\bw^i\cdot(\bv-\bv^i+\tau^*\bv^i)\\
			&=(\bv^i-\tau^*\bv^i)\cdot\sigma^*\bw+\mathcal{C}_q\bv\cdot\tau^*\bv^i-\mathcal{C}_q\bv^i\cdot\bv+\sigma^*\bw^i\cdot(\bv^i-\tau^*\bv^i)\\
			&=(\bv^i-\tau^*\bv^i)\cdot\sigma^*\bw+\sigma^*\bw^i\cdot(\bv^i-\tau^*\bv^i)
		\end{align*}
		where we have used the equality $(\mathcal{C}_q\bv_1)\cdot\tau^*\bv_2=\bv_1\cdot(\mathcal{C}_q\bv_2)$ for any $\bv_1,\bv_2\in\N^{\mathcal{R}^\imath_0-\mathcal{S}^\imath_0}$. 
		
		With these observations, we then deduce from \eqref{lem: multiply by L^0-1} that
		\begin{align*}
			\widetilde{\Delta}^{\bw+\bw^i}_{\bw^i,\bw}\pi(\bv,\bw+\bw^i)&=v^{d(\bw)}\pi(\bv^i,\bw^i)\boxtimes\pi(\bv-\bv^i,\bw)+\cdots\\
			&=v^{d(\bw)}\mathcal{L}(\bv^i,\bw^i)\boxtimes\pi(\bv-\bv^i,\bw)+\cdots,
		\end{align*}
		where we have omitted terms with first component $\neq\mathcal{L}(\bv^i,\bw^i)$. Plugging into \eqref{eqn:decomposition theorem}, we then get
		\begin{align*}
			\sum_{0\leq\bv'\leq\bv}\widetilde{\Delta}^{\bw+\bw^i}_{\bw^i,\bw}a_{\bv,\bv';\bw+\bw^i}\mathcal{L}(\bv',\bw+\bw^i)=v^{d(\bw)}\sum_{0\leq\bv''\leq\bv-\bv^i}a_{\bv-\bv^i,\bv'';\bw}\mathcal{L}(\bv^i,\bw^i)\boxtimes\mathcal{L}(\bv'',\bw)+\cdots.
		\end{align*}
		By the transversal slice theorem we have $a_{\bv,\bv';\bw}=a_{\bv-\bv',0;\bw-{\mathcal C}_q\bv\sigma^{-1}}$, and $\mathcal{L}(\bv^i,\bw^i)$ appears in the decomposition of $\widetilde{\Delta}^{\bw+\bw^i}_{\bw^i,\bw}\mathcal{L}(\bv',\bw+\bw^i)$ only if $\bv^i\leq\bv'$. Therefore we can write
		\begin{equation}\label{lem: multiply by L^0-2}
			\begin{aligned}
				&\sum_{\bv^i\leq\bv'\leq\bv}\widetilde{\Delta}^{\bw+\bw^i}_{\bw^i,\bw}a_{\bv-\bv',0;\bw+\bw^i-{\mathcal C}_q\bv'\sigma^{-1}}\mathcal{L}(\bv',\bw+\bw^i)\\
				&=v^{d(\bw)}\sum_{0\leq\bv''\leq\bv-\bv^i}a_{\bv-\bv^i-\bv'',0;\bw-{\mathcal C}_q\bv''\sigma^{-1}}\mathcal{L}(\bv^i,\bw^i)\boxtimes\mathcal{L}(\bv'',\bw)+\cdots.
			\end{aligned}
		\end{equation}
		We also note that by Remark~\ref{remk:strongly dominant iff dominant}, $(\bv-\bv^i,\bw)$ is strongly $l$-dominant if and only if $(\bv,\bw+\bw^i)$ is strongly $l$-dominant.
		
		To finish the proof, it suffices to show that for any $\bv\geq\bv^i$ with $(\bv,\bw+\bw^i)$ is strongly $l$-dominant,
		\[\widetilde{\Delta}^{\bw+\bw^i}_{\bw^i,\bw}\mathcal{L}(\bv,\bw+\bw^i)=v^{d(\bw)}\mathcal{L}(\bv^i,\bw^i)\boxtimes\mathcal{L}(\bv-\bv^i,\bw)+\cdots.\]
		For this we proceed by induction on the height of $\bv-\bv^i$. The case $\bv=\bv^i$ follows directly from \eqref{lem: multiply by L^0-2}. Now let $\bv>\bv^i$ and assume the assertion for any $\bv'$ such that $\bv>\bv'\geq\bv^i$. Then subtracting the terms $\widetilde{\Delta}^{\bw+\bw^i}_{\bw^i,\bw}\mathcal{L}(\bv',\bw+\bw^i)$, $\bv^i\leq\bv'<\bv$ from \eqref{lem: multiply by L^0-2} then implies the assertion for $\mathcal{L}(\bv,\bw+\bw^i)$. This completes the induction step.
	\end{proof}
	\subsection{Reduction to Gorenstein projectives}\label{subsec:Gproj reduction}
	
	In \cite{SS16} the authors consider the subvariety of $\rep(\bw,\mathcal{S})$ consisting of Gorenstein projective $\mathcal{S}$-modules, and show that its quantum Grothendieck ring recovers $\tR$. In this subsection we will extend this result to $\mathcal{S}^\imath$. Let $\rep^{\Gp}(\bw,\mathcal{S}^\imath)$ be the subvariety of $\rep(\bw,\mathcal{S}^\imath)$ consisting of Gorenstein projective $\mathcal{S}^\imath$-modules. By \cite[Lemma 4.5]{SS16}, this is an open subvariety of $\rep(\bw,\mathcal{S}^\imath)$. 
	
	The following lemma allows us to identify the stratification of $\rep^{\Gp}(\bw,\mathcal{S}^\imath)$ inherited from $\rep(\bw,\mathcal{S}^\imath)$. One should compare with Lemma~\ref{lem:strata M_0^reg char}.
	
	\begin{lemma}\label{lem:Gproj unique strata}
		The stratification of $\rep^{\Gp}(\bw,\mathcal{S}^\imath)$ induced by \eqref{eqn:stratification} coincides with the orbit stratification.
	\end{lemma}
	
	\begin{proof}
		For simplicity, for an $\mathcal{R}^\imath$-module $X$ we write $\dimv X=\dimv^0X+\dimv^+X$ for the decomposition given by Lemma~\ref{lem:DP3}, where $\dimv^0X\in({^\imath V^0},{^\imath W^0})$ and $\dimv^+X\in({^\imath V^+},{^\imath W^+})$.
		
		It suffices to prove that any Gorenstein projective $\mathcal{S}^\imath$-module $M$ is determined (up to isomorphism) by $\dimv K_{LR}(M)$. Now $M$ can be written as $P\oplus G$, where $P$ is a projective $\Lambda^\imath$-module and $G$ has no projective summands. We then have
		\[\dimv K_{LR}(M)=\dimv K_{LR}(P)+\dimv K_{LR}(G).\]
		Since $P$ is a projective $\Lambda^\imath$-module, we have $\dimv K_{LR}(P)\in ({^\imath V^0},{^\imath W^0})$ by \cite[Corollary 4.3]{LW21b}, so $\dimv^+K_{LR}(M)=\dimv^+K_{LR}(G)$. By \cite[Theorem 3.18]{LW19}, we can find a unique $kQ$-module $X$ such that $G\cong X$ in $\cd_{sg}(\Lambda^{\imath})\simeq \ul{\Gproj}(\Lambda^\imath)$. If we choose a short exact sequence
		\[\begin{tikzcd}
			0\ar[r]&P_X\ar[r]&G_X\ar[r]&X\ar[r]&0
		\end{tikzcd}\]
		such that $G_X$ is Gorenstein projective and $P_X$ is projective, then we have $G_X\cong X\cong G$ in $\cd_{sg}(\Lambda^{\imath})\simeq \ul{\Gproj}(\Lambda^\imath)$. Since $G$ has no projective summands, this means there exists a projective $\Lambda^\imath$-modules $P'$ such that $G_X\cong P'\oplus G$. In particular,
		\[\dimv K_{LR}(P')+\dimv K_{LR}(X)=\dimv K_{LR}(G_X).\]
		Again, we know that $\dimv K_{LR}(P')$ is in $({^\imath V^0},{^\imath W^0})$, so 
		\[\dimv^+K_{LR}(G)=\dimv^+K_{LR}(G_X).\]
		On the other hand, from Lemma~\ref{lem:K_LR exact if projdim} we get
		\[\dimv K_{LR}(G_X)=\dimv K_{LR}(P_X)+\dimv K_{LR}(X),\]
		so $\dimv^+K_{LR}(G_X)=\dimv^+K_{LR}(X)$. 
		
		To summarize, we have proved that for any Gorenstein projective $\mathcal{S}^\imath$-module $M=P\oplus G$, with $G$ having no projective summands, there exists a unique $kQ$-module $X$ such that
		\[\dimv^+K_{LR}(M)=\dimv^+K_{LR}(X).\]
		Using Lemma~\ref{lem:unique V^+W^+}, we know that $X$ is uniquely determined by $\dimv^+K_{LR}(M)$. Moreover, $G$ is determined by $X$ in the sense that it is the component of $G_X$ with no projective summands. Also note that once $G$ is determined, the dimension vector $\dimv K_{LR}(P)=\dimv K_{LR}(M)-\dimv K_{LR}(G)$ is also determined, and $P$ is uniquely determined (up to isomorphism) by $\dimv K_{LR}(P)$ (see \cite[Corollary 4.3]{LW21b} and Lemma~\ref{lem:R0 subalgebra}, note that indecomposable projective $\mathcal{S}^\imath$-modules form a basis of the Grothendieck group $K_0(\mathcal{P}^{\leq1}(\mathcal{S}^\imath))$). The lemma is therefore proved.
	\end{proof}

	For any Gorenstein projective $\mathcal{S}^\imath$-module $M$, by Lemma~\ref{lem:Gproj unique strata} there exists a unique strongly $l$-dominant pair $(\bv,\bw)$ such that $\dimv K_{LR}(M)=(\bv,\bw)$. We denote by $\mathcal{L}(M)$ the pullback of $\mathcal{L}(\bv,\bw)$ to the orbit of $M$, which is also the intersection cohomology complex associated to the constant local system on $\mathfrak{O}_M$. Let $i:\rep^{\Gp}(\bw,\mathcal{S}^\imath)\hookrightarrow \rep(\bw,\mathcal{S}^\imath)$ denotes the natural inclusion. Then $i^*$ commutes with the restriction functor $\Delta^{\bw}_{\bw_1,\bw_2}$ \cite[Lemma 4.9]{SS16}. We denote by $\tR^{\imath,\Gp}$ the quantum Grothendieck ring defined for $\rep^{\Gp}(\bw,\mathcal{S}^\imath)$ and let $L(M)$ be the dual basis of $\mathcal{L}(M)$. Then there is a natural embedding
	\[\hR^{\imath,\Gp}\hookrightarrow \hRi,\quad L(M)\mapsto L(\bv,\bw)\]
	where $\dimv K_{LR}(M)=(\bv,\bw)$. The following theorem is proved in \cite{SS16} for $\tR$, and it can be extended to $\tRi$.

	\begin{theorem}\label{tRi reduction to Gp}
		The pullback map induces an isomorphism $\tR^{\imath,\Gp}\stackrel{\sim}{\to}\tRi$, which preserves dual canonical basis.
	\end{theorem}
	\begin{proof}
		In view of Lemma~\ref{lem: multiply by L^0} it suffices to show that for any $(\bv^0,\bw^0)\in({^\imath V^0},{^\imath W^0})$ and $(\bv^+,\bw^+)\in({^\imath V^+},{^\imath W^+})$, the elements $L(\bv^0,\bw^0)$ and $L(\bv^+,\bw^+)$ belong to the image of $\tR^{\imath,\Gp}$. The case for $L(\bv^0,\bw^0)$ is proved exactly as in \cite[Lemma 5.9]{SS16}, so we concentrate on $L(\bv^+,\bw^+)$. In this case let $X$ be the unique $k Q$-module with $\dimv K_{LR}(X)=(\bv^+,\bw^+)$, and choose a short exact sequence $0\to P_X\to G_X\to X\to0$ 
		such that $G_X$ is Gorenstein projective and $P_X$ is projective. In this case
		\[\dimv K_{LR}(P_X)+\dimv K_{LR}(X)=\dimv K_{LR}(G_X)\]
		and $\dimv K_{LR}(P_X)\in({^\imath V^0},{^\imath W^0})$, so by Lemma~\ref{lem: multiply by L^0},
		\[L(P_X)L(\bv^+,\bw^+)\in v^{\frac{1}{2}\Z}L(G_X).\]
		We can therefore write $L(\bv^+,\bw^+)\in v^{\frac{1}{2}\Z}L(P_X)^{-1}L(G_X)$, and the assertion follows.
	\end{proof}

	
	\subsection{Dual canonical bases of $\tU^+$}
	\label{dCB of HA subsec}
	
	For a Dynkin quiver $Q$, we let $\widetilde{\ch}(\bfk Q)$ be the twisted Hall algebra of $\mod(\bfk Q)$ over $\Q(\sqq^{1/2})$, that is, 
	$$[M]\cdot [N]=\sqq^{\langle M,N\rangle_Q}\sum_{[L]}\frac{|\Ext^1(M,N)_L|}{|\Hom(M,N)|}[L],\quad \forall M,N\in\mod(\bfk Q).$$
	Let $\widetilde{\ch}(Q)$ be the generic version; see \cite[\S4.1]{LP25}. Then $\widetilde{\ch}(Q)$ has a basis $\{\fu_\lambda\mid \lambda\in\fp\}$. There exists an isomorphism of $\Q(v^{1/2})$-algebras
	\begin{align*}
		\psi^+:\tU^+\longrightarrow \widetilde{\ch}(Q),\quad E_i\mapsto v^{-\frac{1}{2}}\fu_{\alpha_i},\quad\forall i\in\I.
	\end{align*}

	Similarly, there exists an isomorphism of $\Q(v^{1/2})$-algebras 
	\begin{align*}
		\kappa^+:\tU^+\longrightarrow \widetilde{\mathbf{R}}^+,\quad E_i\mapsto L(0,\e_{\sigma\tS_i}),\quad \forall i\in\I.
	\end{align*}
	Note that there is an algebra isomorphism 
	\[\Omega^+:\widetilde{\ch}(Q)\longrightarrow \tR^+,\quad \fu_{\alpha_i}\mapsto v^{\frac{1}{2}}L(0,\e_{\sigma\tS_i}),\quad\forall i\in\I.\]
	So we obtain a commutative diagram
	\begin{equation}\label{eq:diagram of kappa psi Omega}
		\begin{tikzcd}[row sep=5mm,column sep=5mm]
			{}&\tR^+\\
			\U^+\ar[ru,"\kappa^+"]\ar[rd,swap,"\psi^+"]&{}\\
			{}&\widetilde{\ch}(Q)\ar[uu,swap,"\Omega^+"]
		\end{tikzcd}
	\end{equation}
	
	For $\lambda\in\fp$, we view $\mathfrak{U}_\lambda$ defined in \eqref{eq:HA element H_lambda} to be in $\widetilde{\ch}(Q)$. The bar-involution of $\tU^+$ induces the one of $\widetilde{\ch}(Q)$, which is determined by $\ov{\fu_{\alpha_i}}=\fu_{\alpha_i}$, for $i\in\I$. 
	\begin{theorem}
		\label{thm:dualCB-U+}
		For $\lambda\in\fp$, there exists a unique element $\mathfrak{C}_\lambda\in\widetilde{\ch}(Q)$ such that $\ov{\mathfrak{C}_\lambda}=\mathfrak{C}_\lambda$ and 
		\[\mathfrak{C}_\lambda-\mathfrak{U}_\lambda\in\sum_{\mu\in\mathfrak{P}}v^{-1}\Z[v^{-1}]\mathfrak{U}_\mu.\]
		Moreover, $\mathfrak{C}_\lambda$ satisfies
		\[\mathfrak{C}_\lambda-\mathfrak{U}_\lambda\in\sum_{\lambda\prec\mu}v^{-1}\Z[v^{-1}]\mathfrak{U}_\mu.\]
	\end{theorem}
	
	Let $\mathbf{B}^-$ be 
	Lusztig's canonical basis of $\U^-$ (via the generators $\cf_i:=\frac{E_i}{v-v^{-1}}$ ($i\in\I$)). 
	Following \cite{Ka91}, we define a Hopf pairing on $\U^-\otimes\U^+$ by
	\begin{align} 
		\label{eq:hopf}(F_i,E_j)_{K}=\delta_{ij}(v-v^{-1}).
	\end{align}
	
	For $b\in\mathbf{B}^-$, we denote by $\delta_b\in\U^+$ the dual basis of $b$ under the pairing $(\cdot,\cdot)_K$. We define a norm function $N:\Z^\I\rightarrow\Z$ by
	\[N(\alpha)=\frac{1}{2}(\alpha,\alpha)_Q-\eta(\alpha).\]
	where $\eta:\Z^\I\rightarrow\Z$ is the augmentation map defined by $\eta(\sum_ia_i\alpha_i)=\sum_ia_i$ (cf. \cite{GLS13,HL15}). The rescaled dual canonical basis of $\U^+$ is then defined to be
	\[\widetilde{\mathbf{B}}^+:=\{v^{\frac{1}{2}N(-\deg(b))}\delta_b\mid b\in\mathbf{B}^-\}.\]

	\begin{theorem}[{\cite{HL15, Qin}}]\label{dCB of U^+ by L}
		Under the isomorphism $\kappa^+:\tU^+\rightarrow\tR^+$, $E_i\mapsto L(0,\e_{\sigma\tS_i})$, the basis $\{L(\bv,\bw)\mid (\bv,\bw)\in(V^+,W^+)\}$ gets identified with the rescaled dual canonical basis $\widetilde{\mathbf{B}}^+$ of $\tU^+$. 
	\end{theorem}

	\begin{proposition}\label{dCB of HA is L}
		We have $\Omega^+(\mathfrak{C}_\lambda)=L(\bv_\lambda,\bw_\lambda)$ for any $\lambda\in\fp$, where $(\bv_\lambda,\bw_\lambda)\in(V^+,W^+)$ is such that $\sigma^*\bw_\lambda-\mathcal{C}_q\bv_\lambda=\lambda$. 
	\end{proposition}

	For any $\lambda\in\fp$, we define a PBW basis of $\widetilde{\ch}(Q)$ by 
	\[
	\mathfrak{E}_\lambda=v^{\dim \End(M_q(\lambda))-\frac{1}{2}\dim M_q(\lambda)}\frac{\fu_\lambda}{a_\lambda(v)},
	\]
	here $a_\lambda(v)\in\Z[v,v^{-1}]$ is the polynomial such that $a_\lambda(\sqq)=|\aut(M_q(\lambda))|$; see e.g. \cite[Lemma 10.19]{DDPW}. The corresponding canonical basis $\mathfrak{B}_\lambda\in \U^+$ then satisfies
	\[
	\psi^+(\mathfrak{B}_\lambda)-\mathfrak{E}_\lambda\in\sum_{\mu\prec\lambda}v^{-1}\Z[v^{-1}]\mathfrak{E}_\mu.
	\]
	Let $\mathfrak{E}_\lambda^*$ ($\lambda\in\fp$) be the dual PBW basis with respect to the paring of $\widetilde{\ch}(Q)$:
	\begin{align}
		\label{eq:hopf-Hall}
		([M],[N])_K=\delta_{M,N}|\aut(M)|.
	\end{align}
	Then one can check that 
	\begin{align}
		\label{eq:TE}
		\mathfrak{U}_\lambda=v^{\frac{1}{2}n_\lambda}\mathfrak{E}_\lambda^*, \text{ where $n_\lambda=N(\sum_k\lambda(\beta_k)\beta_k)$}.
	\end{align}
	
	Let $\mathfrak{B}_\lambda^*$ ($\lambda\in\fp$) be the dual canonical basis of $\U^+$ with respect to the paring \eqref{eq:hopf}. Then we have 
	\begin{align*}
		\psi^+(\mathfrak{B}_\lambda^*)\in \mathfrak{E}_\lambda^*+\sum_{\lambda\prec\mu}v^{-1}\Z[v^{-1}]\mathfrak{E}_\mu^*,\quad \ov{\mathfrak{B}_\lambda^*}=v^{n_\lambda}\mathfrak{B}_\lambda^*.
	\end{align*}
	In view of \eqref{eq:TE} and Theorem~\ref{thm:dualCB-U+}, this means $\psi^+(v^{\frac{1}{2}n_\lambda}\mathfrak{B}_\lambda^*)=\mathfrak{C}_\lambda$. 

	\section{Dual canonical bases via perverse sheaves}\label{sec:dCB via sheaf}
	
	In this section we will relate the construction made in \S\ref{subsec:dcb of iHall} with quantum Grothendieck rings by proving that the dual canonical basis of $\widetilde{\ch}(Q,\varrho)$ is mapped to the basis $\{L(\bv,\bw)\mid\text{$(\bv,\bw)$ is strongly $l$-dominant}\}$ of $\tRi$ under the isomorphism $\widetilde{\Omega}:\widetilde{\ch}(Q,\btau)\rightarrow\tRi$.
	
	
	\subsection{Integral forms}

	Using the isomorphisms given in Theorem \ref{thm:iQG-sheaf} and Lemma \ref{lem:Hall-iQG}, we obtain two isomorphisms of $\Q(v^{1/2})$-algebras:
	\begin{align}
		\widetilde{\Omega}:&\widetilde{\ch}(Q,\btau)\longrightarrow \tR^{\imath},\qquad \widehat{\Omega}:\widehat{\ch}(Q,\btau)\longrightarrow \hRi,
	\end{align}
	defined by $
	\fu_{\alpha_i}\mapsto v^{\frac{1}{2}}L(0,1_{\sigma S_i})$ and $\K_{\alpha_i}\mapsto L(\mathbf{v}^{\varrho i},\mathbf{w}^i)$. We shall prove that $\widetilde{\Omega}$ and $\widehat{\Omega}$ induce isomorphisms of integral forms.


	\begin{lemma}
		\label{lem: hom for Dynkin}
		Let $(Q,\btau)$ be a  Dynkin $\imath$quiver.  For any indecomposable $\bfk Q$-modules $M$ and $N$ with their supports on the same connected component of $Q$, if
		$\Hom_{\bfk Q}(M,N)=0$, then $\dim_\bfk\Hom_{\bfk Q}(M,\btau (N))\leq1$.
	\end{lemma}
	
	\begin{proof}
		By applying the AR-translation, we can assume that $M=P_l$ for some indecomposable projective $\bfk Q$-module $P_l$.
		
		Let $\dimv N= (a_i)_{i\in\I}$. Then $\dimv \btau (N)= (a_{\btau i})_{i\in\I}$. It follows from $\Hom_{\bfk Q}(M,N)=0$ that $a_l=0$. Moreover, $\dim_\bfk\Hom_{\bfk Q}(M,\btau (N))=\dim_\bfk\Hom_{\bfk Q}(P_l,\btau( N))=a_{\btau l}$. It is equivalent to prove that $a_{\btau l}\leq1$. The proof is broken into the following two cases.
		
		{\bf Case (i)} $\btau l=l$. Then  $\dim_\bfk\Hom_{\bfk Q}(M,\btau (N))=0$.
		
		{\bf Case (ii)} $\btau l\neq l$. Then $\btau\neq\Id$. If $(Q,\btau)$ is of diagonal type (see Example \ref{ex:diagquiver}), we have $\Hom_{\bfk Q}(M,\varrho (N))=0$ by noting that their supports are not on the same connected component of $Q$.
		
		It remains to consider the case that $Q$ is connected. Then $Q$ is of type $A_{2r+1}$, $D_n$ and $E_6$.
		If $Q$ is of type $A_{2r+1}$, we get $a_{\btau l}\leq1$, then $\dim_\bfk\Hom_{\bfk Q}(M,\btau (N))\leq1$.
		If $Q$ is of type $D_n$ with its underlying graph as \eqref{diag: D} 
		shows, then $l=n-1$ or $n$. So $a_{\btau l}\leq1$. If $Q$ is of type $E_6$ with its underlying graph as \eqref{diag: E} 
		shows, then $l=1,2,5,6$. If $l=1$, from the root system of $E_6$, we know $a_6\leq1$; if $l=2$, similarly, $a_5\leq1$. 
		
		In conclusion,  $\dim_\bfk\Hom_{\bfk Q}(M,\btau (N))\leq1$.
	\end{proof}
	
	\begin{lemma}[cf. \cite{Rin3}]
		\label{lem:short exact sequence}
		Let $k$ be any field, and $Q$ be of type ADE.
		For any indecomposable $kQ$-module $M$, there exist two indecomposable $kQ$-modules $M_1$ and $M_2$ such that the following hold:
		\begin{align}
			\label{eqn: split cond 1}
			&\Hom(M_1,M_2)=0=\Hom(M_2,M_1),\qquad \dim\Ext^1_{kQ}(M_2,M_1)=1;
			\\\label{eqn: split cond 2}
			&\text{there exists a short exact sequence
			} 0\rightarrow M_1\rightarrow M\rightarrow M_2\rightarrow0.
		\end{align}
		Furthermore, in this case, we have $\Ext^1_{kQ}(M_1,M_2)=0$.
	\end{lemma}

	\begin{theorem}
		\label{prop:Hall-sheaf-integ}
		The isomorphisms of  $\Q(v^{1/2})$-algebras $\widetilde{\Omega}:\widetilde{\ch}(Q,\btau)\rightarrow \tR^{\imath}$,
		$\widehat{\Omega}:\widehat{\ch}(Q,\btau)\rightarrow \hRi$  induces isomorphisms of $\cz$-algebras $\widetilde{\Omega}:\widetilde{\ch}(Q,\btau)_\cz\rightarrow \tRiZ$ and 
		$\widehat{\Omega}:\widehat{\ch}(Q,\btau)_\cz\rightarrow \hRiZ$.
	\end{theorem}
	
	\begin{proof}
		From its definition we know that $\widehat{\Omega}$ restricts to an embedding $\widehat{\ch}(Q,\btau)_\cz\hookrightarrow \hRi$. First, we prove that $\widehat{\Omega}(\fu_M)\in \hRiZ$ for any $M\in\mod(\bfk Q)\subseteq\mod(\Lambda^\imath)$. We prove it by induction on the dimension of $M$. If $M$ is simple, it is obvious.
		
		By abuse of notation, we denote by $\fu_M\in\widehat{\ch}(Q,\btau)_\cz$ the corresponding element of $M\in\mod(\bfk Q)$. 
		For general $M$, if $M$ is decomposable, that is, $M=\bigoplus_{i=1}^{n} M_i$ with $M_i$ indecomposable and $n\geq 2$. Denote by $N=\bigoplus_{i=2}^{n}M_i$. Since $Q$ is of Dynkin type, we can assume that $\Ext^1_{\bfk Q}(M_1,N)=0$. If $\Ext^1_{\Lambda^\imath}(M_1,N)\neq0$, then for any $L\in\mod(\Lambda^\imath)$ such that
		$L\notin \mod(\bfk Q)$ and $|\Ext^1_{\Lambda^\imath}(M_1,N)_L|\neq 0$, we have $[L]=[X\oplus\E_\alpha]$ in 
		$\widehat{\ch}(Q,\btau)$ for some $X\in\mod(\bfk Q)$ and $\alpha\in \N^\I$. Then
		\begin{align*}
			\fu_{M_1}\ast\fu_N=v^{-\langle M_1,N\rangle_Q}\big(\fu_M+\sum_{[X]\in\Iso(\mod(\bfk Q)),\alpha\in\N^\I} f_{X,\alpha}(v) \fu_X\ast\K_{\alpha} \big),
		\end{align*}
		with $f_{X,\alpha}(v)\in\Z[v,v^{-1}]$. 
		By the inductive assumption, it is easy to see that $\widehat{\Omega}(\fu_M)\in \hRiZ$.
		
		If $M$ is indecomposable, Lemma \ref{lem:short exact sequence} shows that we can find indecomposable $\bfk Q$-modules $M_1,M_2$ which admit a short exact sequence
		$0\rightarrow M_1\rightarrow M\rightarrow M_2\rightarrow0$. Furthermore, $\Hom(M_2,M_1)=0=\Hom(M_1,M_2)$ and $\Ext_{\bfk Q}^1(M_1,M_2)=0$, $\Ext^1_{\bfk Q}(M_2,M_1)=\bfk$. On the other hand, $\Ext^1_{\Lambda^\imath}(M_2,M_1)\cong\Ext^1_{\bfk Q}(M_2,M_1)\oplus \Hom_{\bfk Q}(M_2,\btau (M_1))$; see e.g. \cite[Lemma 2.9]{LR24}. We claim that $\Hom_{\bfk Q}(M_2,\btau (M_1))=0$. Otherwise,
		we have a cycle in $\mod(\bfk Q)$: $M_1\rightarrow M\rightarrow M_2\rightarrow \btau (M_1)\rightarrow \btau (M)\rightarrow \btau (M_2)\rightarrow M_1$, contradicts to that $\bfk Q$ is representation-directed. We then get
		\begin{align}
			\label{eqn:M2M1}
			\fu_{M_2}\ast\fu_{M_1}&=v^{-1}\fu_{M_1\oplus M_2}+(v-v^{-1})\fu_M,\\
			\label{eqn:M1M2}
			\fu_{M_1}\ast\fu_{M_2}&=\fu_{M_1\oplus M_2} + \sum_{[X]\in\Iso(\mod(\bfk Q)),\alpha\in\N^\I} f_{X,\alpha}(v) \fu_{X}\ast\K_\alpha
		\end{align}
		for some $f_{X,\alpha}(v) \in\Z[v,v^{-1}]$ such that $(v^2-1)\mid f_{X,\alpha}(v)$, i.e., there is $g_{X,\alpha}(v)\in\Z[v,v^{-1}]$ such that $f_{X,\alpha}(v)=(v^2-1)g_{X,\alpha}(v)$. Note that
		$$\Ext^1_{\Lambda^\imath}(M_1,M_2)\cong\Ext^1_{\bfk Q}(M_1,M_2)\oplus \Hom_{\bfk Q}(M_1,\btau (M_2))=\Hom_{\bfk Q}(M_1,\btau (M_2)).$$
		Lemma \ref{lem: hom for Dynkin} shows that $\dim_\bfk\Ext^1_{\Lambda^\imath}(M_1,M_2)\leq 1$, so there exists at most one $L\notin \mod(\bfk Q)$ such that
		$|\Ext^1_{\Lambda^\imath}(M_1,M_2)_L|\neq 0$. If such $L$ exists, then  $|\Ext^1_{\Lambda^\imath}(M_1,M_2)_L|=q-1$. We therefore conclude that
		\begin{align}
			\fu_M=\frac{1}{v-v^{-1}}(\fu_{M_2}\ast\fu_{M_1}-v^{-1}\fu_{M_1}\ast\fu_{M_2})+s_{M_1,M_2}v^{m} \fu_{X}\ast\K_\alpha
		\end{align}
		for some $X$, $\alpha\in\N^\I$, $m\in\Z$, and $s_{M_1,M_2}=0$ or $1$. By inductive assumption, we have $\widehat{\Omega}(\fu_{X}\ast\K_\alpha)=\widehat{\Omega}(\fu_X)\cdot \widehat{\Omega}(\K_\alpha)\in \hRiZ$.
		
		Let $\bw'=\dimv M_1$ and $\bw''=\dimv M_2$. By definition of $\widehat{\Omega}$ and \eqref{equation multiplication}, we know
		$$\widehat{\Omega}(\fu_{M_1})=\sum_{\bv'}a_{\bv'}(v) L(\bv',\bw'),\qquad \widehat{\Omega}(\fu_{M_2})=\sum_{\bv''}b_{\bv''}(v) L(\bv'',\bw''),$$
		and then
		by inductive assumption, we can see $a_{\bv'}(v),a_{\bv''}(v)\in\mathcal{Z}$. %
		Note that $\langle\bw',\bw''\rangle_{Q,a}=1$. For any $L(\bv',\bw')$, $L(\bv'',\bw'')$, by \eqref{equation multiplication}, we have \begin{align*}
			L(\bv'',\bw'')\cdot L(\bv',\bw')=&v^{1/2}\sum_{\bv\geq \bv'+\bv''}c_{\bv'',\bv'}^{\bv}(v) L(\bv,\bw'+\bw''),\\
			L(\bv',\bw')\cdot L(\bv'',\bw'')=&v^{-1/2}\sum_{\bv\geq \bv'+\bv''}c_{\bv',\bv''}^{\bv}(v) L(\bv,\bw'+\bw'')
		\end{align*}
		for some $c_{\bv',\bv''}^{\bv}(v),c_{\bv'',\bv'}^{\bv}(v)\in\N[v,v^{-1}]$. Then 
		\begin{align*}
			&\frac{1}{v-v^{-1}}\widehat{\Omega}(\fu_{M_2}\ast\fu_{M_1}-v^{-1}\fu_{M_1}\ast\fu_{M_2})\\
			=&\frac{1}{v-v^{-1}}\Big(\sum_{\bv',\bv''}a_{\bv''}(v)a_{\bv'}(v) (L(\bv'',\bw'')\cdot L(\bv',\bw')-v^{-1}L(\bv',\bw')\cdot L(\bv'',\bw''))\Big)\\
			=&\frac{1}{v-v^{-1}}\sum_{\bv',\bv''}a_{\bv''}(v)a_{\bv'}(v)\sum_{\bv\geq \bv'+\bv''}(v^{\frac{1}{2}}c_{\bv'',\bv'}^{\bv}(v)-v^{-\frac{3}{2}}c_{\bv',\bv''}^{\bv}(v))\cdot L(\bv,\bw'+\bw'').
		\end{align*}
		A simple calculation shows that
		\[
		\frac{v^{\frac{1}{2}}c_{\bv'',\bv'}^{\bv}(v)-v^{-\frac{3}{2}}c_{\bv',\bv''}^{\bv}(v)}{v-v^{-1}}\in\mathcal{Z}
		\]
		by noting that $c_{\bv'',\bv'}^{\bv}(v)=c_{\bv',\bv''}^{\bv}(v^{-1})$; cf. \eqref{eqn:comultiplication} and \eqref{eq:tw}. So $\widehat{\Omega}(\fu_M)\in\hRiZ$.
		
		To see that $\widehat{\Omega}(\widehat{\ch}(Q,\btau))=\hRiZ$, we define a partial order on strongly $l$-dominant pairs for $\mathcal{R}^\imath$ as follows: recall that by Lemma~\ref{lem:DP3}, every $l$-dominant pair $(\bv,\bw)$ for $\mathcal{R}^\imath$ can be uniquely written as $(\bv,\bw)=(\bv^0,\bw^0)+(\bv^+,\bw^+)$, where $(\bv^0,\bw^0)\in (V^0,W^0)$ and $(\bv^+,\bw^+)\in({^\imath V^+},{^\imath W^+})$. We then say that $(\bv,\bw)\prec(\tilde{\bv},\tilde{\bw})$ if $\bw=\tilde{\bw}$ and 
		\begin{itemize}
			\item either $(\bv^0,\bw^0)\prec(\tilde{\bv}^0,\tilde{\bw}^0)$ (that is, $(\bv^0,\bw^0)\neq(\tilde{\bv}^0,\tilde{\bw}^0)$ and $(\tilde{\bv}^0-\bv^0,\tilde{\bw}^0-\bw^0)\in {^\imath V^+}$), or
			\item $(\bv^0,\bw^0)=(\tilde{\bv}^0,\tilde{\bw}^0)$ (hence $\bw^+=\tilde{\bw}^+$) and $\bv^+<\tilde{\bv}^+$.
		\end{itemize}
		Consider the quotient maps 
		\begin{align}
			\label{eq:iHA quotient}
			\pi:&\widehat{\ch}(Q,\varrho)\longrightarrow \widetilde{\ch}(Q),\quad \K_\alpha\ast\fu_\lambda\mapsto \begin{cases}
				\fu_\lambda&\text{if $\alpha=0$},\\
				0&\text{otherwise},
			\end{cases}\\
			\label{eq:tRi quotient}
			\pi:&\hRi\longrightarrow \hR^+,\quad L(\bv,\bw)\mapsto \begin{cases}
				L(\bv,\bw)&\text{if $(\bv,\bw)\in({^\imath V^+},{^\imath W^+})$},\\
				0&\text{otherwise}.
			\end{cases}
		\end{align}
		Note that these maps are homomorphisms of $\Q(v^{1/2})$-algebras; cf. \cite[Proposition 3.8]{LP25}, Lemma~\ref{lem: multiply by L^0}, and the proof of \cite[Proposition 6.19]{LW21b}. We have a commutative diagram
		\[\begin{tikzcd}
			\widehat{\ch}(Q,\varrho)\ar[d,swap,"\pi"]\ar[r,"\widehat{\Omega}"]&\hRi\ar[d,"\pi"]\\
			\widetilde{\ch}(Q)\ar[r,"\Omega^+"]&\hR^+
		\end{tikzcd}\]
		By the relation between the basis $\mathfrak{C}_\lambda$ and $\mathfrak{U}_\lambda$ of $\widetilde{\ch}(Q)$ (together with Proposition~\ref{dCB of HA is L}), we see that 
		\[\Omega^+(\mathfrak{U}_\lambda)\in L(\bv_\lambda,\bw_\lambda)+\sum_{\substack{(\tilde{\bv},\tilde{\bw})\in({^\imath V^+},^{\imath W^+})\\(\bv_\lambda,\bw_\lambda)\prec(\tilde{\bv},\tilde{\bw})}} \mathcal{Z}\cdot L(\tilde{\bv},\tilde{\bw}),\quad\forall (\bv,\bw)\in({^\imath V^+},{^\imath W^+})\]
		where $(\bv_\lambda,\bw_\lambda)$ is such that $\sigma^*\bw_\lambda-\mathcal{C}_q\bv_\lambda=\lambda$. From the definition of our quotient maps, it follows that
		\[\widehat{\Omega}(\mathfrak{U}_\lambda)\in L(\bv_\lambda,\bw_\lambda)+\sum_{(\bv_\lambda,\bw_\lambda)\prec(\tilde{\bv},\tilde{\bw})}\mathcal{Z}\cdot L(\tilde{\bv},\tilde{\bw})\]
		The assertion now follows by induction on dimension.
	\end{proof}
	
	\begin{corollary}
		Let $(Q,\varrho)$ be a Dynkin $\imath$quiver. Then there exist isomorphisms of $\cz$-algebras  
		$\widetilde{\psi}:\tUi_\cz\stackrel{\simeq}{\rightarrow}\widetilde{\ch}(Q,\btau)_\cz$ sending  $B_i,\tk_i$ as in \eqref{eq:psi},  
		and $\widetilde{\kappa}:\tUi_\cz\stackrel{\simeq}{\rightarrow}\tRiZ$ sending  $B_i,\tk_i$ as in \eqref{eq:kappa}.
	\end{corollary}

	\subsection{$\imath$Hall algebras via functions}\label{subsec:QV function iHA}
	
	In this subsection the quiver varieties will be assumed to be defined over $\F$, the algebraic closure of a finite field $\F_q$. By considering $\F$-vector spaces with $\F_q$-structures, we may assume that $\rep(\bw,\mathcal{S}^\imath)$ and $\mathcal{M}_0^{\reg}(\bv,\bw,\mathcal{R}^\imath)$ have $\F_q$-structures.
	
	We fix a prime number $\ell$, invertible in $\F_q$. Given $\bw\in\N^{\mathcal{S}^\imath_0}$, let $\mathcal{F}_\bw$ be the vector space of $\ov{\Q}_\ell$-valued functions on the set of $\F_q$-rational points of $\rep(\bw,\mathcal{S}^\imath)$ which are constant on the set of $\F_q$-rational points of any $G_\bw$-orbit on $\rep(\bw,\mathcal{S}^\imath)$.
	
	For dimension vectors $\mathbf{w}_1,\mathbf{w}_2$ of $\N^{\mathcal{S}^\imath_0}$, we fix an $\I$-graded vector space $V$ of dimension $\bw=\bw_1+\bw_2$ and a subspace $W\subseteq V$ of dimension $\bw_2$. Let 
	\[F_{\bw_1,\bw_2}=\{y\in\rep(\bw,\mathcal{S}^\imath)\mid y(W)\subseteq W\}.\]
	Then we have the following convolution diagram: 
	\[\begin{tikzcd}
		\rep(\mathbf{w}_1,\mathcal{S}^\imath)\times\rep(\mathbf{w}_2,\mathcal{S}^\imath) & F_{\bw_1,\bw_2}\ar[r,"q"]\ar[l,swap,"p"] & \rep(\mathbf{w}_1+\bw_2,\mathcal{S}^\imath)
	\end{tikzcd}\]
	Using this we can define a restriction map $\widetilde{\res}_{\mathbf{w}_1,\mathbf{w}_2}^{\mathbf{w}}$ on $\mathcal{F}_{\bw}$ as follows:
	\[
	\widetilde{\res}_{\mathbf{w}_1,\mathbf{w}_2}^{\bw}:\mathcal{F}_{\bw} \longrightarrow \mathcal{F}_{\mathbf{w}_1}\times \mathcal{F}_{\mathbf{w}_2},\quad f\mapsto p_!q^*(f).
	\]
	As we shall see, this is related to the restriction functor defined in \S\ref{subsec: graded Groth ring}, hence defines a comultiplication on $\mathcal{F}(Q,\varrho)=\bigoplus_{\bw}\mathcal{F}_{\bw}$. For a representation $M$ in $\rep(\bw,\mathcal{S}^\imath)$ defined over $\F_q$, we denote by $\delta_{[M]}$ the characteristic function on the orbit $\mathfrak{O}_M$ of $M$. It is clear that these functions form a basis of $\mathcal{F}_\mathbf{w}$.
	
	\begin{lemma}\label{res of const function iHA}
		Given $M\in\rep(\bw_1,\mathcal{S}^\imath)(\F_q)$ and $N\in\rep(\bw_2,\mathcal{S}^\imath)(\F_q)$, we have
		\[\widetilde{\res}_{\mathbf{w}_1,\mathbf{w}_2}^{\bw}(\delta_{[L]})(M,N)=q^{\mathbf{w}_1\cdot\mathbf{w}_2}\frac{|\Ext^1_{\mathcal{S}^\imath}(M,N)_L|}{|\Hom_{\mathcal{S}^\imath}(M,N)|}.\]
	\end{lemma}
	\begin{proof}
		For an $\F$-variety $X$ with an $\F_q$-structure, we will use $|X|$ to denote the number of $\F_q$-rational points of $X$. By definition, $\widetilde{\res}_{\mathbf{w}_1,\mathbf{w}_2}^{\bw}(\delta_{[L]})(M,N)$ is computed as 
		\[\widetilde{\res}_{\mathbf{w}_1,\mathbf{w}_2}^{\bw}(\delta_{[L]})(M,N)=|D(M,N)\cap\mathfrak{O}_L|\]
		where $\mathfrak{O}_L$ is the orbit of $L$ and $D(M,N)$ is the fiber of $p$ at $(M,N)$. Assume that $N=(V,x_h, x_i)_{h\in Q_1,i\in Q_0}$ and $M=(W,y_h,y_i)_{h\in Q_1,i\in Q_0}$. For a linear space $U$ of $\dim V+\dim W$, we fix an isomorphism $U\cong V\oplus W$ and identify them in the following. Then an element of $D(M,N)$ consists of maps
		\[z_h=\begin{pmatrix}
			x_h&w_h\\
			0&y_h
		\end{pmatrix},\quad z_i=\begin{pmatrix}
			x_i&\eta_i\\
			0&y_i
		\end{pmatrix}
		\]
		for any $h=(i\rightarrow j)\in Q_1$, $i\in\I$, and
		\[w_h:W_i\longrightarrow V_j,\quad \eta_i:W_i\longrightarrow V_{\varrho j},
		\]
		so we may denote any element $K$ of $D(M,N)$ by $K=(U,w,\eta)$, where $w=(w_h)_{h\in Q_1}$, $\eta=(\eta_i)_{i\in\I}$. 
		We claim that there is an exact sequence
		\[\begin{tikzcd}[column sep=6mm]
			0\ar[r]&\Hom_{\mathcal{S}^\imath}(M,N)\ar[r]&\prod_{i\in\I}\Hom(W_i,V_i)\ar[r,"\alpha"]&D(M,N)\ar[r,"\pi"]&\Ext_{\mathcal{S}^\imath}^1(M,N)\ar[r]&0
		\end{tikzcd}\]
		In fact, the image of $K=(U,w,\eta)$ and $\widetilde{K}=(U,\tilde{w},\tilde{\eta})$ in $\Ext_{\mathcal{S}^\imath}^1(M,N)$ are equal if and only if there exists a commutative diagram
		\[\begin{tikzcd}
			0 \ar[r] & N \ar[r] \ar[d,equal] &        K  \ar[r] \ar[d,"\cong"'] & M \ar[r] \ar[d,equal] & 0 \\
			0 \ar[r] & N \ar[r]              &\widetilde{K} \ar[r]            & M \ar[r]              & 0
		\end{tikzcd}\]
		In other words, the middle isomorphsim is given by linear maps of the form
		\[\begin{pmatrix}
			1&s_i\\
			0&1
		\end{pmatrix}:U_i\rightarrow U_i\]
		where $s_i:W_i\rightarrow V_i$ satisfy the following equations:
		\begin{align*}
			w_h-\tilde{w}_h             &= x_hs_i-s_jy_h, \text{ for any } (h:i\rightarrow j)\in Q_1,           \\
			\eta_i-\tilde{\eta}_i       &= x_is_i-s_{\varrho i}y_i, \text{ for any }i\in\I.
		\end{align*}
		The map $\alpha:\prod_i\Hom(M_i,N_i)\rightarrow D(M,N)$ is then defined by
		\[(s_i)_{i\in\I}\mapsto (U,(x_hs_i-s_jy_h)_{(h:i\rightarrow j)\in Q_1},(x_is_i-s_{\varrho i}y_i)_{i\in\I}),\]
		which satisfies $\ker(\pi)=\Im(\alpha)$. This map is not injective, and its kernel is exactly $\Hom_{\mathcal{S}^\imath}(M,N)$. We therefore get the desired exact sequence.
		
		Now the inverse image $\pi^{-1}(\Ext_{\mathcal{S}^\imath}^1(M,N)_L)$ is equal to $D(M,N)\cap\mathfrak{O}_L$, so by taking $\F_q$-rational points we find
		\begin{align*}
			|D(M,N)\cap\mathfrak{O}_L|&=|\Ext_{\mathcal{S}^\imath}^1(M,N)_L|\cdot\frac{|\prod_{i\in\I}\Hom(M_i,N_i)|}{|\Hom_{\mathcal{S}^\imath}(M,N)|}= q^{\bw_1\cdot\bw_2}\frac{|\Ext_{\mathcal{S}^\imath}^1(M,N)_L|}{|\Hom_{\mathcal{S}^\imath}(M,N)|},
		\end{align*}
		which proves the lemma.
	\end{proof}
	
	Fix a square root $\sqq$ of $q$ (also a square root of $\sqq$) in $\ov{\Q}_\ell$. We define 
	\[\res_{\mathbf{w}_1,\mathbf{w}_2}^{\bw} = \sqq^{\frac{1}{2}\langle\mathbf{w}_1,\mathbf{w}_2\rangle_{Q,a}} \cdot \widetilde{\res}_{\mathbf{w}_1,\mathbf{w}_2}^{\bw}.\]
	By taking graded dual, the maps $\res_{\mathbf{w}_1,\mathbf{w}_2}^{\bw}$ defines a multiplication on $\mathcal{F}^*(Q,\varrho)=\bigoplus\mathcal{F}^*_{\bw}$, which is denoted by $\star$. Let $\xi_{[M]}$ be the basis of $\mathcal{F}^*(Q,\varrho)$ dual to $\delta_{[M]}$.
	
	\begin{proposition}\label{iHA by const function}
		Let $\widetilde{\mathcal{H}}(\Lambda^\imath)$ be the twisted Hall algebra of the $\imath$quiver algebra $\Lambda^\imath$. Then we have an algebra isomorphism 
		\[\Upsilon:\widetilde{\mathcal{H}}(\Lambda^\imath) \longrightarrow \mathcal{F}^*(Q,\varrho),\quad [M]\mapsto \sqq^{(\dimv M)\cdot(\dimv M)-\frac{1}{2}\langle M,M\rangle_Q}\xi_{[M]}.\]
	\end{proposition}
	\begin{proof}
		We first recall that $\widetilde{\mathcal{H}}(\Lambda^\imath)$ is defined by the following multiplication:
		\[[M]\ast[N]=\sqq^{\langle M,N\rangle_Q}\sum_{[L]}\frac{|\Ext^1_{\Lambda^\imath}(M,N)_L|}{|\Hom_{\Lambda^\imath}(M,N)|}[L].\]
		If we set $\bw_1=\dimv M$ and $\bw_2=\dimv N$, then
		\begin{align*}
			&\Upsilon([M])\star\Upsilon([N])
			\\
			&=\sqq^{\bw_1\cdot\bw_1-\frac{1}{2}\langle \bw_1,\bw_1\rangle_Q+\bw_2\cdot\bw_2-\frac{1}{2}\langle \bw_2,\bw_2\rangle_Q}\xi_{[M]}\star\xi_{[N]}\\
			&=\sqq^{\bw_1\cdot\bw_1-\frac{1}{2}\langle \bw_1,\bw_1\rangle_Q+\bw_2\cdot\bw_2-\frac{1}{2}\langle \bw_2,\bw_2\rangle_Q+\frac{1}{2}\langle\bw_1,\bw_2\rangle_{Q,a}+2\bw_1\cdot\bw_2}\sum_{[L]}\frac{|\Ext^1_{\Lambda^\imath}(M,N)_L|}{|\Hom_{\Lambda^\imath}(M,N)|}\xi_{[L]}\\
			&=\sqq^{(\bw_1+\bw_2)\cdot(\bw_1+\bw_2)-\frac{1}{2}\langle\bw_1+\bw_2,\bw_1+\bw_2\rangle_Q+\langle\bw_1,\bw_2\rangle_Q}\sum_{[L]}\frac{|\Ext^1_{\Lambda^\imath}(M,N)_L|}{|\Hom_{\Lambda^\imath}(M,N)|}\xi_{[L]}\\
			&=\sqq^{\langle\bw_1,\bw_2\rangle_Q}\sum_{[L]}\frac{|\Ext^1_{\Lambda^\imath}(M,N)_L|}{|\Hom_{\Lambda^\imath}(M,N)|}\Upsilon([L])=\Upsilon([M]\ast[N]).\qedhere
		\end{align*}
	\end{proof}
	
	\begin{remark}\label{rmk:F^*(Q)}
		Similar to $\cf^*(Q,\varrho)$, we can define $\cf^*(Q)$ (with the same twist) by considering representations of $Q$. Similar to Proposition \ref{iHA by const function}, see also \cite[\S9.3]{Lus90},
		one can prove that there is  
		an algebra isomorphism 
		\[\Upsilon^+:\widetilde{\mathcal{H}}(\bfk Q) \longrightarrow \mathcal{F}^*(Q),\quad [M]\mapsto \sqq^{\dimv M\cdot\dimv M-\frac{1}{2}\langle M,M\rangle_Q}\xi_{[M]}.\]
	\end{remark}
	
	Recall the notations $\cd_{sg}(\Lambda^\imath)$ and $\res_\BH$ in \S\ref{subsec:i-quiver}. By identifying $\mod(\Lambda^\imath)$ and $\mod(\cs^\imath)$, we can prove the following characterization on the stratification of $\rep(\bw,\mathcal{S}^\imath)$.
	
	\begin{lemma}\label{lem:strata M_0^reg char}
		Two $\mathcal{S}^\imath$-modules $M,N$ lie in the same stratum of $\rep(\bw,\mathcal{S}^\imath)$ if and only if $M\cong N$ in $\mathcal{D}_{sg}(\Lambda^\imath)$ and $\res_\BH(M)=\res_\BH(N)$. 
	\end{lemma}
	\begin{proof}
		First we assume that $M\cong N$ in $\mathcal{D}_{sg}(\Lambda^\imath)$ and $\res_{\BH}(M)\cong\res_{\BH}(N)$. Then we can find short exact sequences of $\Lambda^\imath$-modules
		$$0\longrightarrow K_1\longrightarrow A\longrightarrow M\longrightarrow0,\qquad 0\longrightarrow K_2\longrightarrow A\longrightarrow N\longrightarrow0$$
		where $K_1,K_2$ are projective $\Lambda^\imath$-modules. From Lemma~\ref{lem:K_LR exact if projdim}, these give us
		\begin{equation}\label{prop:strata M_0^reg char-1}
			\dimv K_{LR}(M)+\dimv K_{LR}(K_1)=\dimv K_{LR}(A)=\dimv K_{LR}(N)+\dimv K_{LR}(K_2).
		\end{equation}
		After restricting to $\mod(\BH)$ the two exact sequences split, so they give
		\[\res_{\BH}(M)\oplus\res_{\BH}(K_1)=\res_{\BH}(A)=\res_{\BH}(N)\oplus\res_{\BH}(K_2).\]
		Since we have already assumed $\res_{\BH}(M)\cong\res_{\BH}(N)$, this implies $\res_{\BH}(K_1)\cong\res_{\BH}(K_2)$, and in particular (see \cite[Corollary 4.3]{LW21b})
		\[\dimv K_{LR}(K_1)=\dimv K_{LR}(K_2).\]
		Combining with \eqref{prop:strata M_0^reg char-1} we then get $\dimv K_{LR}(M)=\dimv K_{LR}(N)$.
		
		Conversely, let $M$ be a $\Lambda^\imath$-module and $\dimv K_{LR}(M)=(\bv,\bw)$. By \cite[Corollary 3.21]{LW19}, we can find a unique $\bfk Q$-module $X$ such that $X\cong M$ in $\mathcal{D}_{sg}(\Lambda^\imath)$, so there are short exact sequences
		$$0\longrightarrow K_3\longrightarrow B\longrightarrow M\longrightarrow0,\qquad 0\longrightarrow K_4\longrightarrow B\longrightarrow X\longrightarrow0$$
		such that $K_3,K_4$ are projective. Then Lemma~\ref{lem:K_LR exact if projdim} implies
		\begin{equation}\label{prop:strata M_0^reg char-2}
			\dimv K_{LR}(M)+\dimv K_{LR}(K_3)=\dimv K_{LR}(K_4)+\dimv K_{LR}(X).
		\end{equation}
		If we write $(\bv,\bw)=(\bv^0,\bw^0)+(\bv^+,\bw^+)$ for the decomposition of $(\bv,\bw)$, then we find
		\[\dimv K_{LR}(X)=(\bv^+,\bw^+).\]
		In other words, the $\bfk Q$-module $X$ is uniquely determined by $(\bv,\bw)$ (here Lemma~\ref{lem:unique V^+W^+} is applied). Now note that
		\[(\bv,\bw)=\dimv K_{LR}(M)=\dimv K_{LR}(X)+\dimv K_{LR}(K_4)-\dimv K_{LR}(K_3).\]
		Since $\dimv K_{LR}(X)=(\bv^+,\bw^+)$, the uniqueness of the decomposition implies
		\[\dimv K_{LR}(K_4)-\dimv K_{LR}(K_3)=(\bv^0,\bw^0)\in({^\imath V^0},{^\imath W^0}).\]
		In particular, we can find $\alpha\in\N^\I$ such that
		\[\dimv K_{LR}(\K_\alpha)=\dimv K_{LR}(K_4)-\dimv K_{LR}(K_3).\]
		Now by restricting to $\mod(\BH)$, the two short exact sequences split and give an isomorphism
		\[\res_{\BH}(M)\oplus\res_{\BH}(K_3)\cong\res_{\BH}(X)\oplus\res_{\BH}(K_4).\]
		From Krull-Schmidt theorem we then find $\res_{\BH}(M)\cong\res_{\BH}(X)\oplus\res_{\BH}(\K_\alpha)$. We have therefore proved that for any $\mathcal{S}^\imath$-module $M$, there are $\alpha\in\N^\I$ and a unique $\bfk Q$-module $X$ such that $X\oplus\K_\alpha\cong M$ in $\mathcal{D}_{sg}(\Lambda^\imath)$ and $\res_{\BH}(M)=\res_{\BH}(X\oplus\K_\alpha)$. Since $X$ and $\alpha$ are uniquely determined by $\dimv K_{LR}(M)$, this completes the proof of the lemma.
	\end{proof}
	
	For each strongly $l$-dominant pair $(\bv,\bw)$ for $\mathcal{R}^\imath$, we denote by $\delta_{(\bv,\bw)}$ the constant function on $\mathcal{M}_0^{\reg}(\mathbf{v},\mathbf{w},\mathcal{R}^\imath)$, and by $\mathcal{M}_\mathbf{w}$ to be the subspace of $\mathcal{F}_\mathbf{w}$ spanned by the $\delta_{(\bv,\bw)}$. From Lemma~\ref{lem:strata M_0^reg char} and \cite[Lemma 3.7]{LP25}, we see that $\res^{\bw}_{\bw_1,\bw_2}$ can be restricted to 
	\[\res^{\bw}_{\bw_1,\bw_2}:\mathcal{M}_{\bw}\longrightarrow \mathcal{M}_{\bw_1}\times\mathcal{M}_{\bw_2}.\]
	Let $\mathcal{M}^*(Q,\varrho)=\bigoplus_\bw\mathcal{M}_\bw^*$. Then $\res^{\bw}_{\bw_1,\bw_2}$ endows $\mathcal{M}^*(Q,\varrho)$ with an algebra structure, and there is a homomorphism
	\begin{equation}\label{dual trace map diagram}
		\pi:\mathcal{F}^*(Q,\varrho)\longrightarrow \mathcal{M}^*(Q,\varrho)
	\end{equation}
	If we denote by $\xi_{(\mathbf{v},\mathbf{w})}$ the basis dual to $\delta_{(\bv,\bw)}$, then $\pi(\xi_{[M]})=\xi_{(\mathbf{v},\mathbf{w})}$ if $\dimv K_{LR}(M)=(\bv,\bw)$. The following lemma is now immediate.
	
	\begin{lemma}\label{iHA isomorphic to M^*}
		The isomorphism $\Upsilon$ induces an isomorphism 
		\[\ov{\Upsilon}:\widehat{\mathcal{H}}(Q,\varrho)\longrightarrow\mathcal{M}^*(Q,\varrho)\]
		and there is a commutative diagram
		\[\begin{tikzcd}
			\mathcal{F}^*(Q,\varrho)\ar[r,two heads,"\pi"] & \mathcal{M}^*(Q,\varrho)\\
			\widetilde{\mathcal{H}}(\Lambda^\imath)\ar[r]\ar[u,"\Upsilon"]&\widehat{\mathcal{H}}(Q,\varrho)\ar[u,swap,"\ov{\Upsilon}"]
		\end{tikzcd}\]
	\end{lemma}
	
	The $\imath$Hall algebra $\widetilde{\ch}(\bfk Q,\varrho)$ is the localization of $\widehat{\mathcal{H}}(Q,\varrho)$ with respect to $[K]$, with $K$ of finite projective dimension. From Lemma~\ref{iHA isomorphic to M^*}, we see the \cite[Lemma A.4]{LW19} is also valid for $\mathcal{M}^*(Q,\varrho)$, so we can form the localization of $\mathcal{M}^*(Q,\varrho)$ with respect to $\xi_{(\bv^0,\bw^0)}$, $(\bv^0,\bw^0)\in ({^\imath V^0},{^\imath W^0})$, which we denote by $\widetilde{\mathcal{M}}^*(Q,\varrho)$. We then obtain an induced isomorphism
	\begin{equation}\label{iHA to const function}
		\widetilde{\Upsilon}:\widetilde{\ch}(\bfk Q,\varrho) \longrightarrow \widetilde{\mathcal{M}}^*(Q,\varrho).
	\end{equation}

	
	Inspired by Lusztig's construction \cite[\S9.4]{Lus90}, we now introduce the generic version of $\widetilde{\mathcal{M}}^*(Q,\varrho)$. Let $\mathbf{A}$ be the group ring of units in $\ov{\Q}_\ell$. For $\bw\in\N^{\mathcal{S}^\imath_0}$, we define $\mathbf{M}_\bw$ to be the free $\mathbf{A}$-module generated by symbols $\tilde{\delta}_{(\bv,\bw)}$,  
	with $(\bv,\bw)$ being strongly $l$-dominant pairs. Inside $\mathbf{M}_\bw$ there are naturally defined elements
	\begin{equation}\label{eq:formal trace on K group}
		\mathcal{L}_{\bv,\bw}:=\sum_{\bv'}p_{\bv',\bv}^{\bw}\tilde{\delta}_{(\bv',\bw)}
	\end{equation}
	where $p_{\bv',\bv}$ is the (formal) alternating sums of eigenvalues (in $\ov{\Q}_\ell^*$) of Frobenius map on the stalks of the cohomology sheaves of $\mathcal{L}(\bv,\bw)$ at any $\F_q$-rational point of $\mathcal{M}_0^{\reg}(\bv,\bw,\mathcal{R}^\imath)$ (the sum is taken in $\mathbf{A}$ rather than $\ov{\Q}_\ell$). Note that $p_{\bv,\bv}^{\bw}=1$, and $p_{\bv',\bv}^{\bw}=0$ whenever the condition $\bv'\leq\bv$ is not satisfied. In particular, the $\mathcal{L}_{\bv,\bw}$ form a basis for $\mathbf{M}_\bw$.
	
	For dimension vectors $\bw_1,\bw_2$ and $\bw=\bw_1+\bw_2$, recall that we have defined a function $\res^{\bw}_{\bw_1,\bw_2}\delta_{(\bv,\bw)}$ in $\mathcal{M}_{\bw_1}\times\mathcal{M}_{\bw_2}$. By the Hall polynomial property of $\widetilde{\mathcal{H}}(Q,\varrho)$, there exist polynomials $f_{\bv_1,\bv_2}^{\bv}(v)\in\Z[v,v^{-1}]$ such that  
	\[\res^{\bw}_{\bw_1,\bw_2}\delta_{(\bv,\bw)}=\sum_{\bv_1,\bv_2}f_{\bv_1,\bv_2}^{\bv}(\sqq)\delta_{(\bv_1,\bw_1)}\otimes\delta_{(\bv_2,\bw_2)}.\]
	We now define an $\mathbf{A}$-bilinear map
	\[\res^{\bw}_{\bw_1,\bw_2}:\mathbf{M}_\bw\longrightarrow \mathbf{M}_{\bw_1}\times\mathbf{M}_{\bw_2},\quad \tilde{\delta}_{(\bv,\bw)}\mapsto\sum_{\bv_1,\bv_2}f_{\bv_1,\bv_2}^{\bv}(\sqq)\tilde{\delta}_{(\bv_1,\bw_1)}\otimes\tilde{\delta}_{(\bv_2,\bw_2)}\]
	where the coefficients $f_{\bv_1,\bv_2}^{\bv}(\sqq^\frac{1}{2})$ are now considered as elements in $\mathbf{A}$. Let $\mathbf{M}^*(Q,\varrho)=\bigoplus_\bw\mathbf{M}_\bw^*$ and denote by $\tilde{\xi}_{(\bv,\bw)}$ the basis dual to $\tilde{\delta}_{(\bv,\bw)}$. Then it is clear that the restriction map endows $\mathbf{M}^*(Q,\varrho)$ with an algebra structure, and we can form the localization $\widetilde{\mathbf{M}}(Q,\varrho)$ with respect to $\tilde{\xi}_{(\bv^i,\bw^i)}$ ($i\in\I$), as in \S\ref{subsec:QV function iHA}.
	
	There is a natural group homomorphism $\mathbf{M}_\bw\rightarrow\mathcal{M}_\bw$ sending $\tilde{\delta}_{(\bv,\bw)}$ to $\delta_{(\bv,\bw)}$ and on coefficients is given by the ring homomorphism $\mathbf{A}\rightarrow\ov{\Q}_\ell$ that is the identity on $\ov{\Q}_\ell^*$. It follows easily from the definition that the induced map $\widetilde{\mathbf{M}}(Q,\varrho)\rightarrow \widetilde{\mathcal{M}}(Q,\varrho)$ is an algebra homomorphism.
	
	If we replace $q$ by a power $q^s$, then $\widetilde{\mathbf{M}}(Q,\varrho)$ and $\widetilde{\mathcal{M}}(Q,\varrho)$ become $\widetilde{\mathbf{M}}(Q,\varrho)_s$ and $\widetilde{\mathcal{M}}(Q,\varrho)_s$ and we have as above a ring homomorphism $\widetilde{\mathbf{M}}(Q,\varrho)_s\rightarrow \widetilde{\mathcal{M}}(Q,\varrho)_s$. On the other hand, there is a ring homomorphism $\widetilde{\mathbf{M}}(Q,\varrho)\rightarrow\widetilde{\mathbf{M}}(Q,\varrho)_s$ that is the identity on each basis element and on coefficients is given by the ring homomorphism $\mathbf{A}\rightarrow\mathbf{A}$ defined by raising to the $s$-th power on $\ov{\Q}_\ell^*$. 
	Composing this we then obtain a ring homomorphism $\widetilde{\mathbf{M}}(Q,\varrho)\rightarrow \widetilde{\mathcal{M}}(Q,\varrho)_s$. Thus we have infinitely many homomorphisms of $\widetilde{\mathbf{M}}(Q,\varrho)$ to the various $\widetilde{\mathcal{M}}(Q,\varrho)_s$. Using these, one can deduce various properties of $\widetilde{\mathbf{M}}(Q,\varrho)$ from the corresponding properties of $\widetilde{\mathcal{M}}(Q,\varrho)_s$. For example, from \eqref{iHA to const function} we obtain the following.
	
	\begin{lemma}\label{lem:iHA to bfM}
		There is an isomorphism of algebras 
		\[\widetilde{\Upsilon}:\widetilde{\ch}(Q,\varrho)_\mathcal{Z}\otimes_{\mathcal{Z}}\mathbf{A}\longrightarrow \widetilde{\mathbf{M}}^*(Q,\varrho)\]
		defined as in Proposition~\ref{iHA by const function}. Here we view $\mathcal{Z}$ as a subring of $\mathbf{A}$ by fixing a square root of $q$ in $\ov{\Q}_\ell^*$).
	\end{lemma}
	
	For any strongly $l$-dominant pair $(\bv,\bw)$, we define elements $H_{\bv,\bw}$ and $L_{\bv,\bw}$ in $\widetilde{\mathbf{M}}^*(Q,\varrho)$ such that 
	\[H_{\bv,\bw}=\sqq^{\dim\mathcal{M}_0^{\reg}(\bv,\bw,\mathcal{R}^\imath)}\tilde{\xi}_{(\bv,\bw)}\]
	and $L_{\bv,\bw}$ is the dual basis to $\mathcal{L}_{\bv,\bw}$ defined in \eqref{eq:formal trace on K group}. If we write
	\begin{equation}\label{eq:M_bv into L_vw}
		H_{\bv,\bw}=\sum_{\bv'}h_{\bv,\bv'}^{\bw}L_{\bv',\bw}
	\end{equation}
	then $h_{\bv,\bv'}^{\bw}=\sqq^{\dim\mathcal{M}_0^{\reg}(\bv,\bw,\mathcal{R}^\imath)}p_{\bv,\bv'}^{\bw}$ and $h_{\bv,\bv'}^{\bw}=0$ if $\bv\leq\bv'$ is not satisfied.
	
	Note that Verdier duality induces an anti-automorphism $\bar{\cdot}:\widetilde{\mathbf{M}}^*(Q,\varrho)\rightarrow\widetilde{\mathbf{M}}^*(Q,\varrho)$ that is anti-linear with respect to the ring homomorphism $\bar{\cdot}:\mathbf{A}\rightarrow\mathbf{A}$ takes each element in $\ov{\Q}_\ell^*$ to its inverse. Moreover, we find that
	\begin{equation}\label{eq:Upsilon compatible with bar}
		\widetilde{\Upsilon}(\ov{u})=\ov{\widetilde{\Upsilon}(u)},\quad \forall u\in\widetilde{\ch}(Q,\varrho)\otimes_\mathcal{Z}\mathbf{A}.
	\end{equation}
	In fact, it suffices to check this on generators.
	
	From the property of IC sheaves we know that each $L_{\bv,\bw}$ is invariant under bar involution. Moreover, by \eqref{eq:M_bv into L_vw} we have
	\begin{equation}\label{eq:bar of M_vw}
		\ov{H_{\bv,\bw}}=\sum_{\bv\leq\bv'}b_{\bv,\bv'}^{\bw}H_{\bv',\bw}
	\end{equation}
	where $b_{\bv,\bv'}^\bw\in\mathbf{A}$ and $b_{\bv,\bv}^\bw=1$.
	
	\begin{lemma}\label{lem:Hall basis map to M_vw}
		For any $\alpha\in\Z^\I$ and $\lambda\in\fp$, we have
		\begin{equation}
			\widetilde{\Upsilon}(\K_\alpha\diamond \mathfrak{U}_\lambda)=H_{\bv_\alpha+\bv_\lambda,\bw_\alpha+\bw_\lambda},
		\end{equation}\label{lem:Hall basis map to M_vw-1}
		where $\widetilde{\Upsilon}(\K_\alpha)=H_{\bv_\alpha,\bw_\alpha}$ and $(\bv_\lambda,\bw_\lambda)$ is such that $\sigma^*\bw_\lambda-\cc_q\bv_\lambda=\lambda$. Moreover, 
		\begin{equation}\label{lem:Hall basis map to M_vw-2}
			H_{\bv_\alpha,\bw_\alpha}\cdot H_{\bv_\lambda,\bw_\lambda}\in\sqq^{\frac{1}{2}\Z} H_{\bv_\alpha+\bv_\lambda,\bw_\alpha+\bw_\lambda}.
		\end{equation}
	\end{lemma}
	\begin{proof}
		From the definition of $\widetilde{\Upsilon}$ and $\mathfrak{U}_\lambda$, it is easy to see that 
		\begin{align*}
			\widetilde{\Upsilon}(\mathfrak{U}_\lambda)&=\sqq^{-\dim\End_{\bfk Q}(M_q(\lambda))+\frac{1}{2}\langle M_q(\lambda),M_q(\lambda)\rangle_Q}\widetilde{\Upsilon}(\fu_\lambda)
			\\
			&=\sqq^{-\dimv M_q(\lambda)\cdot \dimv M_q(\lambda)+\frac{1}{2}\langle M_q(\lambda),M_q(\lambda)\rangle_Q+\dim\mathcal{M}_0^{\reg}(\bv_\lambda,\bw_\lambda,\mathcal{R}^\imath)}\widetilde{\Upsilon}(\fu_\lambda)\\
			&=\sqq^{\dim\mathcal{M}_0^{\reg}(\bv_\lambda,\bw_\lambda,\mathcal{R}^\imath)}\tilde{\xi}_{(\bv^\lambda,\bw^\lambda)}=H_{\bv_\lambda,\bw_\lambda}.
		\end{align*}
		We also note that $\widetilde{\Upsilon}$ sends $\K_i$ to $H_{\bv^{\varrho i},\bw^i}=L_{\bv^{\varrho i},\bw^i}$, so for each $\alpha\in\Z^\I$ one can find a unique pair $(\bv_\alpha,\bw_\alpha)$ such that $\widetilde{\Upsilon}(\K_\alpha)=H_{\bv_\alpha,\bw_\alpha}$. Then from \cite[Lemma A8]{LW19} we see that
		\[\widetilde{\Upsilon}(\K_\alpha\diamond \mathfrak{U}_\lambda)=\sqq^{n}H_{\bv_\alpha+\bv_\lambda,\bw_\alpha+\bw_\lambda}\]
		for some $n\in\Z$. To see that $n=0$, note that in the $\imath$Hall algebra $\widetilde{\ch}(Q,\varrho)$ we have 
		\begin{equation}\label{lem:Hall basis map to M_vw-1}
			\ov{\K_\alpha\diamond \mathfrak{U}_\lambda}\in \K_\alpha\diamond \mathfrak{U}_\lambda+\sum\Z[v,v^{-1}]\K_\beta\diamond \mathfrak{U}_\mu.
		\end{equation}
		Combining this with \eqref{eq:Upsilon compatible with bar} and \eqref{eq:bar of M_vw}, we obtain the claim.
	\end{proof}
	
	From the proof of Lemma~\ref{lem:Hall basis map to M_vw} and \eqref{lem:Hall basis map to M_vw-1}, we know that each $b_{\bv,\bv'}^{\bw}$ belongs to $\Z[\sqq,\sqq^{-1}]$. We can therefore apply Lusztig's Lemma \cite[Theorem 1.1]{BZ14} to show the following.
	
	\begin{lemma}\label{lem:unique zeta_vv}
		There exists a unique family of elements $\zeta_{\bv,\bv'}^{\bw}$ such that $\zeta_{\bv,\bv}^{\bw}=1$, $\zeta_{\bv,\bv'}^{\bw}\in\sqq^{-1}\Z[\sqq^{-1}]$ if $\bv\neq\bv'$, and
		\begin{equation}\label{lem:unique zeta_vv-1}
			\sum_{\bv\leq\bv'\leq\bv''}b_{\bv,\bv'}^{\bw}\zeta_{\bv',\bv''}^{\bw}=\ov{\zeta_{\bv,\bv''}^{\bw}}.
		\end{equation}
	\end{lemma}
	\begin{proof}
		The result of \cite[Theorem 1.1]{BZ14} (see also \cite[eq. (7.5)]{BZ14}) provides us a unique family of elements $\eta_{\bv,\bv'}^{\bw}$ such that $\eta_{\bv,\bv}^{\bw}=1$, $\eta_{\bv,\bv'}^{\bw}\in\sqq^{-1}\Z[\sqq^{-1}]$ if $\bv\neq\bv'$, and
		\[ \sum_{\bv\leq\bv'\leq\bv''}\ov{\eta_{\bv,\bv'}^{\bw}}b_{\bv',\bv''}^{\bw}=\eta_{\bv,\bv''}^{\bw}.\]
		It then suffices to take $(\zeta_{\bv,\bv'}^{\bw})_{\bv,\bv'}$ to be the inverse matrix of $(\eta_{\bv,\bv'}^{\bw})_{\bv,\bv'}$.
	\end{proof}
	
	We now prove that the coefficient $h_{\bv,\bv'}^{\bw}$ is uniquely determined in this way.
	
	\begin{theorem}\label{thm:IC coefficient is zeta_vv}
		For each $\bv\leq\bv'$, we have $h_{\bv,\bv'}^{\bw}=\zeta_{\bv,\bv'}^{\bw}$. In particular, $\zeta_{\bv,\bv'}^{\bw}\in\sqq^{-1}\Z[\sqq^{-1}]$ if $\bv\neq\bv'$.
	\end{theorem}
	\begin{proof}
		From the fact that $L_{\bv,\bw}$ is bar-invariant and \eqref{eq:M_bv into L_vw}, we find that
		\begin{equation}\label{thm:IC coefficient is zeta_vv-1}
			\ov{h_{\bv,\bv''}^{\bw}}=\sum_{\bv\leq\bv'\leq\bv''}b_{\bv,\bv'}^{\bw}h_{\bv',\bv''}^{\bw}.
		\end{equation}
		Substracting this from \eqref{lem:unique zeta_vv-1}, we obtain
		\begin{equation}\label{thm:IC coefficient is zeta_vv-2}
			\ov{h_{\bv,\bv''}^{\bw}}-\ov{\zeta_{\bv,\bv''}^{\bw}} = \sum_{\bv\leq\bv'\leq\bv''}b_{\bv,\bv'}^{\bw}(h_{\bv',\bv''}^{\bw}-\zeta_{\bv',\bv''}^{\bw}).
		\end{equation}
		Let $\mu_{\bv,\bv''}^{\bw}=h_{\bv,\bv''}^{\bw}-\zeta_{\bv,\bv''}^{\bw}$. Then we can deduce from \eqref{thm:IC coefficient is zeta_vv-2} that
		\begin{equation}\label{thm:IC coefficient is zeta_vv-3}
			\ov{\mu_{\bv,\bv''}^{\bw}}-\mu_{\bv,\bv''}^{\bw} = \sum_{\bv<\bv'\leq\bv''}b_{\bv,\bv'}^{\bw}\mu_{\bv',\bv''}^{\bw}.
		\end{equation}
		We shall prove that $\mu_{\bv,\bv''}^{\bw}=0$ for all $\bv\leq\bv''$ by induction on the maximal length of a chain between $\bv$ and $\bv''$. If $\bv=\bv''$ then this is clear as $h_{\bv,\bv}^{\bw}=\zeta_{\bv,\bv}^{\bw}=1$, so assume $\bv<\bv''$. We can assume that the statement is already known for all $\bv<\bv'\leq\bv''$. Then the right-hand side of \eqref{thm:IC coefficient is zeta_vv-3} is zero and therefore its left-hand side is also zero:
		\begin{equation}\label{thm:IC coefficient is zeta_vv-4}
			\ov{\mu_{\bv,\bv''}^{\bw}}=\mu_{\bv,\bv''}^{\bw}.
		\end{equation}
		By Gabber's purity theorem \cite[5.3.4]{BBD}, each $\mathcal{L}(\bv',\bw)$ is a pure complex on $\rep(\bw,\mathcal{S}^\imath)$. This means the eigenvalues of Frobenius map on the stalks of the $i$-th cohomology sheaves of $\mathcal{L}(\bv',\bw)$ at any $\F_q$-rational point of $\rep(\bw,\mathcal{S}^\imath)$ are algebraic numbers all of whose complex conjugates have absolute value $\leq q^{-i/2}$. On the other hand, by the property of IC sheaves \cite[2.1.9]{BBD}, the stalk of $\mathcal{L}(\bv',\bw)$ at any $\F_q$-rational point of $\mathcal{M}_0^{\reg}(\bv,\bw,\mathcal{R}^\imath)$ is a complex concentrated at degrees $<-\dim\mathcal{M}_0^{\reg}(\bv,\bw,\mathcal{R}^\imath)$, so $p_{\bv,\bv''}^{\bw}$ is a (formal) $\Z$-linear combination of elements of $\ov{\Q}_\ell^*$ that are algebraic numbers all of whose complex conjugates have absolute value $\leq q^{-(\dim\mathcal{M}_0^{\reg}(\bv,\bw,\mathcal{R}^\imath)+1)/2}$. Using this, together with Lemma~\ref{lem:unique zeta_vv} and the definition of $h_{\bv,\bv'}^{\bw}$, we see that $\mu_{\bv,\bv'}^{\bw}$ is a (formal) $\Z$-linear combination of elements of $\ov{\Q}_\ell^*$ that are algebraic numbers all of whose complex conjugates have absolute value $\leq q^{-1/2}$. This is compatible with \eqref{thm:IC coefficient is zeta_vv-4} only if both sides are zero. Thus our assertion is proved, and the theorem follows.
	\end{proof}
	
	\begin{corollary}\label{coro:L_vw unique determined}
		For any strongly $l$-dominant pair $(\bv,\bw)$, $L_{\bv,\bw}$ is the unique element in $\widetilde{\mathbf{M}}^*(Q,\varrho)$ such that $\ov{L_{\bv,\bw}}=L_{\bv,\bw}$ and 
		\[L_{\bv,\bw}\in H_{\bv,\bw}+\sum_{\bv<\bv'}\sqq^{-1}\Z[\sqq^{-1}]H_{\bv',\bw}.\]
	\end{corollary}
	\begin{proof}
		This follows from Theorem \ref{thm:IC coefficient is zeta_vv} and Lemma~\ref{lem:unique zeta_vv}, by taking the inverse matrix of $(h_{\bv,\bv'}^{\bw})_{\bv,\bv'}$.
	\end{proof}
	
	\begin{corollary}\label{coro:Lv^0 times Lv^+ prop}
		For any strongly $l$-dominant pairs $(\bv,\bw)$ and $(\bv^0,\bw^0)\in ({^\imath V^0},{^\imath W^0})$ such that $\sigma^*\bw^0-\mathcal{C}_q\bv^0=0$, we have 
		\[L_{\bv^0,\bw^0}\cdot L_{\bv,\bw}=\ov{L_{\bv,\bw}\cdot L_{\bv^0,\bw^0}}\in\sqq^{\frac{1}{2}\Z}L_{\bv^0+\bv,\bw^0+\bw}.\]
	\end{corollary}
	
	\begin{proof}
		From Corollary~\ref{coro:L_vw unique determined} we know that $L_{\bv,\bw}$ is uniquely determined by $\ov{L_{\bv,\bw}}=L_{\bv,\bw}$ and the following equation:
		\begin{equation}
			\label{eq:L-M}
			L_{\bv,\bw}\in H_{\bv,\bw}+\sum_{\bv<\bv'}\sqq^{-1}\Z[\sqq^{-1}]H_{\bv',\bw}.
		\end{equation}
		Then we also have
		\begin{equation}
			\label{eq:M-L}
			H_{\bv,\bw}\in L_{\bv,\bw}+\sum_{\bv<\bv'}\sqq^{-1}\Z[\sqq^{-1}]L_{\bv',\bw}.
		\end{equation}
		By induction and Lemma \ref{lem:R0 subalgebra} we may assume that $(\bv^0,\bw^0)=(\bv^i,\bw^i)$ for some $i\in Q_0$. In this case we have $L_{\bv^i,\bw^i}=H_{\bv^i,\bw^i}$, and from \cite[Lemma 3.5]{LP25}, \eqref{eq:L-M}, \eqref{eq:M-L} and Lemma \ref{lem:Hall basis map to M_vw} we find that 
		\begin{align*}
			L_{\bv^i,\bw^i}\cdot L_{\bv,\bw}&=\sum_{\bv\leq\tilde{\bv}}f_{\bv,\tilde{\bv}}^{\bw}H_{\bv^i,\bw^i}\cdot H_{\tilde{\bv},\bw}=\sqq^n\sum_{\bv\leq\tilde{\bv}}f_{\bv,\tilde{\bv}}^{\bw}H_{\bv^i+\tilde{\bv},\bw^i+\bw}=\sqq^n\sum_{\bv+\bv^i\leq\tilde{\bv}}g_{\bv,\tilde{\bv}}^{\bw}L_{\tilde{\bv},\bw^i+\bw},
		\end{align*}
		and similarly, 
		\begin{align*}
			L_{\bv,\bw}\cdot L_{\bv^i,\bw^i}&=\sqq^{-n}\sum_{\bv+\bv^i\leq\tilde{\bv}}h_{\bv,\tilde{\bv}}^{\bw}L_{\tilde{\bv},\bw^i+\bw},
		\end{align*}
		where $n$ is a half integer and $g_{\bv,\bv+\bv^i}^{\bw}=1=h_{\bv,\bv+\bv^i}^{\bw}$, $g_{\bv,\tilde{\bv}}^{\bw},h_{\bv,\tilde{\bv}}^{\bw}\in \sqq^{-1}\Z[\sqq^{-1}]$ if $\bv\neq\tilde{\bv}$. Since $\ov{L_{\bv,\bw}\cdot L_{\bv^i,\bw^i}}=L_{\bv^i,\bw^i}\cdot L_{\bv,\bw}$, we know that $g_{\bv,\tilde{\bv}}^{\bw}=0=h_{\bv,\tilde{\bv}}^{\bw}$ if $\bv\neq\tilde{\bv}$. Then 
		\begin{equation*}
			L_{\bv^i,\bw^i}\cdot L_{\bv,\bw}=\sqq^{n}L_{\bv^i+\bv,\bw^i+\bw},\quad L_{\bv,\bw}\cdot L_{\bv^i,\bw^i}=\sqq^{-n}L_{\bv^i+\bv,\bw^i+\bw}.\qedhere
		\end{equation*}
	\end{proof}

	\begin{corollary}\label{coro:Upsilon maps dCB}
		The isomorphism $\widetilde{\Upsilon}$ sends the dual canonical basis of $\widetilde{\ch}(Q,\varrho)_\mathcal{Z}\otimes_{\mathcal{Z}}\mathbf{A}$ to the basis $\{L_{\bv,\bw}\mid \text{$(\bv,\bw)$ is strongly $l$-dominant}\}$ of $\widetilde{\mathbf{M}}^*(Q,\varrho)$.
	\end{corollary}
	\begin{proof}
		From Theorem~\ref{thm:IC coefficient is zeta_vv} and Corollary~\ref{coro:L_vw unique determined} we see that $\widetilde{\Upsilon}(\mathfrak{L}_\lambda)=L_{\bv^\lambda,\bw^\lambda}$. Then Corollary~\ref{coro:Lv^0 times Lv^+ prop} and Lemma~\ref{lem:Hall basis map to M_vw} imply $\widetilde{\Upsilon}(\K_\alpha\diamond \mathfrak{L}_\lambda)=\widetilde{\Upsilon}(\mathfrak{L}_{\alpha,\lambda})=L_{\bv_\alpha+\bv_\lambda,\bv_\alpha+\bv_\lambda}$, whence the claim.
	\end{proof}

	\subsection{Comparison of dual canonical bases}\label{subsec:dCB compare}
	
	To compare the dual canonical basis of $\widetilde{\mathbf{M}}^*(Q,\varrho)$ with that of $\tRi$, we need to give an isomorphism from $\widetilde{\mathbf{M}}^*(Q,\varrho)$ to $\tRi\otimes_\mathcal{Z}\mathbf{A}$, with $L_{\bv,\bw}$ sent to $L(\bv,\bw)$. The problem is that transversal slice theorem does not hold over positive characteristic, so we can not define the analogue of $\tRi$ over $\F$. To overcome this we shift to the subvariety of Gorenstein projective $\mathcal{S}^\imath$-modules, as in \S\ref{subsec:Gproj reduction}. 
	
	Consider the subvariety $\rep^{\Gp}(\bw,\mathcal{S}^\imath)$ of $\rep(\bw,\mathcal{S}^\imath)$ defined over $\F$, equipped with an $\F_q$-structure so that the natural inclusion $i:\rep^{\Gp}(\bw,\mathcal{S}^\imath)\rightarrow \rep(\bw,\mathcal{S}^\imath)$ admits an $\F_q$-structure. Let $\mathcal{F}_\bw^{\Gp}$ be the vector space of $\ov{\Q}_\ell$-valued functions on the set of $\F_q$-rational points of $\rep^{\Gp}(\bw,\mathcal{S}^\imath)$ that are constant on the set of $\F_q$-rational points of any $G_\bw$-orbits on $\rep^{\Gp}(\bw,\mathcal{S}^\imath)$. The function $\res^{\bw}_{\bw_1,\bw_2}$ can be restricted to $\mathcal{F}_\bw^{\Gp}$, and we obtain an algebra structure on $\mathcal{F}^{\Gp,*}(Q,\varrho):=\bigoplus_\bw\mathcal{F}_\bw^{\Gp,*}$. 
	
	Let $\widetilde{\ch}(\Gproj(\Lambda^\imath))$ be the twisted Hall algebra of $\Gproj(\Lambda^\imath)$, with its multiplication twisted by the Euler form of $\mod(\bfk Q)$. Note that $\widetilde{\ch}(\Gproj(\Lambda^\imath))$ can be viewed as the subalgebra of $\widetilde{\ch}(\iLa)$. 
	From the proof of Proposition~\ref{iHA by const function} we get the following.

	\begin{proposition}
		There is an algebra isomorphism
		\[\Upsilon:\widetilde{\ch}(\Gproj(\Lambda^\imath))\longrightarrow \mathcal{F}^{\Gp,*}(Q,\varrho),\quad [M]\mapsto \sqq^{\dimv M\cdot\dimv M-\frac{1}{2}\langle M,M\rangle_Q}\tilde{\xi}_{[M]}.\]
	\end{proposition}
	
	We denote by $\mathcal{SD}\widetilde{\ch}(\Gproj(\Lambda^\imath))$ the localization of $\widetilde{\ch}(\Gproj(\Lambda^\imath))$ with respect to projective $\Lambda^\imath$-modules. Then it is known \cite[Theorem A15]{LW19} that the natural inclusion $\widetilde{\ch}(\Gproj(\Lambda^\imath))\hookrightarrow\widetilde{\ch}(\Lambda^\imath)$ induces an isomorphism 
	\[\mathcal{SD}\widetilde{\ch}(\Gproj(\Lambda^\imath))\stackrel{\sim}{\longrightarrow} \widetilde{\ch}(\bfk Q,\varrho).\]
	
	The advantage of shifting from $\rep(\bw,\mathcal{S}^\imath)$ to $\rep^{\Gp}(\bw,\mathcal{S}^\imath)$ is that the stratification of $\rep^{\Gp}(\bw,\mathcal{S}^\imath)$ is given by orbits (see Lemma \ref{lem:Gproj unique strata}), and any $G_\bw$-equivariant local system on an orbit in $\rep^{\Gp}(\bw,\mathcal{S}^\imath)$ is trivial. 
	
	We now recall the definition of sheaf-function correspondence. Let $X$ be an $\F$-variety together with a connected algebraic group $G$-action. Suppose $X,G$ have $\F_q$-structures, we denote by $\Fr$ their Frobenius maps. Let $D^b_{G_\bw}(X)$ denote the category consisting of objects of the form $(L,\varphi)$, where $L$ is a $G_\bw$-equivariant complex and $\varphi:\Fr^*L\stackrel{\cong}{\rightarrow} L$ is an isomorphism. For any $\F_q$-point $x\in \rep(\mathbf{w},\mathcal{S}^\imath)$, the isomorphism $\varphi$ induces an isomorphism $\varphi_x:L_x\rightarrow L_x$. Taking the alternating sum of the traces of the isomorphisms on cohomology groups, we obtain
	\[\chi_L(x)=\sum_{i\in\Z}(-1)^i\mathrm{tr}(H^i(\varphi_x):H^i(L_x)\rightarrow H^i(L_x)).\]
	In this way, we obtain a $G(\F_q)$-invariant function on $X(\F_q)$. We denote by $L\mapsto L(1)$ the Tate twist, and fix a square root $L\mapsto L(\frac{1}{2})$. The following
	theorem is called the sheaf-function correspondence, see \cite[III. Theorem 12.1]{RW01} or \cite[Section 5.3]{Ach21}.
	
	\begin{proposition}
		\label{prop:sheaf-function}
		Let $X$, $Y$ be $\F$-varieties together with connected algebraic group $G$-actions, suppose they are defined over $\F_q$, and $f:X\rightarrow Y$ is a $G$-equivariant morphism which is compatible with $\F_q$-structures.
		\begin{enumerate}
			\item[(a)] For any distinguished triangle $L'\rightarrow L\rightarrow L''\rightarrow L'[1]$ in $D^b_G(X)$, we have
			\[\chi_L=\chi_{L'}+\chi_{L''}.\]
			\item[(b)] For any $L\in D^b_G(X)$ and $L'\in D^b_G(Y)$, we have
			\begin{gather*}
				\chi_{L[n]}=(-1)^n\chi_L,\quad \chi_{L(n)}=q^{-n}\chi_L,\quad \chi_{L\boxtimes L'}=\chi_L\otimes\chi_{L'},\\
				\chi_{f_!L}=f_!\chi_L,\quad \chi_{f^*L}=f^*\chi_L.
			\end{gather*}
		\end{enumerate}
	\end{proposition}
	
	For each Gorenstein projective $\mathcal{S}^\imath$-module $M$, we denote by $\mathcal{L}(M)$ the intersection cohomology complex associated to the constant local system on the orbit $\mathfrak{O}_M$. Let $K_\bw(\Gproj(\mathcal{S}^\imath))$ be the Grothendieck group of shifts of perverse sheaves $\mathcal{L}(M)$, with $\mathfrak{O}_M$ runs through orbits in $\rep^{\Gp}(\bw,\mathcal{S}^\imath)$. The $\mathcal{Z}$-module structure is given by $v^{-d}L:=L[d](\frac{d}{2})$. We can define the quantum Grothendieck ring for $\rep^{\Gp}(\bw,\mathcal{S}^\imath)$ as in the characteristic zero case. Using the general machinery in \cite[\S 6]{BBD}, we know that this ring is isomorphic to $\tR^{\imath,\Gp}$. 
	
	For each $M\in\Gproj(\Lambda^\imath)$ of dimension $\bw$, we denote by $\mathfrak{O}_M$ the orbit of $M$ in $\rep^{\Gp}(\bw,\mathcal{S}^\imath)$. Define a partial order by $N\prec M$ if $\mathfrak{O}_N$ is contained in the closure of $\mathfrak{O}_M$. We write $\mathbf{F}_\bw^{\Gp}$ for the free $\mathbf{A}$-module generated by symbols $\tilde{\delta}_{[M]}$ indexed by orbits $\mathfrak{O}_M$ in $\rep^{\Gp}(\bw,\mathcal{S}^\imath)$. Then we can define a natural map
	\begin{equation}
		\tilde{\chi}:K_\bw(\Gproj(\Lambda^\imath))\otimes_{\mathcal{Z}}\mathbf{A}\longrightarrow \mathbf{F}_\bw^{\Gp},\quad \mathcal{L}(M)\mapsto \sum_{N}p_{N,M}\tilde{\delta}_{[N]}
	\end{equation}
	where $p_{N,M}$ is the (formal) alternating sums of eigenvalues (in $\ov{\Q}_\ell^*$) of Frobenius map on the stalks of the cohomology sheaves of $\mathcal{L}(M)$ at any $\F_q$-rational point of $\mathfrak{O}_N$. To ensure that $\tilde{\chi}$ is an $\mathbf{A}$-linear map, we need to choose $\sqq=-\sqrt{q}$. Note that $p_{M,M}=1$ and $p_{N,M}=0$ if the condition $N\prec M$ is not satisfied.
	
	The restriction map $\res^{\bw}_{\bw_1,\bw_2}$ can be similarly defined on $\mathbf{F}_\bw^{\Gp}$, and we obtain an algebra structure on $\mathbf{F}^{\Gp,*}(Q,\varrho)=\bigoplus_\bw\mathbf{F}_\bw^{\Gp,*}$. Let $\tilde{\xi}_{[M]}$ be the dual basis of $\tilde{\delta}_{[M]}$ and $L_{[M]}$ the dual basis to the image of $\mathcal{L}(M)$ under $\tilde{\chi}$. If we define $H_{[M]}=\sqq^{\dim\mathfrak{O}_M}\tilde{\xi}_{[M]}$, then 
	\[H_{[M]}=\sum_{N}h_{M,N}L_{[N]}\]
	where $h_{M,N}=\sqq^{\dim\mathfrak{O}_M}p_{M,N}$ and $h_{M,N}=0$ if $M\prec N$ is not satisfied. There is a bar-involution on $\mathbf{F}^{\Gp,*}(Q,\varrho)$ defined by Verdier duality, and we have
	\[\ov{H_{[M]}}=\sum_{N}b_{M,N}H_{[N]}.\]
	We can find unique family of elements $\zeta_{M,N}$ such that $\zeta_{M,M}=1$, $\zeta_{M,N}\in\sqq^{-1}\Z[\sqq^{-1}]$ if $N\neq M$, and 
	\[\sum_{M\preceq P\preceq N}b_{M,P}\zeta_{P,N}=\ov{h_{M,N}}.
	\]
	Like Theorem~\ref{thm:IC coefficient is zeta_vv}, one can show that $h_{M,N}=\zeta_{M,N}$ for any $N\prec M$. We also have the following analogue of Corollary~\ref{coro:Lv^0 times Lv^+ prop}.
	
	\begin{proposition}
		For any Gorenstein projective $\mathcal{S}^\imath$-module $M$ and projective $\mathcal{S}^\imath$-module $P$, we have
		\[L_{[P]}\cdot L_{[M]}=\ov{L_{[M]}\cdot L_{[P]}}\in\sqq^{\frac{1}{2}\Z}L_{[P\oplus M]}.\]
	\end{proposition}

	We now consider the following diagram of homomorphisms
	\[\begin{tikzcd}
		\hR^{\imath,\Gp}\otimes_{\mathcal{Z}}\mathbf{A}\ar[r,hook]&\hRi\otimes_{\mathcal{Z}}\mathbf{A}\\
		\mathbf{F}^{\Gp,*}(Q,\varrho)\ar[r,hook]\ar[u,"\tilde{\chi}^*"]&\mathbf{M}^*(Q,\varrho)
	\end{tikzcd}\]
	Let $\widetilde{\mathbf{F}}^{\Gp,*}(Q,\varrho)$ be the localization of $\mathbf{F}^{\Gp,*}(Q,\varrho)$ with respect to $\tilde{\xi}_{[P]}$, $P$ being projective. By the results in \cite[Appendix A.4]{LW19}, this algebra is isomorphic to $\widetilde{\mathbf{M}}^*(Q,\varrho)$, and we get an induced isomorphism 
	\[\tilde{\chi}^*:\widetilde{\mathbf{M}}^*(Q,\varrho)\stackrel{\sim}{\longrightarrow} \tRi\otimes_\mathcal{Z}\mathbf{A}.\]
	
	\begin{proposition}\label{prop:iHA dual trace map}
		The isomorphsim $\tilde{\chi}^*:\widetilde{\mathbf{M}}^*(Q,\varrho)\stackrel{\sim}{\rightarrow} \tRi\otimes_\mathcal{Z}\mathbf{A}$ sends $L_{\bv,\bw}$ to $L(\bv,\bw)$.
	\end{proposition}
	\begin{proof}
		In view of Lemma~\ref{lem: multiply by L^0} and Corollary~\ref{coro:Lv^0 times Lv^+ prop}, it suffices to deal with $L_{\bv^+,\bw^+}$ with $(\bv^+,\bw^+)\in({^\imath V^0},{^\imath W^0})$. For such a pair we choose (unique up to isomorphism) a $\bfk Q$-module $X$ such that $\dimv K_{LR}(X)=(\bv^+,\bw^+)$. Then there is a short exact sequence
		$$0\longrightarrow P_X\longrightarrow G_X\longrightarrow X\longrightarrow0$$
		where $G_X$ is Gorenstein projective and $P_X$ is projective. The desired isomorphism then sends $L_{\bv^+,\bw^+}$ to $v^{d}L_{[P_X]}^{-1}L_{[G_X]}$ in $\mathbf{F}^{\Gp,*}(Q,\varrho)$, for some half integer $d$ so that $v^{d}L_{[P_X]}^{-1}L_{[G_X]}$ is bar-invariant. Then $v^{d}L_{[P_X]}^{-1}L_{[G_X]}$ is sent to $v^dL(P_X)^{-1}L(G_X)$ in $\tR^{\imath,\Gp}$, which is still bar-invariant by the sheaf function correspondence. Its image in $\tRi\otimes_{\mathcal{Z}}\mathbf{A}$ is then equal to $v^{-d'}L(\bv^+,\bw^+)$. Since $v^{-d'}L(\bv^+,\bw^+)$ is bar-invariant, we must have $d'=0$, whence the claim.
	\end{proof}
	
	\begin{remark}
		Consider the composition $\tilde{\chi}^*\circ\widetilde{\Upsilon}:\widetilde{\ch}(Q,\varrho)\otimes_\mathcal{Z}\mathbf{A}\rightarrow\tRi\otimes_\mathcal{Z}\mathbf{A}$. By definition we have
		\begin{alignat*}{2}
			\widetilde{\Upsilon}([S_i])&=\sqq^{\frac{1}{2}}\xi_{(0,\e_{\sigma\tS_i})},\quad\quad &&
			\widetilde{\Upsilon}([\K_i])=\begin{cases}
				\sqq\xi_{(\bv^{\varrho i},\bw^i)} &\text{if $\varrho i=i$},\\
				\sqq^2\xi_{(\bv^{\varrho i},\bw^i)} &\text{if $\varrho i\neq i$},
			\end{cases}\\
			\tilde{\chi}^*(\xi_{(0,\e_{\sigma\tS_i})})&=L(0,\e_{\sigma\tS_i}),\quad\quad &&
			\tilde{\chi}^*(\xi_{(\bv^i,\bw^i)})=\begin{cases}
				\sqq^{-1}L(\bv^i,\bw^i) &\text{if $\varrho i=i$},\\
				\sqq^{-2}L(\bv^i,\bw^i) &\text{if $\varrho i\neq i$},
			\end{cases}
		\end{alignat*}
		so $\tilde{\chi}^*\circ\widetilde{\Upsilon}$ acts on generators by (compare this with the definition of $\widetilde{\Omega}$)
		\[[S_i]\mapsto \sqq^{\frac{1}{2}}L(0,\e_{\sigma\tS_i}),\quad [\K_i]\mapsto L(\bv^{\varrho i},\bw^i).\]
	\end{remark}

	\section{Fourier transforms of dual canonical basis}\label{sec:FT of iHA}
	
	In this section, all Hall algebras and $\imath$Hall algebras are defined over the field of complex numbers $\C$. We shall recall the Fourier transforms of $\imath$Hall algebras constructed in \cite{LP25}. As in \cite{SV99}, we consider $\imath$Hall algebras via functions constructed in \S\ref{subsec:QV function iHA}.

	\subsection{Fourier transforms of dual canonical bases}\label{subsec:FT of dCB}
	
	We return to the set up of \S\ref{subsec:QV function iHA}. For a dimension vector $\bw$, we denote by $E_{Q,\bw}$ the variety of representations of $Q$ of dimension $\bw$. There is a natural inclusion
	\[i:E_{Q,\bw}\longrightarrow\rep(\bw,\mathcal{S}^\imath).\]
	An element of $\rep(\bw,\mathcal{S}^\imath)$ has the form $(x_h,\eps_i,\eps_{\varrho i})_{h\in Q_1,i\in Q_0}$, where 
	\[x_h:W_{\ts(h)}\longrightarrow W_{\tt(h)},\quad \eps_i:W_i\longrightarrow W_{\varrho i},\quad \eps_{\varrho i}:W_{\varrho i}\longrightarrow W_i.\]
	We now consider a $\G_m$-action on $\rep(\bw,\mathcal{S}^\imath)$ defined by 
	\[t\diamond(x_h,\eps_i,\eps_{\varrho i}):=(x_h,t\eps_i,t\eps_{\varrho i}).\]
	This action clearly commutes with the action of $G_\bw$ on $\rep(\bw,\mathcal{S}^\imath)$ and its fixed points are exactly $E_{Q,\bw}$. Moreover, each stratum $\mathcal{M}_0^{\reg}(\bv,\bw,\mathcal{R}^\imath)$ is a $\G_m$-stable subset. In particular, each IC sheaf $\mathcal{L}(\bv,\bw)$ is $\G_m$-equivariant.
	
	Let us denote by $\mathcal{K}_\bw$ the category of direct sums of shifts of $\mathcal{L}(\bv,\bw)$, for $(\bv,\bw)$ being strongly $l$-dominant. Recall that Lusztig has defined a semisimple category $\mathcal{Q}_\bw$ over $E_{Q,\bw}$ consisting of direct sums of Lusztig's perverse sheaves \cite{Lus90}.
	
	\begin{lemma}\label{lem:pullback to Q_w}
		There is an induced functor $i^*:\mathcal{K}_\bw\rightarrow\mathcal{Q}_\bw$.
	\end{lemma}
	
	\begin{proof}
		It is well-known that $\mathcal{Q}_\bw$ is generated by IC sheaves associated to the trivial local system on the $G_\bw$-orbits of $E_{Q,\bw}$. If we denote by $p:\rep(\bw,\mathcal{S}^\imath)\rightarrow E_{Q,\bw}$ the map defined by taking limit, then \cite[Theorem 2.10.3]{Ach21} implies that
		\[i^*\mathcal{L}(\bv,\bw)\cong p_*\mathcal{L}(\bv,\bw).\]
		In particular, $i^*\mathcal{L}(\bv,\bw)$ is a pure perverse sheaf on $E_{Q,\bw}$, hence semisimple \cite[5.4.5]{BBD}. Since it is also $G_\bw$-equivariant, it then lies in the category $\mathcal{Q}_\bw$.
	\end{proof}

	By using the same twisted coproducts, we can construct the dual Grothendieck ring for $E_{Q,\mathbf{w}}$, which will be isomorphic to $\hR^+$. For each $(\bv^+,\bw^+)\in({^\imath V^+},{^\imath W^+})$, we denote by $\mathcal{P}(\bv,\bw)$ the restriction of $\mathcal{L}(\bv,\bw)$, which is also the IC sheaf associated to the $G_\bw$-orbit $\mathcal{M}_0^{\reg}(\bv,\bw,\mathcal{R}^\imath)\subseteq E_{Q,\bw}$. Let $B(\mathbf{v},\mathbf{w})$ denote the dual basis of $\mathcal{P}(\mathbf{v},\mathbf{w})$. For any strongly $l$-dominant pair $(\bv,\bw)$ for $\mathcal{R}^\imath$, we can write
	\begin{equation}\label{eq:i^* of Lvw}
		i^*\mathcal{L}(\bv,\bw)=\sum_{\substack{(\bv',\bw)\in({^\imath V^+},{^\imath W^+})\\ \bv'\leq\bv}}a_{\bv',\bv}\mathcal{P}(\bv',\bw)
	\end{equation}
	where $a_{\bv',\bv}\in\N[v,v^{-1}]$.
	
	\begin{lemma}\label{lem:i^* induced map triangular}
		We have $a_{\bv,\bv}=1$ and $a_{\bv,\bv'}\in v^{-1}\N[v^{-1}]$ if $\bv\neq\bv'$.
	\end{lemma}
	\begin{proof}
		Let $(\bv,\bw)\in ({^\imath V^+},{^\imath W^+})$ be a strongly $l$-dominant pair for $\mathcal{R}^\imath$. From the definition of $\mathcal{P}(\bv,\bw)$ we know that $a_{\bv,\bv}=1$. Now assume that $\bv\neq\bv'$ and let $x_{\bv,\bw}$ be a generic point in $\mathcal{M}_0^{\reg}(\bv,\bw,\mathcal{R}^\imath)$ (viewed as a subset of $E_{Q,\bw}$). Then the decomposition \eqref{eq:i^* of Lvw} for $i^*\mathcal{L}(\bv',\bw)$ tells us
		\begin{align*}
			\mathcal{L}(\bv',\bw)_{x_{\bv,\bw}}&=\sum_{\substack{(\bv'',\bw)\in({^\imath V^+},{^\imath W^+})\\ \bv\leq\bv''\leq \bv'}}a_{\bv'',\bv'}\mathcal{P}(\bv'',\bw)_{x_{\bv,\bw}}\\
			&=a_{\bv,\bv'}\mathcal{P}(\bv,\bw)_{x_{\bv,\bw}}+\sum_{\substack{(\bv'',\bw)\in({^\imath V^+},{^\imath W^+})\\ \bv<\bv''\leq\bv'}}a_{\bv'',\bv'}\mathcal{P}(\bv'',\bw)_{x_{\bv,\bw}}.
		\end{align*}
		By \cite[2.1.9]{BBD} the complex $\mathcal{L}(\bv',\bw)_{x_{\bv,\bw}}$ is concentrated at degrees $<-\dim\mathcal{M}_0^{\reg}(\bv,\bw,\mathcal{R}^\imath)$, while $\mathcal{P}(\bv,\bw)_{x_{\bv,\bw}}$ is concentrated at degree $-\dim\mathcal{M}_0^{\reg}(\bv,\bw,\mathcal{R}^\imath)$. Since each $a_{\bv'',\bv'}$ have positive coefficients, this implies $a_{\bv,\bv'}\in v^{-1}\N[v^{-1}]$.
	\end{proof}

	By applying the isomorphisms in Lemma~\ref{lem:iHA to bfM} and Proposition~\ref{prop:iHA dual trace map} to $\widetilde{\mathcal{H}}(Q)$ and $\widetilde{\mathcal{H}}(Q,\varrho)$ (see Remark~\ref{rmk:F^*(Q)}, the generic version of $\mathcal{F}^*(Q)$, denoted by $\mathbf{F}^*(Q)$, is defined in the same way), we obtain a commutative diagram
	\[\begin{tikzcd}
		\hR^+\otimes_{\mathcal{Z}}\mathbf{A}\ar[r,hook,"\iota"]&\tRi\otimes_{\mathcal{Z}}\mathbf{A}\\
		\mathbf{F}^*(Q)\ar[r,hook,"\iota"]\ar[u,"\tilde{\chi}^*"]&\widetilde{\mathbf{M}}^*(Q,\varrho)\ar[u,swap,"\tilde{\chi}^*"]\\
		\widetilde{\mathcal{H}}(Q)\otimes_{\mathcal{Z}}\mathbf{A}\ar[r,hook,"\iota"]\ar[u,"\Upsilon^+"]\ar[uu,bend left=60pt,"\Omega^+"]&\widetilde{\mathcal{H}}(Q,\varrho)\otimes_{\mathcal{Z}}\mathbf{A}\ar[u,swap,"\widetilde{\Upsilon}"]\ar[uu,bend right=60pt,swap,"\widetilde{\Omega}"]
	\end{tikzcd}\]
	Here the embedding in the second row is induced by $i^*$, and that in the third row is the natural inclusion
	\begin{equation}\label{eq:inclusion of HA to iHA}
		\iota:\widetilde{\ch}(Q)\hookrightarrow \widehat{\ch}(Q,\varrho),\quad \fu_\lambda\mapsto\fu_\lambda,\forall\lambda\in\mathfrak{P}.
	\end{equation}
	From the definition of $\tilde{\chi}^*$, we see that the embedding in the first row is also induced by $i^*$.
	
	Recall that Verdier duality defines an anti-automorphism $\bar{\cdot}$ on $\tRi$ by $\ov{L(\bv,\bw)}=L(\bv,\bw)$, $\ov{v}=v^{-1}$. Then using Lemma~\ref{lem: multiply by L^0}, we can define an action of the subalgebra generated by $L(\bv^0,\bw^0)$, $(\bv^0,\bw^0)\in ({^\imath V^0},{^\imath W^0})$ on $\tRi$, characterized by
	\begin{equation}\label{eq:diamond for L def}
		L(\bv^0,\bw^0)\diamond L(\bv^+,\bw^+) = L(\bv^0+\bv^+,\bw^0+\bw^+) = \ov{L(\bv^0,\bw^0)\diamond L(\bv^+,\bw^+)}
	\end{equation}
	where $(\bv^+,\bw^+)\in ({^\imath V^+},{^\imath W^+})$. From this we can also see that $\widetilde{\Omega}$ commutes with $\diamond$-action: for any $\fu_\lambda\in\widetilde{\ch}(Q,\varrho)$ we have (notations as in Lemma~\ref{lem:Hall basis map to M_vw}, note that $\widetilde{\Omega}=\tilde{\chi}^*\circ\widetilde{\Upsilon}$)
	\[\widetilde{\Omega}(\K_\alpha\diamond \fu_\lambda)=L(\bv_\alpha,\bw_\alpha)\diamond \widetilde{\Omega}(\fu_\lambda).\]
	
	Let us recall the Fourier transforms of $\imath$Hall algebras established in \cite{LP25}.
	\begin{theorem}[{\cite[Theorem {\bf B}]{LP25}}]
		\label{thm:FThall}
		Let $(Q,\btau)$ be a Dynkin $\imath$quiver, and $(Q',\btau)$ be the Dynkin $\imath$quiver constructed from $(Q,\btau)$ by reversing the arrows in a subset $\ce$ of $Q_1$ satisfying $\btau(\ce)=\ce$. Then the Fourier transform $\Phi_{Q',Q}:\widehat{\ch}(\bfk Q,\btau)\rightarrow \widehat{\ch}(\bfk Q',\btau)$ is an isomorphism of $\C$-algebras with the inverse given by $\ov{\Phi}_{Q,Q'}:\widehat{\ch}(\bfk Q',\btau)\rightarrow \widehat{\ch}(\bfk Q,\btau)$.
	\end{theorem}
	
	The Fourier transform $\Phi_{Q',Q}:\widehat{\ch}(\bfk Q,\btau)\rightarrow \widehat{\ch}(\bfk Q',\btau)$ induces an algebra isomorphism  
	$$\Phi:=\Phi_{Q',Q}:\widetilde{\ch}(\bfk Q,\btau)\longrightarrow \widetilde{\ch}(\bfk Q',\btau),$$
	which is called the Fourier transform of $\imath$Hall algebras.

	Let $\Phi_{Q',Q}:\widetilde{\ch}(Q)\rightarrow \widetilde{\ch}(Q')$ be the Fourier transform constructed in \cite{SV99}. From the construction of the Fourier transform $\Phi:\widetilde{\ch}(\bfk Q,\btau)\longrightarrow \widetilde{\ch}(\bfk Q',\btau)$ in \cite[\S5]{LP25}, we have the following commutative diagram
	\[\begin{tikzcd}
		\widetilde{\ch}(Q')\otimes_{\mathcal{Z}}\mathbf{A}\ar[r,hook]&\widetilde{\ch}(Q',\varrho)\otimes_{\mathcal{Z}}\mathbf{A}\\
		\widetilde{\ch}(Q)\otimes_{\mathcal{Z}}\mathbf{A}\ar[r,hook]\ar[u,"\Phi_{Q',Q}"]&\widetilde{\ch}(Q,\varrho)\otimes_{\mathcal{Z}}\mathbf{A}\ar[u,swap,"\Phi"]
	\end{tikzcd}\]

	Now consider a different orientation $Q'$ as in Theorem \ref{thm:FThall}, and let $'\hR^+$, $'\tRi$ be the corresponding dual Grothendieck rings for $(Q',\varrho)$. Recall that Fourier-Deligne transform on $E_{Q,\bw}$ preserves the category $\mathcal{Q}_\bw$ (cf. \cite{Lus93}, see also \cite{Lau87} for the basic properties of Fourier-Deligne transform). So it defines a Fourier transform $\Psi_{Q',Q}:\hR^+\rightarrow {'\hR^+}$ preserving the dual canonical basis of both sides. Moreover, there is a commutative diagram (cf. \cite[Théorème 1.2.1.2]{Lau87})
	\[\begin{tikzcd}
		\hR^+\otimes_{\mathcal{Z}}\mathbf{A}\ar[r,"\Psi_{Q',Q}"]&{'\hR^+}\otimes_{\mathcal{Z}}\mathbf{A}\\
		\widetilde{\ch}(Q)\otimes_{\mathcal{Z}}\mathbf{A}\ar[r,"\Phi_{Q',Q}"]\ar[u,"\Omega^+"]&\widetilde{\ch}(Q')\otimes_{\mathcal{Z}}\mathbf{A}\ar[u,swap,"\Omega^+"]
	\end{tikzcd}\]
	
	We can therefore extend $\Psi_{Q',Q}$ uniquely to an algebra isomorphism $\Psi=\Psi_{Q',Q}:\tRi\rightarrow {'\tRi}$ such that there is a commutative diagram
	\[\begin{tikzcd}[row sep=4mm,column sep=4mm]
		&{'\hR^+}\otimes_{\mathcal{Z}}\mathbf{A}\ar[rr,hook]&&{'\tRi}\otimes_{\mathcal{Z}}\mathbf{A}\\
		\hR^+\otimes_{\mathcal{Z}}\mathbf{A}\ar[ru,"\Psi_{Q',Q}"]&&\tRi\otimes_{\mathcal{Z}}\mathbf{A}\ar[ru,dashed,"\Psi"]\\
		&\widetilde{\mathcal{H}}(Q')\otimes_{\mathcal{Z}}\mathbf{A}\ar[rr,hook]\ar[uu]&&\widetilde{\mathcal{H}}(Q',\varrho)\otimes_{\mathcal{Z}}\mathbf{A}\ar[uu]\\
		\widetilde{\mathcal{H}}(Q)\otimes_{\mathcal{Z}}\mathbf{A}\ar[rr,hook]\ar[uu]\ar[ru,"\Phi_{Q',Q}"]&&\widetilde{\mathcal{H}}(Q,\varrho)\otimes_{\mathcal{Z}}\mathbf{A}\ar[ru,"\Phi"]
		\arrow[from=2-1,to=2-3,hook,crossing over]
		\arrow[from=4-3,to=2-3,crossing over]
	\end{tikzcd}\]
	
	\begin{theorem}\label{FT of CB}
		The map $\Psi:\tRi\rightarrow {'\tRi}$ preserves dual canonical bases.
	\end{theorem}
	\begin{proof}
		From \eqref{eq:i^* of Lvw} we see that for a strongly $l$-dominant pair $(\bv,\bw)\in({^\imath V^+},{^\imath W^+})$, 
		\[\iota(B(\bv,\bw))=\sum_{\bv\leq\bv'}a_{\bv,\bv'}L(\bv',\bw)\]
		with $a_{\bv,\bv'}\in\N[v,v^{-1}]$. Combined with \eqref{eq:diamond for L def}, this implies for any strongly $l$-dominant pair $(\bv,\bw)$, $L(\mathbf{v},\mathbf{w})$ can be expressed as 
		\begin{equation}\label{FT of CB-1}
			L(\mathbf{v},\mathbf{w})\in L(\mathbf{v}^0,\mathbf{w}^0)\diamond \iota(B(\mathbf{v}^+,\mathbf{w}^+))+\sum_{\substack{\mathbf{w}=\tilde{\mathbf{w}}^0+\tilde{\mathbf{w}}^+\\(\mathbf{v}^0,\mathbf{w}^0)\prec(\tilde{\mathbf{v}}^0,\tilde{\mathbf{w}}^0)}} v^{-1}\Z[v^{-1}]L(\tilde{\mathbf{v}}^0,\tilde{\mathbf{w}}^0) \diamond \iota(B(\tilde{\mathbf{v}}^+,\tilde{\mathbf{w}}^+))
		\end{equation}
		where we have used the decomposition $(\bv,\bw)=(\bv^0,\bw^0)+(\bv^+,\bw^+)$ from Lemma~\ref{lem:DP3}. Using Lusztig's Lemma \cite[Theorem 1.1]{BZ14}, it is easy to see that $L(\mathbf{v},\mathbf{w})$ is the unique bar-invariant element satisfying this property. 
		
		Since $\Psi$ commutes with the bar-involution, $\Psi(L(\mathbf{v},\mathbf{w}))$ is bar-invariant. Also, (\ref{FT of CB-1}) gives
		\[\Psi(L(\mathbf{v},\mathbf{w}))\in L(\mathbf{v}^0,\mathbf{w}^0)\diamond \iota(B'(\mathbf{v},\mathbf{w}))+\sum v^{-1}\Z[v^{-1}]L(\tilde{\mathbf{v}}^0,\tilde{\mathbf{w}}^0)\diamond \iota(B'(\tilde{\mathbf{v}}^+,\tilde{\mathbf{w}}^+)),\]
		where we have set $\Psi_{Q',Q}(B(\mathbf{v},\mathbf{w}))=B'(\mathbf{v},\mathbf{w})$. Since $\Psi$ preserves the dimension vector, it follows from the above characterization that $\Psi(L(\mathbf{v},\mathbf{w}))$ is an element of the dual canonical basis in ${'\tRi}$.
	\end{proof}

	\subsection{Dual canonical bases of $\imath$quantum groups}
	\label{subsec:dCB-iQG}
	
	From Theorem \ref{FT of CB}, we know that the dual canonical basis $\{L(\bv,\bw)\mid \text{$(\bv,\bw)$ is strongly $l$-dominant}\}$ of $\hRi$ (and also $\hRiZ$) does not depend on the orientation of the underlying Dynkin diagram. The following is a direct consequence of Corollary~\ref{coro:Upsilon maps dCB} and Proposition~\ref{prop:iHA dual trace map}.
	
	\begin{proposition}\label{iHA dCB compare}
		For $\alpha\in\Z^\I$ and $\lambda\in\mathfrak{P}$, the dual canonical basis $\K_\alpha\diamond \mathfrak{L}_\lambda$ of $\widetilde{\mathcal{H}}(Q,\varrho)$ is mapped to $L(\bv_\alpha+\bv_\lambda,\bw_\alpha+\bw_\lambda)$ under the isomorphism $\widetilde{\Omega}:\widetilde{\ch}(Q,\varrho)\rightarrow\tRi$, where $(\bv_\lambda,\bw_\lambda)$ is such that $\sigma^*\bw_\lambda-\cc_q\bv_\lambda=\lambda$ and $\widetilde{\Omega}(\K_\alpha)=L(\bv_\alpha,\bw_\alpha)$.
	\end{proposition}
	
	From the observations in \S\ref{subsec:FT of dCB}, especially \eqref{eq:i^* of Lvw}, we can prove the following positivity result of the dual canonical basis of $\widetilde{\ch}(Q,\varrho)$, which is reminiscent to \cite[\S 10]{Lus90}.
	
	\begin{theorem}\label{iHA dCB positivity}
		For $\alpha\in\N^\I$ and $\lambda\in\fp$, in $\widetilde{\ch}(Q,\varrho)$ we have 
		\[\K_\alpha\diamond\mathfrak{U}_\lambda\in \K_\alpha\diamond\mathfrak{L}_\lambda+\sum_{(\alpha,\lambda)\prec(\beta,\mu)}v^{-1}\N[v^{-1}]\K_\beta\diamond \mathfrak{L}_\mu.\] 
	\end{theorem}
	
	\begin{proof}
		It suffices to consider the case $\alpha=0$. By the positivity of Lusztig's canonical basis \cite{Lus90}, we find that (see notations in \S\ref{dCB of HA subsec})
		\[\mathfrak{E}_\lambda^*\in \psi^+(\mathfrak{B}_\lambda^*)+\sum_{\lambda\prec\mu}v^{-1}\N[v^{-1}]\psi^+(\mathfrak{B}_\mu^*).\]
		By \eqref{eq:TE} and Proposition~\ref{dCB of HA is L}, this translates to
		\begin{align*}
			\Omega^+(\mathfrak{U}_\lambda)&\in B(\bv_\lambda,\bw_\lambda)+\sum_{\lambda\prec\mu}v^{-1}\N[v^{-1}]B(\bv_\mu,\bw_\mu).
		\end{align*}
		On the other hand, we have already remarked in Lemma~\ref{lem:i^* induced map triangular} that 
		\[\iota(B(\bv,\bw))\in L(\bv,\bw)+\sum_{\bv<\bv'}v^{-1}\N[v^{-1}]L(\bv',\bw).\]
		The assertion now follows from Proposition~\ref{iHA dCB compare}, by noting that $\iota(\Omega^+(\mathfrak{U}_\lambda))=\widetilde{\Omega}(\mathfrak{U}_\lambda)$.
	\end{proof}
	
	From Theorem \ref{FT of CB} and Proposition \ref{iHA dCB compare}, we know the dual canonical basis $\{\K_\alpha\diamond \mathfrak{L}_\lambda \mid \lambda\in\fp,\alpha\in\Z^\I\}$ of $\widetilde{\ch}(Q,\varrho)$ (and also $\widetilde{\ch}(Q,\varrho)_\cz$) does not depend on the orientation of $Q$. From them, we  can also obtain the similar statement for $\widehat{\ch}(Q,\btau)$ (and $\widehat{\ch}(Q,\btau)_\cz$), so we make the following definition. 
	
	\begin{definition}[Dual canonical basis of $\imath$quantum groups]
		\label{def:dual CB}
		The dual canonical basis for $\widetilde{\ch}(Q,\varrho)_\cz$ is transferred to a basis for $\tUi_\cz$ via the isomorphism in Lemma~\ref{lem:Hall-iQG}, which are called the {\em dual canonical basis} for $\tUi$.
	\end{definition}
	
	The dual canonical basis of $\tUi$ can be equivalently defined by using the dual IC basis of $\tRi$, and in particular it is integral and positive. Similarly, we can define the dual canonical  basis for $\hUi$.

	For any $\beta\in\Phi^+$, we construct the root vector $B_\beta\in \tUi$ by using braid group actions; see \cite{LW21a,LP25}. 
	Given any total order of $\Phi^+$, we have defined a PBW basis $\{B^\ba \K_\alpha\mid \ba\in\N^{|\Phi^+|},\alpha\in\Z^\I\}$ for $\tUi$, where 
	\begin{equation}\label{eq:iQG PBW basis def}
		B^{\ba}:=\prod_{\gamma\in\Phi^+} B_{\gamma}^{a_\gamma},\quad \forall \ba=(a_\gamma)_\gamma\in\N^{|\Phi^+|}.
	\end{equation}

	\begin{proposition}\label{prop:iQG PBW to dCB positivity}
		The transition matrix coefficients of $\{B^\ba \K_\alpha\mid \ba\in\N^{|\Phi^+|},\alpha\in\Z^\I\}$ with respect to the dual canonical basis of $\tUi$ belong to $\N[v^{\frac{1}{2}},v^{-\frac{1}{2}}]$.
	\end{proposition}
	
	\begin{proof}
		Using \cite[Theorem {\bf C}]{LP25}, we know that the element $B_\beta$ belongs to the dual canonical basis of $\tUi$ for each $\beta\in\Phi^+$. 
		Then the result follows since the structure constants of dual canonical basis are integral and positive.
	\end{proof}

	\section{Dual canonical basis of $\tUi_v(\mathfrak{sl}_2)$}\label{sec:dCB irank I}
	
	In this section we consider the $\imath$quiver $Q$ with one vertex $1$, no edge, and $\varrho=\Id$. The regular NKS category $\mcr^\imath$ is
	\[
	\begin{tikzcd}
		\tS_1\ar[r,shift left=2pt,"\alpha"]& \sigma(\tS_1)\ar[l,shift left=2pt,"\beta"]
	\end{tikzcd}
	\]
	subject to $\alpha\beta=0$. The corresponding $\imath$quantum group is $\tUi_v(\mathfrak{sl}_2)$ of split type, which is a commutative algebra $\Q(v^{1/2})[B,\K^{\pm1}]$. We will identify $\tUi(\mathfrak{sl}_2)$ with $\tRi$ in the following, in particular, the generators $B=L(0,\e_{\sigma\tS_1})$ and $\K=L(\bv^1,\bw^1)$. 
	
	By definition, we have
	\[\N^{\mathcal{R}^\imath_0-\mathcal{S}^\imath_0}=\N\e_{\tS_1},\quad \N^{\mathcal{S}^\imath_0}=\N\e_{\sigma \tS_1}.\]
	So we may write $(v,w)$ for $l$-dominant pairs in $\mathcal{R}^\imath$. Then the quiver varieties are
	\begin{align*}
		\mathcal{M}(v,w,\mathcal{R}^\imath)&=\{V\subseteq \C^{w},x:\C^{w}\rightarrow \C^{w}\mid \dim V=v,x(\C^{w})\subseteq V,x(V)=0\},\\
		\mathcal{M}_0(w,\mathcal{R}^\imath)&=\{x:\C^{w}\rightarrow \C^{w}\mid x^2=0\}.
	\end{align*}
	In particular, we see that a pair $(v,w)$ is strongly $l$-dominant if and only if $0\leq v\leq\lfloor\frac{w}{2}\rfloor$, and the stratification of $\mathcal{M}_0(w)$ is
	\begin{equation}\label{eq:irank I stratification}
		\mathcal{M}_0(w)=\bigsqcup_{0\leq v\leq\lfloor\frac{w}{2}\rfloor}\mathcal{M}_0^{\reg}(v,w).
	\end{equation}
	Here the stratum $\mathcal{M}_0^{\reg}(v,w)$ consists of matrices $x:\C^{w}\rightarrow \C^{w}$ of rank $v$.
	
	The decomposition of $\pi(v,w)$ can be worked out explicitly. To this end, recall the definition of a semismall map. Let $f:X\rightarrow Y$ be a proper surjective map and $Y=\bigsqcup_\lambda S_\lambda$ a stratification of $Y$ such that each restriction $f_\lambda:f^{-1}(S_\lambda)\rightarrow S_\lambda$ is a topologically locally trivial fibration. The map $f$ is called semismall if for every $\lambda$ we have 
	\[2\dim(f^{-1}(s_\lambda))+\dim(S_\lambda)\leq\dim(X)\]
	where for given stratum $S_\lambda$, we let $s_\lambda\in S_\lambda$ denote any of its points. A stratum $S_\lambda$ of $Y$ is called relevant if $2\dim(f^{-1}(s_\lambda))+\dim(S_\lambda)=\dim(X)$. The most important result for semismall maps is the following decomposition theorem, proved in \cite{dM02}:
	
	\begin{theorem}\label{semismall map decomposition}
		Let $f:X\rightarrow Y$ be a proper surjective semi-small map with $X$ smooth of pure complex dimension $n$. Let $Y=\bigsqcup_\lambda S_\lambda$ be a stratification of $Y$ as above and for each relevant stratum $S_\lambda$, let $L_S=R^{n-\dim(S_\lambda)}f_*\C_X$ be the corresponding local system. There is a canonical isomorphism
		\[f_*\C_X[n]=\bigoplus_{S}\IC_{\ov{S}}(L_S)\]
		where the sum is over all relevant strata $S$.
	\end{theorem}
	
	\begin{proposition}\label{L_k,n decomposition}
		The map $\pi:\mathcal{M}(v,w)\rightarrow\mathcal{M}_0(w)$ is semismall with respect to the stratification \eqref{eq:irank I stratification}, and we have a decomposition
		\[\pi(v,w)=\bigoplus_{0\leq v'\leq\min\{v,w-v\}}\mathcal{L}(v',w).\]
	\end{proposition}
	\begin{proof}
		The fact that $\pi_{k,n}$ is semismall can be checked directly from the following formulas
		\begin{gather*}
			\dim(\mathcal{M}(v,w,\mathcal{R}^\imath))=2v(w-v),\quad \dim(\mathcal{M}_0^{\reg}(v',w,\mathcal{R}^\imath))=2v'(n-v'),\\
			\dim(\pi^{-1}(x_{v',w}))=(v-v')(w-v-v'),
		\end{gather*}
		where $x_{v',w}$ is a generic point in $\mathcal{M}_0^{\reg}(v',w,\mathcal{R}^\imath)$. One can also see that every stratum in the image of $\pi$ is relevant.
	\end{proof}
	
	Given integers $w_1,w_2$, we now consider the restriction functor $\Delta^w_{w_1,w_2}$, where $w=w_1+w_2$. In this case \eqref{eqn:comultiplication} reads
	\begin{equation}\label{eq:irank I Res formula}
		\Res_{w_1,w_2}^{w}(\pi(v,w))=\bigoplus_{v_1+v_2=v}\pi(v_1,w_1)\boxtimes \pi(v_1,w_2).
	\end{equation}
	From this we deduce the following multiplication rule in $\tRi$:
	\begin{equation}\label{eq:irank I multiplication formula}
		L(v_1,w_1)\cdot L(v_2,w_2)=\sum_{v=v_1+v_2}^{\min\{w_1-v_1+v_2,w_2-v_2+v_1,\lfloor\frac{w_1+w_2}{2}\rfloor\}}L(v,w).
	\end{equation}
	In particular, the generators $B=L(0,\e_{\sigma\tS_1})$ and $\K=L(\bv^1,\bw^1)$ satisfy the following equation:
	\[B^a\K^b=\sum_{i=0}^{\lfloor\frac{a}{2}\rfloor}C_{i,a-1}L(i+b,a+2b),\]
	where $C_{i,a-1}=\binom{a-1}{i}-\binom{a-1}{i-2}$. To find an explicit expression of $L(v,w)$, let us write
	\begin{align*}
		L(k,2n)&=\sum_{i=0}^{n-k}X_{i,n-k}B^{2i}\K^{n-i},\\
		L(k,2n+1)&=\sum_{i=0}^{n-k}Y_{i,n-k}B^{2i+1}\K^{n-i},
	\end{align*}
	where the coefficients $X_{k,n}$ and $Y_{k,n}$ are defined inductively as follows
	\begin{align}
		\label{equation for X_kn}
		\begin{aligned}
			X_{n,n+1}&=-C_{1,2n+1},\quad X_{n+1,n+1}=1,\\
			X_{k,n+1}&=-C_{1,2n+1}X_{k,n}-\sum_{i=2}^{n+1-k}C_{i,2n+1}X_{k,n-i+1},\quad 0\leq k\leq n-2,
		\end{aligned}\\
		\label{equation for Y_kn}
		\begin{aligned}
			Y_{n,n+1}&=-C_{1,2n+2},\quad Y_{n+1,n+1}=1,\\
			Y_{k,n+1}&=-C_{1,2n+2}Y_{k,n}-\sum_{i=2}^{n+1-k}C_{i,2n+2}Y_{k,n-i+1},\quad 0\leq k\leq n-2.
		\end{aligned}
	\end{align}
	
	\begin{proposition}\label{I_k,n coefficient formula}
		The above coefficients are given by
		\[X_{k,n}=(-1)^{n-k}\binom{n+k}{2k},\quad Y_{k,n}=(-1)^{n-k}\binom{n+k+1}{2k+1},\]
		and we can therefore write 
		\[L(k,n)=\sum_{j=k}^{\lfloor\frac{n}{2}\rfloor}(-1)^{j-k}\binom{n-k-j}{n-2j}B^{n-2j}\K^j.\]
	\end{proposition}
	\begin{proof}
		We only provide the proof for $X_{k,n}$, that for $Y_{k,n}$ is similar. We fix $k\geq 0$ and consider the sequence $\{X_{k,n}\}_{n\geq k}$. It then suffices to prove the following binomial identity:
		\[\sum_{i=0}^{n}(-1)^i\binom{2k+i}{i}\binom{2n+2k-1}{n-i}=\sum_{i=0}^{n}(-1)^i\binom{2k+i}{i}\binom{2n+2k-1}{n-i-2}\]
		where $k,n\geq 1$. This follows from a direct computation
		\begin{align*}
			\sum_{i=0}^{n}(-1)^i\binom{2k+i}{i}\binom{2n+2k-1}{n-i}=\sum_{i=0}^{n}\binom{-2k-1}{i}\binom{2n+2k-1}{n-i}=\binom{2n-2}{n},
		\end{align*}
		\begin{align*}
			\sum_{i=0}^{n}(-1)^i\binom{2k+i}{i}\binom{2n+2k-1}{n-i-2}&=\sum_{i=0}^{n-2}(-1)^i\binom{2k+i}{i}\binom{2n+2k-1}{n-i-2}\\
			&=\sum_{i=0}^{n-2}\binom{-2k-1}{i}\binom{2n+2k-1}{n-i-2}=\binom{2n-2}{n-2},
		\end{align*}
		where we have used the following standard identities ($x,y,m$ are integers and $n\geq 0$):
		\begin{gather*}
			\binom{x}{n}=(-1)^n \binom{n-x-1}{n},\quad \sum_{k=0}^n \binom{x}{n-k} \binom{y}{k} = \binom{x+y}{n}.\qedhere
		\end{gather*}
	\end{proof}

\end{document}